\documentclass[a4paper,10pt,fleqn,leqno]{article}
\usepackage[dvipdfmx]{graphicx,xcolor}
\usepackage[fleqn]{amsmath}
\numberwithin{equation}{section}
\usepackage{latexsym}
\usepackage{amssymb}
\usepackage{amsbsy}
\usepackage{enumerate}
\usepackage[varg]{txfonts}

\title{Paper-Scissors-Stone Model for Interacting 
Population and  its Limit Theorem\footnote {Research Memorandum 485  The Institute of Statitical Mathematics, 16 September1993, 
Yasunori Okabe (Department of Mathematics, Faculty of Science, Hokkaido University), Hajime Mano (Institute of Statistical Mathematics and The Graduate University for Advanced Studies), 
Yoshiaki Itoh (Institute of Statistical Mathematics and The Graduate University for Advanced Studies)}}
\author{Yasunori Okabe, Hajime Mano, Yoshiaki Itoh }
\date{}
\makeatletter
\renewcommand{\section}{%
  \@startsection {section}{1}{\z@}%
    {-3.5ex \@plus -1ex \@minus -.2ex}%
    {2.3ex \@plus.2ex}%
    {\normalfont\Large\bfseries\scshape}%
}
\makeatother

\def\dint{\displaystyle\int}
\newtheorem{theo}{Theorem}[section]

\newtheorem{rem}{\normalfont \textit{Remark}}[section]

\newtheorem{proof}{\normalfont \textit{Proof}.}

\newtheorem{coro}{Corollary}[section]
\newtheorem{lemma}{Lemma}[section]

\def\dsum{\displaystyle\sum}
\def\dint{\displaystyle \int}
\def\dlim{\displaystyle\lim}
\begin{document}
\maketitle
\abstract{This paper treats a random collision model of three species, which is represented by the random time change of three standard Poisson processes. The prey-predator relation in the random collision model looks like paper-scissors-stone 
 game, and the model is called the paper-scissors model. At first, we investigate the stochastic structure of our model. By using stochastic calculus, the model is decomposed into a semi-martingale, and we prove a weak law of large numbers and a central limit theorem. The main purpose of this paper is to obtain an ordinary differential equation from the weak law and a stochastic differential equation from the central limit theorem.}
 
\noindent{\it Keywords}: martingale, optional sampling theorem, standard Poisson process, stopping time, strong law of large numbers, weak law of large numbers
 
\section{Introduction}
\label{sec1}

Problems of interspecific competitions have been studied by many authors since 
Lotka \cite{10} and Volterra \cite{14}, who studied interacting populations as 
a deterministic system. The larger populations are implicitly assumed for the 
deterministic system. For smaller populations the random sampling effect 
should be taken into account. Ehrenfest's urn model was mathematically analyzed by 
Kac \cite{7}. Moran \cite{11} studied an urn model for the random genetic 
drift introduced by Fisher \cite{2} and Wright \cite{15}. 
Itoh \cite{4,5,6} introduced a random collision model which is an urn model for 
competing species in finite numbers of individuals of several types interacting 
with each other and studied the probability of coexistence of species by use of 
oriented graphs.

We discuss the random collision model (\cite{6}) which satisfies the following:
\begin{enumerate}[(i)]
  \item 
  There are three species 1, 2 and 3 whose numbers of particles at time $t$ are 
  $X_1^{(M)}(t)$, $X_2^{(M)}(t)$ and $X_3^{(M)}(t)$ respectively, 
  where $X_1^{(M)}(t)+X_2^{(M)}(t)+X_3^{(M)}(t)=M$. 
  We denote $X^{(M)}(t)=(X_1^{(M)}(t), X_2^{(M)}(t), X_3^{(M)}(t))$.
  \item 
  Each particle collides with another particle $dt$ times on the average per time 
  length $dt$.
  \item 
  Each particle is in a chaotic bath of particles. 
  Each colliding pair is equally likely chosen.
  \item 
  Collisions between particles of the same species do not make any change. 
  A particle of species $i$ and a particle of species $i+1$ collide with each other 
  and become two particles of species $i$, where $i=1,2,3$ 
  and if $i=3$ then we set $i+1=1$ and if $i=1$ then we set $i-1=3$ from now on.
\end{enumerate}
A model written by the random time change of three standard Poisson processes 
is given by Itoh \cite{6}:
{\mathindent=0mm
\begin{align*}
  \left\{\renewcommand{\arraystretch}{2}
    \begin{array}{@{}l@{}}
    X_1^{(M)}(t)=X_1^{(M)}(0)+N_{12}
    \left(\dfrac{\lambda}{M} \dint_0^t X_1^{(M)}(s) X_2^{(M)}(s) d s\right)
    -N_{31}\left(\dfrac{\lambda}{M} \dint_0^t X_3^{(M)}(s) X_1^{(M)}(s) d s\right),
    \\
    X_2^{(M)}(t)=X_2^{(M)}(0)+N_{23}
    \left(\dfrac{\lambda}{M} \dint_0^t X_2^{(M)}(s) X_3^{(M)}(s) d s\right)
    -N_{12}\left(\dfrac{\lambda}{M} \dint_0^t X_1^{(M)}(s) X_2^{(M)}(s) d s\right), 
    \\
    X_3^{(M)}(t)=X_3^{(M)}(0)+N_{31}
    \left(\dfrac{\lambda}{M} \dint_0^t X_3^{(M)}(s) X_1^{(M)}(s) d s\right)
    -N_{23}\left(\dfrac{\lambda}{M} \dint_0^t X_2^{(M)}(s) X_3^{(M)}(s) d s\right), 
    \\
    X_1^{(M)}(0)+X_2^{(M)}(0)+X_3^{(M)}(0)=M,
    \end{array}\right.
\end{align*}
}%
where $X_i^{(M)}(0)$ are initial values $(i=1,2,3)$. 
We call this model paper-scissors-stone model because of the cyclic 
prey-predator relation, as in paper-scissors-stone game.

We discuss this model in this paper. A random collision model of two species represented by the random time change of one Poisson process is analyzed to obtain
a strong law of large numbers in \cite{12}. 
We develop a stochastic analysis for the following queuing model to the 
paper-scissors-stone model.

Kogan, Liptser, Shiryayev and Smorodinski \cite{8,9} treated a queuing model of 
computer networks. The queuing model, discussed there, is constructed by mutually 
independent queues. They successfully analyzed their queuing model by the martingale 
method. 
They proved a weak law of large numbers and a central limit theorem for a certain 
queuing model by using stochastic calculus and obtained a system of ordinary 
differential equations by a weak law of large numbers and a system of stochastic 
differential equations of the Gaussian diffusion process by a central limit theorem.

This paper treats a random collision model of three species represented by 
the random time change of Poisson processes. Three cyclic prey-predator relations in 
the model complicate the situation. Motivated by the martingale method, 
we analyze the paper-scissors-stone model and investigate limit theorems in detail. 
In the queuing model and our model each component of the stochastic process is 
decomposed into a counting process of the number arriving over time $t$ and 
a counting process of the number serviced by time $t$. In our model the number 
increasing over time $t$ of $i$-th component is equal to the number decreasing by 
time $t$ of $i+1$-th component. 
Differently from the queuing model, our model has a stochastic structure that 
martingales are not orthogonal and that bounded variations are continuous. 
We obtain a system of ordinary differential equations by a weak law of large numbers 
and a system of stochastic differential equations by a central limit theorem.

In the present paper we mainly aim for the paper-scissors-stone model to obtain 
an ordinary differential equation from a weak law of large numbers and a stochastic 
differential equation of the Gaussian diffusion process from a central limit theorem. 
We solve the paper-scissors-stone model explicitly in section~\ref{sec2}. 
We find a reference family which represents for our problem to apply the optional 
sampling theorem to get a stochastic structure of our model in section~\ref{sec3}. 
Martingales in different components are not orthogonal. In section~\ref{sec4} and 
section~\ref{sec6}, we briefly mention about the extension of the weak law of large 
numbers and of the central limit theorem for the paper-scissors-stone model. 
For the paper-scissors-stone model we obtain a system of ordinary differential 
equations from a weak law of large numbers in section 5 and a system of stochastic 
differential equations of the Gaussian diffusion process from 
a central limit theorem in section~\ref{sec7}.

\section{Paper-scissors-stone model and solution}
\label{sec2}

Let us consider a population of three species in which individuals randomly interact 
with each other. Changes occur by interactions only between two particles of different 
species. 
If two individuals annihilate by the interaction, then two individuals of the dominant 
species are created. Thus the total number of the particles is invariant under 
interactions.

We set any positive integer $M$ which denotes the total number of the particles of 
a system. For each $j$, $j=1,2,3$, let $X_j^{(M)}(*)$ be the stochastic process 
which denotes the number of individuals of species $j$. 
We assume that $X_j^{(M)}(*)$ is dominant and $X_{j+1}^{(M)}(*)$ 
is recessive between species $j$ and species $j+1$ ($j=1,2,3$ and 
if $j=3$ then we set $j+1=1$ and if $j=1$ 
then we set $j-1=3$ from now on). Moreover it is
assumed that each stochastic process is represented by the time change of standard 
Poisson processes $N_{\diamond}(*)$ in a differential form as
\begin{equation}
    \renewcommand{\arraystretch}{2}
  \left\{
    \begin{array}{@{}l@{}}
    dX_1^{(M)}(t)=dN_{12}
    \left(\dfrac{\lambda}{M} \dint_0^t X_1^{(M)}(s) X_2^{(M)}(s)ds\right)
    -d N_{31}\left(\dfrac{\lambda}{M} 
    \dint_0^t X_3^{(M)}(s) X_1^{(M)}(s) d s\right) ,
    \\
    d X_2^{(M)}(t)=d N_{23}\left(\dfrac{\lambda}{M} 
    \dint_0^t X_2^{(M)}(s) X_3^{(M)}(s) d s\right)-d N_{12}
    \left(\dfrac{\lambda}{M} \dint_0^t X_1^{(M)}(s) X_2^{(M)}(s)ds\right), 
    \\
    d X_3^{(M)}(t)=d N_{31}\left(\dfrac{\lambda}{M} \dint_0^t X_3^{(M)}(s) 
    X_1^{(M)}(s) d s\right)-d N_{23}\left(\dfrac{\lambda}{M} 
    \dint_0^t X_2^{(M)}(s) X_3^{(M)}(s) d s\right),
  \end{array}\right. \label{eq2.1}
\end{equation}
where $\lambda$ is a positive constant. 
This is also written in the integral form as
\begin{equation}
  \left\{
    \begin{aligned}
      &
      X_1^{(M)}(t)=X_1^{(M)}(0) 
      +N_{12}\left(\frac{\lambda}{M} \int_0^t X_1^{(M)}(s) X_2^{(M)}(s) 
      d s\right) \\
      &\qquad\qquad
       -N_{31}\left(\frac{\lambda}{M} 
       \int_0^t X_3^{(M)}(s) X_1^{(M)}(s) d s\right), \\
      &
      X_2^{(M)}(t)=X_2^{(M)}(0) +N_{23}\left(\frac{\lambda}{M} 
      \int_0^t X_2^{(M)}(s) X_3^{(M)}(s) d s\right) \\
      &\qquad\qquad 
      -N_{12}\left(\frac{\lambda}{M} 
      \int_0^t X_1^{(M)}(s) X_2^{(M)}(s) d s\right), \\
      &
      X_3^{(M)}(t)=X_3^{(M)}(0) +N_{31}\left(\frac{\lambda}{M} 
      \int_0^t X_3^{(M)}(s) X_1^{(M)}(s) d s\right) \\
      &\qquad\qquad
      -N_{23}\left(\frac{\lambda}{M} 
      \int_0^t X_2^{(M)}(s) X_3^{(M)}(s) d s\right), \\
      &
      X_1^{(M)}(0)+X_2^{(M)}(0) +X_3^{(M)}(0)=M,
\end{aligned}\right.
  \label{eq2.2}
\end{equation}
where $X_j^{(M)}(0)$ are initial values of $X_j^{(M)}(*)$ $(j=1,2,3)$.
\begin{rem}
  \label{rem2.1}\normalfont 
  The case of the $n$-species is treated in a similar way as the paper-scissors-stone model.
\end{rem}

\begin{theo}
  \label{theo2.1}
  There exists a unique solution of equation~{\rm (\ref{eq2.2})}.
\end{theo}

\begin{proof}\normalfont 
  We fix a sample path. We denote $\{\tau_i^{j j+1}\}_{i \geq 0}$ 
  as the set of the jump times of three standard Poisson process 
  $N_{j j+1}(*)$ where we put $\tau_0^{j j+1}=0$ $(j=1,2,3)$. 
  Note that $0=\tau_0^{j j+1} < \tau_1^{j j+1} < \tau_2^{j j+1}
  < \cdots <\tau_i^{j j+1} < \tau_{i+1}^{j j+1} < \cdots$ 
  for $j=1,2,3$.

We construct a solution of equation~(\ref{eq2.2}) actually. 
This construction is done step by step. 
From $t=0$ we trace the time when the system of (\ref{eq2.2}) 
has a change of the previous state. The change of the system occurs by 
the jumps of some of the Poisson processes.

We denote $\sigma(l)$ as the $l$-th jump time of the system at which 
the system has a change and define $\sigma(0)=0$. 
For $j=1,2,3$ we denote $K^{j j+1}(l)$ as the total number of 
the jumps of the Poisson process $N_{j j+1}(*)$ 
to the extent of the $l$-th jump time $\sigma(l)$, 
that we call the $l$-th step, and we define $K^{j j+1}(0)=0$. 
And we define for each $t \in[0, \sigma(l)]$ $(j=1,2,3)$
\begin{align*}
  T_{j j+1}^{(M)}(t)
  =\frac{\lambda}{M} \int_0^t X_j^{(M)}(s) X_{j+1}^{(M)}(s) ds,
\end{align*}
by the constructed solution to the extent of the $l$-th step. 
For an integer $l$, $l \geq 0$, and for an integer $k$, 
$k \geq 1$, $(1 \leq j \leq 3)$ we define two propositions 
$P_j(l)$ and $P_j(k-1, k)$ as
\begin{align*}
  &P_j(l): T_{j j+1}^{(M)}(\sigma(l)) 
  \in \Big[\tau_{K^{j j+1}(l)}^{j j+1}, \tau_{K^{j j+1}(l)+1}^{j j+1}
  \Big),
  \\
  &
  P_j(k-1, k): T_{j j+1}^{(M)}(t) 
  \in \Big[\tau_{K^{j j+1}(k-1)}^{j j+1}, \tau_{K^{j j+1}(k-1)+1}^{j j+1}) 
  ~~\text { \textit{for} }~ 
  t \in (\sigma(k-1), \sigma(k)\Big). 
\end{align*}

We shall prove existence of the solution of the system by mathematical 
induction on $I$.

We prove $P_j(0)$ for $j=1,2,3$. 
The initial values are given as $X_j^{(M)}(\sigma(0))
=  X_j^{(M)}(0)$.
\begin{align*}
  T_{j j+1}^{(M)}(\sigma(0))
  =\frac{\lambda}{M} \int_0^{\sigma(0)} X_j^{(M)}(s) 
  X_{j+1}^{(M)}(s) d s
  =\frac{\lambda}{M} \int_0^0 X_j^{(M)}(s) X_{j+1}^{(M)}(s) ds=0.
\end{align*}
Thus
\begin{align*}
  \tau_0^{j j+1}=T_{j j+1}^{(M)}(\sigma(0)) 
  = 0 <\tau_1^{j j+1} .
\end{align*}
As we define $K^{j j+1}(0)=0$ for $j=1,2,3$, we have 
\begin{align*}
  \tau_{K^{j j+1}(0)}^{j j+1} 
  = T_{j j+1}^{(M)}(\sigma(0)) = 0<\tau_{K ^{j j+1}(0)+1}^{j j+1}.
\end{align*}
Therefore $P_j(0)$ hold for $j=1,2,3$. 

It follows that
\begin{align*}
  N_{j j+1}\left(T_{j j+1}^{(M)}(\sigma(0))\right)=0 . 
\end{align*}
At $t=0$,
\begin{equation}
  \left\{\renewcommand{\arraystretch}{3}
  \begin{array}{@{}l@{}}
    X_1^{(M)}(\sigma(0))
    =X_1^{(M)}(0)+\dsum_{i=1}^{K^{12}(0)}(+1)
    +\dsum_{i=1}^{K^{31}(0)}(-1), \\
    X_2^{(M)}(\sigma(0))
    =X_2^{(M)}(0)+\dsum_{i=1}^{K^{23}(0)}(+1)
    +\dsum_{i=1}^{K^{12}(0)}(-1), \\
    X_3^{(M)}(\sigma(0))=X_3^{(M)}(0)
    +\dsum_{i=1}^{K^{31}(0)}(+1)+\dsum_{i=1}^{K^{23}(0)}(-1),
  \end{array}\right.
  \label{eq2.3}
\end{equation}
is replaced by
\begin{align*}
  \left\{\renewcommand{\arraystretch}{1.5}
    \begin{array}{@{}l@{}}
      X_1^{(M)}(\sigma(0))
      =X_1^{(M)}(0)+N_{12}\left(T_{j j+1}^{(M)}(\sigma(0))\right)
      -N_{31}\left(T_{j j+1}^{(M)}(\sigma(0))\right) ,
      \\
      X_2^{(M)}(\sigma(0))
      =X_2^{(M)}(0)+N_{23}\left(T_{j j+1}^{(M)}(\sigma(0))\right)
      -N_{12}\left(T_{j j+1}^{(M)}(\sigma(0))\right), \\
      X_3^{(M)}(\sigma(0))
      =X_3^{(M)}(0)+N_{31}\left(T_{j j+1}^{(M)}(\sigma(0))\right)
      -N_{23}\left(T_{j j+1}^{(M)}(\sigma(0))\right).
  \end{array}\right.
\end{align*}
Consequently there exists a solution which has a form of (\ref{eq2.3}) 
at $\sigma(0)$.

We assume the solution in $[0, \sigma(I-1)]$ $(I \geq 1)$ with 
propositions. Note that for the mathematical induction we assume the 
solution, at $t=\sigma(I-1)$,
\begin{equation}
  \left\{\renewcommand{\arraystretch}{3}
    \begin{array}{@{}l@{}}
      X_1^{(M)}(\sigma(I-1))=X_1^{(M)}(0)
      +\dsum_{i=1}^{K^{12}(I-1)}(+1)+\dsum_{i=1}^{K^{31}(I-1)}(-1), 
      \\
      X_2^{(M)}(\sigma(I-1))=X_2^{(M)}(0)
      +\dsum_{i=1}^{K^{23}(I-1)}(+1)+\dsum_{i=1}^{K^{12}(I-1)}(-1), 
      \\
      X_3^{(M)}(\sigma(I-1))=X_3^{(M)}(0)
      +\dsum_{i=1}^{K^{31}(I-1)}(+1)+\dsum_{i=1}^{K^{23}(I-1)}(-1),
\end{array}\right.
\label{eq2.4}
\end{equation}
with the propositions $P_j(I-1)$ $(j=1,2,3)$. 
This equation~(\ref{eq2.4}) is obtained from replacing 0 by $I-1$ 
in (\ref{eq2.3}).

For $t \in(\sigma(I-1), \sigma(I))$ 
we construct the solution of the system of (\ref{eq2.2}) as
\begin{equation}
  \left\{\renewcommand{\arraystretch}{2.5}
  \begin{array}{@{}l@{}}
    X_1^{(M)}(t)
    =X_1^{(M)}(0)+\dsum_{i=1}^{K^{12}(I-1)}(+1)
    +\dsum_{i=1}^{K^{31}(I-1)}(-1), \\
    X_2^{(M)}(t)
    =X_2^{(M)}(0)+\dsum_{i=1}^{K^{23}(I-1)}(+1)
    +\dsum_{i=1}^{K^{12}(I-1)}(-1), \\
    X_3^{(M)}(t)
    =X_3^{(M)}(0)+\dsum_{i=1}^{K^{31}(I-1)}(+1)
    +\dsum_{i=1}^{K^{23}(I-1)}(-1),
  \end{array}\right.
  \label{eq2.5}
\end{equation}
and at $t=\sigma(I)$ as
\begin{equation}
  \left\{\renewcommand{\arraystretch}{2.5}
    \begin{array}{@{}l@{}}
      X_1^{(M)}(\sigma(I))=X_1^{(M)}(0)+\dsum_{i=1}^{K^{12}(I)}(+1)
      +\dsum_{i=1}^{K^{31}(I)}(-1), \\
      X_2^{(M)}(\sigma(I))=X_2^{(M)}(0)+\dsum_{i=1}^{K^{23}(I)}(+1)
      +\dsum_{i=1}^{K^{12}(I)}(-1), \\
      X_3^{(M)}(\sigma(I))=X_3^{(M)}(0)+\dsum_{i=1}^{K^{31}(I)}(+1)
      +\dsum_{i=1}^{K^{23}(I)}(-1),
    \end{array}\right.
  \label{eq2.6}
\end{equation}
where $\sigma(I)$ and $K^{j j+1}(I)$ are setted 
in [Case A]$\sim$[Case C]. 

\noindent
[Case A] We consider the case of $X_j^{(M)}(\sigma(l))>0$ for 
$0 \leq l \leq I-1$ and $1 \leq j \leq 3$. 
This case describes that the values of all random variables 
have not reached zero.

We determine $\sigma(I)$ as
\begin{equation}
  \sigma(I)=\min _{1 \leq j \leq 3}
  \left\{
    \sigma(I-1)
    +\dfrac{\tau_{K^{j j+1}(I-1)+1}^{j j+1} 
    - T_{j j+1}^{(M)}(\sigma(I-1))
    }{
      \dfrac{\lambda}{M} X_j^{(M)}(\sigma(I-1)) X_{j+1}^{(M)}
      (\sigma(I-1))}\right\} .
  \label{eq2.7}
\end{equation}
By taking the minimum of $1 \leq j \leq 3$, we count up one in 
$K^{j j+1}(I)$ for the selected number and we do not count up one for 
the not selected number. If by taking the minimum of $1 \leq j \leq 3$ 
the number $j=1$ is selected, for example, then we have 
$K^{12}(I)=K^{12}(I-1)+1$, $K^{23}(I)=K^{23}(I-1)$ and 
$K^{31}(I)=K^{31}(I-1)$. 
If by taking the minimum of $1 \leq j \leq 3$ the numbers $j=1,2$ 
are selected, then we have $K^{12}(I)=K^{12}(I-1)+1$, 
$K^{23}(I)=K^{23}(I-1)+1$ and $K^{31}(I)=K^{31}(I-1)$. 
If by taking the minimum of $1 \leq j \leq 3$ the numbers 
$j=1,2,3$ are selected, then we have $K^{12}(I)=K^{12}(I-1)+1$, 
$K^{23}(I)=K^{23}(I-1)+1$ and $K^{31}(I)=K^{31}(I-1)+1$. 
And a solution of the system of (\ref{eq2.2}) is as 
in (\ref{eq2.5}) and (\ref{eq2.6}).

Now we prove $P_j(I)$ and $P_j(I-1, I)$ for $j=1,2,3$.

If by taking the minimum of $1 \leq j \leq 3$ the number $j=1$ 
is selected for example, we have $K^{12}(I)=K^{12}(I-1)+1$, 
$K^{23}(I)=K^{23}(I-1)$ and $K^{31}(I)=K^{31}(I-1)$.

In the present case the number $j=1$ is selected by taking the minimum 
of (\ref{eq2.7}). This means
\begin{align*}
  & \sigma(I)=\sigma(I-1)
  +\dfrac{\tau_{K^{12}(I-1)+1}^{12}-T_{12}^{(M)}(\sigma(I-1))
  }{\dfrac{\lambda}{M} X_1^{(M)}(\sigma(I-1)) X_2^{(M)}(\sigma(I-1))}, 
  \\
  & \sigma(I)<\sigma(I-1)
  +\dfrac{\tau_{K^{23}(I-1)+1}^{23}-T_{23}^{(M)}(\sigma(I-1))
  }{\dfrac{\lambda}{M} X_2^{(M)}(\sigma(I-1)) X_3^{(M)}(\sigma(I-1))}, 
  \\
  & \sigma(I)<\sigma(I-1)
  +\dfrac{\tau_{K^{31}(I-1)+1}^{31}-T_{31}^{(M)}(\sigma(I-1))
  }{\dfrac{\lambda}{M} X_3^{(M)}(\sigma(I-1)) X_1^{(M)}(\sigma(I-1))}.
\end{align*}
From $P_j(I-1)$ for $1 \leq j \leq 3$ all numerators are positive 
and all random variables $X_j^{(M)}(\sigma(I-1))$ are positive 
in [Case A]. Thus we have $\sigma(I)>\sigma(I-1)$ and
\begin{align*}
  & \tau_{K^{12}(I-1)+1}^{12}
  =T_{12}^{(M)}(\sigma(I-1))+\frac{\lambda}{M} X_1^{(M)}(\sigma(I-1)) 
  X_2^{(M)}(\sigma(I-1))(\sigma(I)-\sigma(I-1)), 
  \\
  & \tau_{K^{23}(I-1)+1}^{23}>T_{23}^{(M)}(\sigma(I-1))
  +\frac{\lambda}{M} X_2^{(M)}(\sigma(I-1)) 
  X_3^{(M)}(\sigma(I-1))(\sigma(I)-\sigma(I-1)), 
  \\
  & \tau_{K^{31}(I-1)+1}^{31}>T_{31}^{(M)}(\sigma(I-1))
  +\frac{\lambda}{M} X_3^{(M)}(\sigma(I-1)) 
  X_1^{(M)}(\sigma(I-1))(\sigma(I)-\sigma(I-1)) .
\end{align*}
[Step 1] We consider the propositions for the selected number $j=1$. 

For $\sigma(I-1)<t<\sigma(I)$ we have
\begin{align*}
  T_{12}^{(M)}(t) 
  & =\int_0^t X_1^{(M)}(s) X_2^{(M)}(s) d s \\
  & =T_{12}^{(M)}(\sigma(I-1))
  +\frac{\lambda}{M} X_1^{(M)}(\sigma(I-1)) X_2^{(M)}(\sigma(I-1))
  (t-\sigma(I-1)),
\end{align*}
and
\begin{align*}
  T_{12}^{(M)}(\sigma(I)) 
  & =\int_0^{\sigma(I)} X_1^{(M)}(s) X_2^{(M)}(s) d s \\
  & =T_{12}^{(M)}(\sigma(I-1))
  +\frac{\lambda}{M} X_1^{(M)}(\sigma(I-1)) 
  X_2^{(M)}(\sigma(I-1))(\sigma(I)-\sigma(I-1)) \\
  & =\tau_{K^{12}(I-1)+1}^{12}.
\end{align*}
The condition of positiveness of random variables 
$X_j^{(M)}(\sigma(I-1))$ in [Case A] leads 
$T_{12}^{(M)}(\sigma(I-1)) < T_{12}^{(M)}(t) 
< T_{12}^{(M)}(\sigma(I))$ for $\sigma(I-1) < t < \sigma(I)$. 
From $P_1(I-1)$ it follows that
\begin{align*}
  \tau_{K^{12}(I-1)}^{12} 
  & \leq T_{12}^{(M)}(\sigma(I-1))<T_{12}^{(M)}(t)
  <T_{12}^{(M)}(\sigma(I))=\tau_{K^{12}(I-1)+1}^{12}, \\
  & \tau_{K^{12}(I-1)+1}^{12}
  =\tau_{K^{12}(I)}^{12}=T_{12}^{(M)}(\sigma(I))
  < \tau_{K^{12}(I)+1}^{12}.
\end{align*}
Therefore $P_1(I-1, I)$ and $P_1(I)$ hold. 
\hfill $\diamondsuit$

\noindent
[Step 2] We consider the propositions for the not selected number 
$j=2$.

For $\sigma(I-1)<t<\sigma(I)$ we have
\begin{align*}
  T_{23}^{(M)}(t)
  =T_{23}^{(M)}(\sigma(I-1))+\frac{\lambda}{M} 
  X_2^{(M)}(\sigma(I-1)) X_3^{(M)}
  (\sigma(I-1))(t-\sigma(I-1)),
\end{align*}
and
\begin{align*}
  T_{23}^{(M)}(\sigma(I)) 
  & =T_{23}^{(M)}(\sigma(I-1))
  +\frac{\lambda}{M} X_2^{(M)}(\sigma(I-1)) 
  X_3^{(M)}(\sigma(I-1))(\sigma(I)-\sigma(I-1)) \\
  & <\tau_{K^{23}(I-1)+1}^{23}.
\end{align*}
From $P_2(I-1)$ it follows that
\begin{align*}
&
  \tau_{K^{23}(I-1)}^{23} \leq T_{23}^{(M)}(\sigma(I-1))
  <T_{23}^{(M)}(t)<T_{23}^{(M)}(\sigma(I))<\tau_{K^{23}(I-1)+1}^{23}, 
  \\
  &\quad
  \tau_{K^{23}(I-1)}^{23}
  =\tau_{K^{23}(I)}^{23}<T_{23}^{(M)}(\sigma(I))
  <\tau_{K^{23}(I-1)+1}^{23}=\tau_{K^{23}(I)+1}^{23}.
\end{align*}
Therefore $P_2(I-1, I)$ and $P_2(I)$ hold. \hfill $\diamondsuit$

For the not selected number $j=3$, $P_3(I-1, I)$ 
and $P_3(I)$ also hold.

If by taking the minimum of $1 \leq j \leq 3$ the numbers $j=1,2$ 
are selected, we have $K^{12}(I)=K^{12}(I-1)+1$, 
$K^{23}(I)=K^{23}(I-1)+1$ and $K^{31}(I)=K^{31}(I-1)$. 
The selection of the numbers $j=1,2$ means
\begin{align*}
& 
  \sigma(I)=\sigma(I-1)
  +\dfrac{\tau_{K^{12}(I-1)+1}^{12}-T_{12}^{(M)}(\sigma(I-1))
  }{\dfrac{\lambda}{M} X_1^{(M)}(\sigma(I-1)) X_2^{(M)}(\sigma(I-1))}, 
  \\
& 
  \sigma(I)=\sigma(I-1)
  +\dfrac{\tau_{K^{23}(I-1)+1}^{23}-T_{23}^{(M)}(\sigma(I-1))
  }{\dfrac{\lambda}{M} X_2^{(M)}(\sigma(I-1)) X_3^{(M)}(\sigma(I-1))}, 
  \\
& 
  \sigma(I)<\sigma(I-1)
  +\dfrac{\tau_{K^{31}(I-1)+1}^{31}-T_{31}^{(M)}(\sigma(I-1))
  }{\dfrac{\lambda}{M} X_3^{(M)}(\sigma(I-1)) X_1^{(M)}(\sigma(I-1))}.
\end{align*}
For the selected number $j=1,2$ we have the propositions $P_j(I-1, I)$ 
and $P_j(I)$ similarly as in [Step~1]. 
For the not selected number $j=3$ the propositions $P_3(I-1, I)$ 
and $P_3(I)$ hold in a similar way as [Step~2].

If by taking the minimum of $1 \leq j \leq 3$ the numbers $j=1,2,3$ are selected, 
we have $K^{j j+1}(I)=K^{j j+1}(I-1)+1$. Then
\begin{align*}
& 
  \sigma(I)=\sigma(I-1)
  +\dfrac{\tau_{K^{12}(I-1)+1}^{12}-T_{12}^{(M)}(\sigma(I-1))
  }{\dfrac{\lambda}{M} X_1^{(M)}(\sigma(I-1)) X_2^{(M)}(\sigma(I-1))}, 
  \\
& 
  \sigma(I)=\sigma(I-1)
  +\dfrac{\tau_{K^{23}(I-1)+1}^{23}-T_{23}^{(M)}(\sigma(I-1))
  }{\dfrac{\lambda}{M} X_2^{(M)}(\sigma(I-1)) X_3^{(M)}(\sigma(I-1))}, 
  \\
&
  \sigma(I)=\sigma(I-1)
  +\dfrac{\tau_{K^{31}(I-1)+1}^{31}-T_{31}^{(M)}(\sigma(I-1))
  }{\dfrac{\lambda}{M} X_3^{(M)}(\sigma(I-1)) X_1^{(M)}(\sigma(I-1))}.
\end{align*}
For the selected number $j=1,2,3$ the propositions $P_j(I-1, I)$ 
and $P_j(I)$ are proved similarly as in [Step 1].

Here we shall prove the existence of the solution of (\ref{eq2.2}). 
In [Case~A] the propositions $P_j(I-1, I)$ and $P_j(I)$ hold for 
$1 \leq j \leq 3$ in each case. The proposition $P_j(I)$ leads
\begin{align*}
  \sum_{i=1}^{K^{j j+1}(I)} 1 
  = K^{j j+1}(I)=N_{j j+1}\left(T_{j j+1}^{(M)}(\sigma(I))\right),
\end{align*}
and $P_j(I-1, I)$ leads, for $\sigma(I-1)<t<\sigma(I)$,
\begin{align*}
  \sum_{i=1}^{K^{j j+1}(I-1)} 1 
  = K^{j j+1}(I-1)=N_{j j+1}\left(T_{j j+1}^{(M)}(t)\right).
\end{align*}
Thus for any $t$, $\sigma(I-1)<t<\sigma(I)$, (\ref{eq2.5}) is replaced by
\begin{align*}
  \left\{\renewcommand{\arraystretch}{1.5}
    \begin{array}{@{}l@{}}
      X_1^{(M)}(t)
      =X_1^{(M)}(0)+N_{12}\left(T_{12}^{(M)}(t)\right)
      -N_{31}\left(T_{31}^{(M)}(t)\right),   \\
      X_2^{(M)}(t)
      =X_2^{(M)}(0)+N_{23}\left(T_{23}^{(M)}(t)\right)
      -N_{12}\left(T_{12}^{(M)}(t)\right), 
      \\
      X_3^{(M)}(t) 
      =X_3^{(M)}(0)+N_{31}\left(T_{31}^{(M)}(t)\right)
      -N_{23}\left(T_{23}^{(M)}(t)\right),
    \end{array}\right.
\end{align*}
At $\sigma(I)$, (\ref{eq2.6}) is replaced by
\begin{align*}
  \left\{\renewcommand{\arraystretch}{1.5}
    \begin{array}{@{}l@{}}
      X_1^{(M)}(\sigma(I))=X_1^{(M)}(0)+N_{12}\left(T_{12}^{(M)}(\sigma(I))\right)
      -N_{31}\left(T_{31}^{(M)}(\sigma(I))\right), 
      \\
      X_2^{(M)}(\sigma(I))
      =X_2^{(M)}(0)+N_{23}\left(T_{23}^{(M)}(\sigma(I))\right)
      -N_{12}\left(T_{12}^{(M)}(\sigma(I))\right), 
      \\
      X_3^{(M)}(\sigma(I))
      =X_3^{(M)}(0)+N_{31}\left(T_{31}^{(M)}(\sigma(I))\right)
      -N_{23}\left(T_{23}^{(M)}(\sigma(I))\right).
    \end{array}\right.
\end{align*}

Consequently there exists a solution of the system of (\ref{eq2.2}) in [Case~A] 
and $I-1$ in (\ref{eq2.4}) is replaced by $I$ in (\ref{eq2.6}).

\noindent
[Case B] We consider the case of $X_{j-1}^{(M)}(\sigma(l))>0$, 
$X_j^{(M)}\left(\sigma\left(l'\right)\right)>0$, 
$X_j^{(M)}\left(\sigma\left(l''\right)\right)=0$ 
and 
$X_{j+1}^{(M)}(\sigma(l))>0$ for $0 \leq l \leq I-1$, $0 \leq l'<k$ 
and $k \leq l'' \leq I-1(0 \leq k \leq I-1)$.
This is the case that the value of one of the random variables has come to zero 
and kept zero in $[\sigma(k), \sigma(I-1)]$. 
For example we prove in the case of $j=2$.

In this case $K^{12}(k)=\cdots=K^{12}(I-1)$ and $K^{23}(k)=\cdots=K^{23}(I-1)$ 
is implicitly assumed. It follows that $X_2^{(M)}(t)=0$ for any 
$t \in[\sigma(k), \sigma(I-1)]$. 
Thus $T_{12}^{(M)}(\sigma(k))=\cdots=T_{12}^{(M)}(\sigma(I-1))=T_{12}^{(M)}(t)$ 
and $T_{23}^{(M)}(\sigma(k))=\cdots=T_{23}^{(M)}(\sigma(I-1))=T_{23}^{(M)}(t)$ 
for any $t \in[\sigma(k), \sigma(I-1)]$. 
In addition to this, $\tau_{K^{12}(k)}^{12}=\cdots=\tau_{K^{12}(I-1)}^{12}$ 
and $\tau_{K^{23}(k)+1}^{23}=\cdots=\tau_{K^{23}(I-1)+1}^{23}$ hold.

Considering $\sigma(I-1)+\frac{\tau_{K^{j j+1}(I-1)+1}^{j j+1}
-T_{j j+1}^{(M)}(\sigma(I-1))}{\frac{\lambda}{M} X_j^{(M)}(\sigma(I-1)) 
X_{j+1}^{(M)}(\sigma(I-1))}$ in the minimum of (\ref{eq2.7}), 
the denominators are zero in [Case~B] and the numerators are positive because of 
$P_j(I-1)$ for $j=1,2$. 
Thus we replace these two terms of $j=1,2$ by infinity in the minimum for [Case~B]. 
We determine $\sigma(I)$ as
\begin{align*}
\sigma(I) 
  & =\min \left\{\sigma(I-1)
    +\dfrac{\tau_{K^{31}(I-1)+1}^{31}-T_{31}^{(M)}(\sigma(I-1))
    }{\dfrac{\lambda}{M} X_3^{(M)}(\sigma(I-1)) X_1^{(M)}(\sigma(I-1))}, 
    \infty, \infty\right\} \\
  & =\sigma(I-1)
    +\dfrac{\tau_{K^{31}(I-1)+1}^{31}-T_{31}^{(M)}(\sigma(I-1))
    }{\dfrac{\lambda}{M} X_3^{(M)}(\sigma(I-1)) X_1^{(M)}(\sigma(I-1))} .
\end{align*}
By taking the minimum of $1 \leq j \leq 3$, we count up one in $K^{j j+1}(I)$ 
for the selected number $j=3$ and we do not count up one for the not selected number 
$j=1,2$. 
We have $K^{31}(I)=K^{31}(I-1)+1$, $K^{12}(I)=K^{12}(I-1)$ 
and $K^{23}(I)=K^{23}(I-1)$. 
Thus the implicit assumption is satisfied to $I$-th step.

By $P_3(I-1)$ the numerator is positive and $\sigma(I)>\sigma(I-1)$. We have
\begin{align*}
  \tau_{K^{31}(I-1)+1}^{31}
    =T_{31}^{(M)}(\sigma(I-1))+\frac{\lambda}{M} X_3^{(M)}
    (\sigma(I-1)) X_1^{(M)}(\sigma(I-1))(\sigma(I)-\sigma(I-1)) .
\end{align*}

Similarly as in [Step 1] in [Case A], 
$P_3(I-1, I)$ and $P_3(I)$ hold for the selected number $j=3$.

\noindent
[Step 3] We consider the propositions for the not selected number $j=1$.

In [Case B] for $\sigma(I-1)<t<\sigma(I)$ we have
\begin{align*}
  T_{12}^{(M)}(t)
    =\int_0^t X_1^{(M)}(s) X_2^{(M)}(s) d s=T_{12}^{(M)}(\sigma(I-1)),
\end{align*}
and
\begin{align*}
  T_{12}^{(M)}(\sigma(I))
    =\int_0^{\sigma(I)} X_1^{(M)}(s) X_2^{(M)}(s) d s=T_{12}^{(M)}(\sigma(I-1)).
\end{align*}
From $P_1(I-1)$ it follows that
\begin{align*}
&  \tau_{K^{12}(I-1)}^{12} 
    \leq T_{12}^{(M)}(\sigma(I-1))
    =T_{12}^{(M)}(t)<\tau_{K^{12}(I-1)+1}^{12}. \\
&
   \tau_{K^{12}(I-1)}^{12}
    =\tau_{K^{12}(I)}^{12} \leq T_{12}^{(M)}(\sigma(I-1))
    =T_{12}^{(M)}(\sigma(I))<\tau_{K^{12}(I-1)+1}^{12}=\tau_{K^{12}(I)+1}^{12}.
\end{align*}
Therefore $P_1(I-1, I)$ and $P_1(I)$ hold. \hfill $\diamond$

For the not selected number $j=2$, $P_2(I-1, I)$ and $P_2(I)$ also hold.

In ($\sigma(I-1), \sigma(I)$), (\ref{eq2.5}) is replaced by 
(\ref{eq2.8}) and (\ref{eq2.6}) is also replaced by (\ref{eq2.8}) 
at $\sigma(I)$.

It follows that there exists a solution of the system in [Case~B] 
and $I-1$ in (\ref{eq2.4}) is replaced by $I$ in (\ref{eq2.6}).

\noindent
[Case C] We consider the case of $X_{j-1}^{(M)}(\sigma
(l'''))>0$, $X_{j-1}^{(M)}(\sigma(I-1))=0$, 
$X_j^{(M)}(\sigma(l'))>0$, $X_j^{(M)}(\sigma(l'))=0$ and 
$X_{j+1}^{(M)}(\sigma(l))>0$ for $0 \leq l \leq I-1$, 
$0 \leq l'<k$, $k \leq l'' \leq I-1$ and 
$0 \leq l'''<I-1$, $(0 \leq k<I-1)$. 
This is the first case in which the values of two of the random 
variables have come to zero at $\sigma(I-1)$, after several times 
of [Case~B]. For example we prove in the case of $j=2$.

By [Case B] we implicitly have $K^{12}(k)=\cdots=K^{12}(I-1)$ 
and $K^{23}(k)=\cdots= K^{23}(I-1)$. 
Thus for $t$, $t \in[\sigma(k), \sigma(I-1)]$, $X_2^{(M)}(t)=0$ 
and $T_{ii+1}^{(M)}(\sigma(I-1))=\cdots=T_{i i+1}^{(M)}(\sigma(k))
=T_{i i+1}^{(M)}(t)$ $(i=1,2)$.

Considering $\sigma(I-1)+\frac{\tau_{K^{j j+1}(I-1)+1}^{j j+1}-T_{j j+1}^{(M)}(\sigma(I-1))}{\frac{\lambda}{M} X_j^{(M)}(\sigma(I-1)) X_{j+1}^{(M)}(\sigma(I-1))}$
in the minimum of (\ref{eq2.7}), all denominators are zero 
in [Case~C] and all numerators are positive because of $P_j(I-$ 1) 
for $j=1,2,3$. Thus we replace all terms by infinity in (\ref{eq2.7}) 
for the present case. We determine $\sigma(I)$ as
\begin{align*}
  \sigma(I) & =\min \{\infty, \infty, \infty\} \\
            & =\infty.
\end{align*}
Thus we do not need the solution of the system at $\sigma(I)=\infty$.

\noindent
[Step 4] We consider propositions for $j=1,2,3$.

For $t$, $\sigma(I-1)<t<\sigma(I)=\infty$, we have $(j=1,2,3)$
\begin{align*}
  T_{j j+1}^{(M)}(t)
    =\int_0^t X_j^{(M)}(s) X_{j+1}^{(M)}(s) ds
    =T_{j j+1}^{(M)}(\sigma(I-1)),
\end{align*}
From $P_j(I-1)$ it follows that
\begin{align*}
  \tau_{K^{j j+1}(I-1)}^{12} 
    \leq T_{j j+1}^{(M)}(\sigma(I-1))
    =T_{j j+1}^{(M)}(t) < \tau_{K^{j j+1}(I-1)+1}^{j j+1},
\end{align*}
Therefore $P_j(I-1, I)$ for $1 \leq j \leq 3$ hold.
\hfill $\diamondsuit$

In ($\sigma(I-1), \infty$), (\ref{eq2.5}) is replaced 
by (\ref{eq2.8}).

Consequently we have a solution in [Case C].

By mathematical induction there exists a solution of the system 
of (\ref{eq2.2}) in $[0, \infty)$.

Now we shall prove that the solution constructed above is unique.

Each random variable $X_j^{(M)}(t)$ has a non-negative initial value. 
In the neighborhood of $t=0$ we see that for $j=1,2,3$,
\begin{align*}
  \frac{\lambda}{M} \int_0^t X_j^{(M)}(s) X_{j+1}^{(M)}(s) d s
    =\frac{\lambda}{M} X_j^{(M)}(0) X_{j+1}^{(M)}(0) t \geq 0.
\end{align*}
Thus the integrals are monotonously non-decreasing in the neighborhood 
of $t=0$. 
Each random variable $X_j^{(M)}(t)$ is integer valued $(j=1,2,3)$. 
If one of the random variables is negative valued after several jumps 
of the system from the non-negative initial value, 
it goes through the value zero. 
We see that the random variables $X_j^{(M)}(*)$ $(1 \leq j \leq 3)$ 
are non-negative by the following claim.

Put $X_k^{(M)}(t)$ $(1 \leq k \leq 3)$ to be a solution of the system 
of (\ref{eq2.2}). 
We claim that when $X_{j-1}^{(M)}(t) \geq 0$ $X_{j+1}^{(M)}(t) \geq 0$ 
and $X_j^{(M)}(t)=0$ for some $t \in(0, \infty)$ and for some 
$j \in\{1,2,3\}$, $X_j^{(M)}(s)=0$ holds for any $s \geq t$.

We set $u$, $u>t$, to be the first jump time of both 
$N_{j-1 j}(\frac{\lambda}{M} \int_0^* X_{j-1}^{(M)}(s) 
X_j^{(M)}(s) d s)$ and $N_{j j+1}(\frac{\lambda}{M} 
\int_0^* X_j^{(M)}(s) X_{j+1}^{(M)}(s) d s)$. 
Then it follows that $X_j^{(M)}(s)=0$ for any $s$, $t \leq s<u$. 
Since $\frac{\lambda}{M} \int_0^s X_{j-1}^{(M)}(s) X_j^{(M)}(s) d s$ 
and $\frac{\lambda}{M} \int_0^s X_j^{(M)}(s) X_{j+1}^{(M)}(s) d s$ 
are continuous,
\begin{align*}
& 
  \frac{\lambda}{M} \int_0^u X_{j-1}^{(M)}(s) X_j^{(M)}(s) d s
    =\frac{\lambda}{M} \int_0^t X_{j-1}^{(M)}(s) X_j^{(M)}(s) d s, 
    \\
& 
  \frac{\lambda}{M} \int_0^u X_j^{(M)}(s) X_{j+1}^{(M)}(s) d s
    =\frac{\lambda}{M} \int_0^t X_j^{(M)}(s) X_{j+1}^{(M)}(s) d s.
\end{align*}
Therefore we have
\begin{align*}
& 
  N_{j-1 j}\left(\frac{\lambda}{M} 
  \int_0^u X_{j-1}^{(M)}(s) X_j^{(M)}(s) d s\right)
    =N_{j-1 j}\left(\frac{\lambda}{M} 
      \int_0^t X_{j-1}^{(M)}(s) X_j^{(M)}(s) d s\right),
  \\
&
  N_{j j+1}\left(\frac{\lambda}{M} 
  \int_0^u X_j^{(M)}(s) X_{j+1}^{(M)}(s) d s\right)
    =N_{j j+1}\left(\frac{\lambda}{M} 
      \int_0^t X_j^{(M)}(s) X_{j+1}^{(M)}(s) d s\right).
\end{align*}
This is contradiction. Therefore the claim holds.\hfill $\sharp$

The random variables $X_j^{(M)}(*)$ are non-negative, 
bounded and integer valued in $[0, M]$ for $1 \leq j \leq 3$. 
The integrals $T_{j j+1}^{(M)}(t)$ $(1 \leq j \leq 3)$ are non-negative 
and monotonously non-decreasing. And $T_{j j+1}^{(M)}(t)$ are bounded for 
$1 \leq j \leq 3$. It follows that all possible classifications are 
covered in the following proof.

We shall prove uniqueness of the solution of the system by mathematical 
induction.

The initial value of the random variables $X_j^{(M)}(*)$ 
$(1 \leq j \leq 3)$ are given. 
At $\sigma(0)=0$ there exists a unique solution.

In $[0, \sigma(I-1)]$ we assume that there exists a unique solution of 
the system of (\ref{eq2.2}) and that the solution coincides with the 
solution constructed actually in the proof of existence the system of 
(\ref{eq2.2}). 
Note that the propositions hold in $[0, \sigma(I-1)]$.

Whenever monotonously non-decreasing $\frac{\lambda}{M} 
\int_0^t X_k^{(M)}(s) X_{k+1}^{(M)}(s) d s$ comes to the jump time of the 
Poisson process, the random variable $X_k^{(M)}(t)$ increases in the 
width of one and the random variable $X_{k+1}^{(M)}(t)$ decreases in the 
width of one $(1 \leq k \leq 3$). 
We trace the time and search the next jump time from $\sigma(I-1)$. 
As $\frac{\lambda}{M} \int_0^t X_k^{(M)}(s) X_{k+1}^{(M)}(s) d s$ 
are monotonously non-decreasing ($1 \leq k \leq 3$), 
the system has a change of the previous state at $s(I)$ such that
\begin{align*}
  s(I)=\min _{1 \leq j \leq 3}
    \left\{\inf \left\{t>\sigma(I-1): 
      T_{j j+1}^{(M)}(t)
      =\tau_{K^{j j+1}(I-1)+1}^{j j+1}\right\}\right\},
\end{align*}
where we set for $1 \leq j \leq 3$
\begin{align*}
  T_{j j+1}^{(M)}(t)
    =\frac{\lambda}{M} 
      \int_0^t X_j^{(M)}(s) X_{j+1}^{(M)}(s) d s.
\end{align*}

As we shall search the next jump, it follows that
\begin{align*}
& 
  \inf \left\{t>\sigma(I-1): 
    T_{j j+1}^{(M)}(t)=\tau_{K^{j j+1}(I-1)+1}^{j j+1}\right\}\\
&\quad
  = \inf \Bigg\{t>\sigma(I-1)
  : \frac{\lambda}{M} X_j^{(M)}(\sigma(I-1)) 
    X_{j+1}^{(M)}(\sigma(I-1))(t-\sigma(I-1))  \\
&\quad 
  =\tau_{K^{j j+1}(I-1)+1}^{j j+1} 
  - T_{j j+1}^{(M)}(\sigma(I-1))\Bigg\}.
\end{align*}

\noindent
[Case a] We consider the case of $X_j^{(M)}(\sigma(l))>0$ 
for $0 \leq l \leq I-1$ and $1 \leq j \leq 3$.

Since
\begin{align*}
& 
  \inf \left\{t>\sigma(I-1): 
    T_{j j+1}^{(M)}(t)=\tau_{K^{j j+1}(I-1)+1}^{j j+1}\right\} 
    \\
&\quad 
  =\inf \left\{t>\sigma(I-1): t 
  =\sigma(I-1)+\frac{\tau_{K^{j j+1}(I-1)+1}^{j j+1} 
    - T_{j j+1}^{(M)}(\sigma(I-1))
    }{\dfrac{\lambda}{M} X_j^{(M)}(\sigma(I-1)) 
    X_{j+1}^{(M)}(\sigma(I-1))}\right\},
\end{align*}
we have
\begin{align*}
  s(I)=\min _{1 \leq j \leq 3}
  \left\{\sigma(I-1)+\frac{\tau_{K^{j j+1}(I-1)+1}^{j j+1}
  - T_{j j+1}^{(M)}(\sigma(I-1))
  }{\dfrac{\lambda}{M} X_j^{(M)}(\sigma(I-1)) 
  X_{j+1}^{(M)}(\sigma(I-1))}\right\}.
\end{align*}

\noindent
[Case b] We consider the case of $X_{j-1}^{(M)}(\sigma(l))>0$, 
$X_j^{(M)}(\sigma(l'))>0$, $X_j^{(M)}(\sigma(l''))=0$ and 
$X_{j+1}^{(M)}(\sigma(l))>0$ for $0 \leq l \leq I-1$, 
$0 \leq l'<k$ and $k \leq l'' \leq I-1$ $(0 \leq k \leq I-1)$. 
For example we prove in the case of $j=2$.

We have $T_{j j+1}^{(M)}(t)=T_{j j+1}^{(M)}(\sigma(I-1))
<\tau_{K^{j j+1}(I-1)+1}^{j j+1}$ for 
$t>\sigma(I-1)$ by $P_j(I-1)$ $(j=1,2)$. 
Thus
\begin{align*}
  \left\{t>\sigma(I-1): 
    T_{j j+1}^{(M)}(t)
    =\tau_{K^{j j+1}(I-1)+1}^{j j+1}\right\}=\emptyset.
\end{align*}

It follows that
\begin{align*}
  s(I)=\min \left\{\sigma(I-1)
  +\frac{\tau_{K^{31}(I-1)+1}^{31} - T_{31}^{(M)}(\sigma(I-1))
  }{\dfrac{\lambda}{M} X_3^{(M)}(\sigma(I-1)) 
   X_1^{(M)}(\sigma(I-1))}, \infty, \infty\right\},
\end{align*}
where we put $\inf \emptyset=\infty$.

\noindent
[Case c] We consider the case of $X_{j-1}^{(M)}(\sigma(l'''))>0$, 
$X_{j-1}^{(M)}(\sigma(I-1))=0$, $X_j^{(M)}(\sigma(l'))>0$, 
$X_j^{(M)}(\sigma(l'))=0$ and $X_{j+1}^{(M)}(\sigma(l))>0$ 
for $0 \leq l \leq I-1$, $0 \leq l'<k$, $k \leq l'' \leq I-1$ 
and $0 \leq l'''<I-1$ $(0 \leq k \leq I-1)$. 
For example we prove in the case of $j=2$.

From $P_j(I-1)$ it follows that for $j=1,2,3$,
\begin{align*}
  \left\{t>\sigma(I-1): T_{j j+1}^{(M)}(t)
    =\tau_{K^{j j+1}(I-1)+1}^{j j+1}\right\}=\emptyset.
\end{align*}

Thus we have
\begin{align*}
  s(I)=\min \{\infty, \infty, \infty\}.
\end{align*}

The jump time $\sigma(I)$ constructed in [Case A]$\sim$[Case C] of 
the proof of existence of the system of (\ref{eq2.2}) coincides with 
$s(I)$ of [Case~a] $\sim$ [Case~c]. 
There are no methods to construct the solution of the system of 
(\ref{eq2.2}) except the construction of the proof of existence of 
a solution, since $\frac{\lambda}{M} \int_0^t X_j^{(M)}(s) 
X_{j+1}^{(M)}(s) d s$ $(1 \leq j \leq 3)$ are monotonously non-decreasing.

Moreover $\sigma(I)$ is determined by $(\sigma(l), X^{(M)}(\sigma(l)))
_{0 \leq l \leq I-1}$ and by the jump times of standard Poisson processes. Thus the constructed solution is unique.

In $(\sigma(I-1), \sigma(I)]$ there exists a unique solution 
and it coincides with the solution constructed actually in the proof 
of existence of a solution.

By mathematical induction we prove that there exists a unique solution 
in $[0, \infty)$. \hfill $\square$
\end{proof}
\begin{coro}
  \label{coro2.1}
  There exists a unique solution of equation~{\normalfont (\ref{eq2.2})}, 
  when $t \in\left[0, t_0\right)$ for $t_0 \in[0, \infty)$.
\end{coro}

\begin{proof}\normalfont
  The proof of existence and the proof of uniqueness of the system of 
  (\ref{eq2.2}) is done step by step. We stop the proof when the step 
  excess the time $t_0$. Then we have the present corollary.
  \hfill $\square$
\end{proof}

For any $v$, $v \geq 0$, we define
\begin{align*}
  N_{j j+1}^v(t)
  = \begin{cases}
    N_{j j+1}(t), & 0 \leq t \leq v, \\ 
    N_{j j+1}(v), & t>v.
  \end{cases}
\end{align*}
We consider the system $N_{12}$ replaced by $N_{12}^v$ 
in (\ref{eq2.2}). This system is
\begin{equation}
  \left\{\begin{aligned}
  X_1^{(M)}(t)=X_1^{(M)}(0) & +N_{12}^v\left(\frac{\lambda}{M} \int_0^t X_1^{(M)}(s) X_2^{(M)}(s) d s\right) \\
  & -N_{31}\left(\frac{\lambda}{M} \int_0^t X_3^{(M)}(s) X_1^{(M)}(s) d s\right), \\
  X_2^{(M)}(t)=X_2^{(M)}(0) & +N_{23}\left(\frac{\lambda}{M} \int_0^t X_2^{(M)}(s) X_3^{(M)}(s) d s\right) \\
  & -N_{12}^v\left(\frac{\lambda}{M} \int_0^t X_1^{(M)}(s) X_2^{(M)}(s) d s\right), \\
  X_3^{(M)}(t)=X_3^{(M)}(0) & +N_{31}\left(\frac{\lambda}{M} \int_0^t X_3^{(M)}(s) X_1^{(M)}(s) d s\right) \\
  & -N_{23}\left(\frac{\lambda}{M} \int_0^t X_2^{(M)}(s) X_3^{(M)}(s) d s\right), \\
  X_1^{(M)}(0)+X_2^{(M)}(0) & +X_3^{(M)}(0)=M.
  \end{aligned}\right.
  \label{eq2.8}
\end{equation}
We have the following theorem.

\begin{theo}\label{theo2.2}
  There exists a unique solution for the system of 
  {\normalfont(\ref{eq2.8})}. 
\end{theo}

\begin{proof}\normalfont
  We fix a sample path. We use the same definition 
  in Theorem~\ref{theo2.1}.
\end{proof}

There exists an integer $K_v$, $K_v \geq 0$ such that 
$\tau_{K_v}^{12} \leq v<\tau_{K_v+1}^{12}$.

When for the fixed sample path the monotonously non-decreasing function 
$T_{12}^{(M)}(*)$ does not reach $\tau_{K_v+1}^{12}$, 
we prove the present theorem in just the same way as 
Theorem~\ref{theo2.1}.

We consider the case in the following way. 
There is the smallest integer $I_0$, $I_0 \geq 1$, 
which denotes the step, such that $K^{12}(I_0-1)=K_v$ and 
$K^{12}(I_0)=K_v+1$ in Theorem~\ref{theo2.1}, 
when $X_1^{(M)}(\sigma(l))>0$ and $X_2^{(M)}(\sigma(l))>0$ for 
$0 \leq l \leq I_0-1$. 
In this situation $I_0$ is the smallest integer of 
$T_{12}^{(M)}(\sigma(I_0))=\tau_{K_{v+1}}^{12}$.

We shall prove existence of the solution of the system by mathematical 
induction on $I$ $(I \geq I_0)$.

Note that the proof from $I_0-1$ to $I_0$ is slightly different from the 
proof from $I-1$ to $I$ in the classification of cases of the mathematical induction.

For the mathematical induction we assume the solution
\begin{equation}
  \left\{
    \begin{array}{@{}l@{}}
      X_1^{(M)}(\sigma(I_0-1))
        =X_1^{(M)}(0)+\dsum_{i=1}^{K^{12}(I_0-1)}(+1)
        +\dsum_{i=1}^{K^{31}(I_0-1)}(-1), \\
      X_2^{(M)}(\sigma(I_0-1))
        =X_2^{(M)}(0)+\dsum_{i=1}^{K^{23}(I_0-1)}(+1)
        +\dsum_{i=1}^{K^{12}(I_0-1)}(-1), \\
      X_3^{(M)}(\sigma(I_0-1))
        =X_3^{(M)}(0)+\dsum_{i=1}^{K^{31}(I_0-1)}(+1)
        +\dsum_{i=1}^{K^{23}(I_0-1)}(-1),
      \end{array}\right.
  \label{eq2.9}
\end{equation}
at $t=\sigma(I_0-1)$ with $P_j(I_0-1)$ $(j=1,2,3)$.

We construct the solution of the system of (\ref{eq2.8}) 
for $t \in(\sigma(I_0-1), \sigma(I_0))$ as
\begin{equation}
  \left\{
    \begin{array}{@{}l@{}}
      X_1^{(M)}(t)=X_1^{(M)}(0)
        +\dsum_{i=1}^{K^{12}(I_0-1)}(+1)
        +\dsum_{i=1}^{K^{31}(I_0-1)}(-1), \\
      X_2^{(M)}(t)=X_2^{(M)}(0)
        +\dsum_{i=1}^{K^{23}(I_0-1)}(+1)
        +\dsum_{i=1}^{K^{12}(I_0-1)}(-1), \\
      X_3^{(M)}(t)=X_3^{(M)}(0)
        +\dsum_{i=1}^{K^{31}(I_0-1)}(+1)
        +\dsum_{i=1}^{K^{23}(I_0-1)}(-1),
    \end{array}\right.
  \label{eq2.10}
\end{equation}
and at $t=\sigma(I_0)$ as
\begin{equation}
  \left\{\renewcommand{\arraystretch}{3}
    \begin{array}{@{}l@{}}
      X_1^{(M)}(\sigma(I_0))=X_1^{(M)}(0)
        +\dsum_{i=1}^{K^{12}(I_0)}(+1)
        +\dsum_{i=1}^{K^{31}(I_0)}(-1), \\
      X_2^{(M)}(\sigma(I_0))=X_2^{(M)}(0)
        +\dsum_{i=1}^{K^{23}(I_0)}(+1)
        +\dsum_{i=1}^{K^{12}(I_0)}(-1),\\
      X_3^{(M)}(\sigma(I_0))=X_3^{(M)}(0)
        +\dsum_{i=1}^{K^{31}(I_0)}(+1)+\dsum_{i=1}^{K^{23}(I_0)}(-1),
  \end{array}\right.
  \label{eq2.11}
\end{equation}
where $\sigma(I_0)$ and $K^{j j+1}(I_0)$ are as follows in [Case~A'1] 
and [Case~B'1].

\noindent
[Case A'1] We consider the case of $X_j^{(M)}(\sigma(l))>0$ for 
$0 \leq l \leq I_0-1$ and $1 \leq j \leq 3$. 
This case describes that all random variables have positive values 
from 0-th step to $I_0-1$-th step.

In Theorem~\ref{theo2.1} we have the jump time of (\ref{eq2.2}) as 
(\ref{eq2.7}). 
The standard Poisson process $N_{12}(*)$ in (\ref{eq2.2}) is replaced by 
$N_{12}^v(*)$ in (\ref{eq2.8}) in the present theorem. 
There are no jumps as to $N_{12}(*)$ after the $I_0$-th step 
and we have $K^{12}(I_0-1)=K^{12}(I_0)$. 
We replace $\tau_{K^{12}(I_0-1)+1}^{12}$ by infinity. 
In [Case~A1] we determine $\sigma(I_0)$ as
\begin{align*}
  \sigma(I_0)=\min _{2 \leq j \leq 3}
    \left\{\sigma(I_0-1)
      +\frac{\tau_{K^{j j+1}(I_0-1)+1 }^{j j+1}
      - T_{j j+1}^{(M)}(\sigma(I_0-1))
      }{\dfrac{\lambda}{M} X_j^{(M)}(\sigma(I_0-1)) 
        X_{j+1}^{(M)}(\sigma(I_0-1))}, \infty\right\}.
\end{align*}
By taking the minimum of $2 \leq j \leq 3$, we count up one in 
$K^{j j+1}(I_0)$ for the selected number and we do not count up one 
for the not selected number. If by taking the minimum of 
$2 \leq j \leq 3$ the number $j=2$ is selected, for example, 
then we have $K^{12}(I_0)=K^{12}(I_0-1)$, $K^{23}(I_0)=K^{23}(I_0-1)+1$ 
and $K^{31}(I_0)=K^{31}(I_0-1)$. 
If by taking the minimum of $2 \leq j \leq 3$ the numbers $j=2,3$ 
are selected, then we have $K^{12}(I_0)=K^{12}(I_0-1)$, 
$K^{23}(I_0)=K^{23}(I_0-1)+1$ and $K^{31}(I_0)=K^{31}(I_0-1)+1$.

If by taking the minimum of $2 \leq j \leq 3$ the number $j=2$ is 
selected, we have $K^{12}(I_0)=K^{12}(I_0-1)$, 
$K^{23}(I_0)=K^{23}(I_0-1)+1$ and $K^{31}(I_0)=K^{31}(I_0-1)$. 
Then
\begin{align*}
& 
  \sigma(I_0)=\sigma(I_0-1)
    +\frac{\tau_{K^{23}(I_0-1)+1}^{23}-T_{23}^{(M)}(\sigma(I_0-1))
    }{\dfrac{\lambda}{M} X_2^{(M)}(\sigma(I_0-1)) X_3^{(M)}
    (\sigma(I_0-1))}, \\
& 
  \sigma(I_0)<\sigma(I_0-1)
    +\frac{\tau_{K^{31}(I_0-1)+1}^{31}-T_{31}^{(M)}(\sigma(I_0-1))
    }{\dfrac{\lambda}{M} X_3^{(M)}(\sigma(I_0-1)) X_1^{(M)}
    (\sigma(I_0-1))} .
\end{align*}
By $P_j(I_0-1)$ for $j=2,3$ numerators are positive. 
Thus $\sigma(I_0)>\sigma(I_0-1)$ and
\begin{align*}
& 
  \tau_{K^{23}(I_0-1)+1}^{23}
    \!=\!T_{23}^{(M)}(\sigma(I_0\!-\!1))
    \!+\!\frac{\lambda}{M} X_{12}^{(M)}(\sigma(I_0\!-\!1)) X_3^{(M)}
    (\sigma(I_0\!-\!1))(\sigma(I_0)\!-\!\sigma(I_0-1)), \\
& 
  \tau_{K^{31}(I_0-1)+1}^{31}
    \!>\!T_{31}^{(M)}(\sigma(I_0\!-\!1))
    \!+\!\frac{\lambda}{M} X_3^{(M)}(\sigma(I_0\!-\!1)) X_1^{(M)}
    (\sigma(I_0\!-\!1))(\sigma(I_0)\!-\!\sigma(I_0\!-\!1)) .
\end{align*}
[Step 5] We consider the propositions for the number $j=1$.

We see that for $\sigma(I_0-1)<t<\sigma(I_0)$
\begin{align*}
  \tau_{K^{12}(I_0-1)}^{12} 
    \leq T_{12}^{(M)}(\sigma(I_0-1))
    <T_{12}^{(M)}(t)<T_{12}^{(M)}(\sigma(I_0))
    <\tau_{K^{12}(I_0-1)+1}^{12}=\infty.
\end{align*}
where
\begin{align*}
  T_{12}^{(M)}(t)
    =T_{12}^{(M)}(\sigma(I_0-1))
    +\frac{\lambda}{M} X_1^{(M)}(\sigma(I_0-1)) X_2^{(M)}
    (\sigma(I_0-1))(t-\sigma(I_0-1))
\end{align*}
and
\begin{align*}
  T_{12}^{(M)}(\sigma(I_0))
    \!=\!T_{12}^{(M)}(\sigma(I_0\!-\!1))
      \!+\!\frac{\lambda}{M} X_1^{(M)}(\sigma(I_0\!-\!1)) X_2^{(M)}
      (\sigma(I_0\!-\!1))(\sigma(I_0)\!-\!\sigma(I_0\!-\!1))
\end{align*}
This leads
\begin{align*}
&
  \tau_{K^{12}(I_0-1)}^{12} 
    \leq T_{12}^{(M)}(\sigma(I_0-1))
    <T_{12}^{(M)}(t)<\tau_{K^{12}(I_0-1)+1}^{12}=\infty, \\
&  
  \tau_{K^{12}(I_0-1)}^{12}
    \!=\!\tau_{K^{12}(I_0)}^{12} \!\leq\! T_{12}^{(M)}(\sigma(I_0\!-\!1))
    \!<\!T_{12}^{(M)}(\sigma(I_0))\!<\!\tau_{K^{12}(I_0-1)+1}^{12}
    \!=\!\tau_{K^{12}(I_0)+1}^{12}\!=\!\infty.
\end{align*}
Therefore $P_1(I_0-1, I)$ and $P_1(I_0)$ hold.\hfill $\diamondsuit$

For the selected number $j=2$, $P_2(I_0-1, I_0)$ and 
$P_2(I_0)$ hold similarly as in [Step~1] of [Case~A] 
in Theorem~\ref{theo2.1}. 
For the not selected number $j=3$, similarly as in [Step~2] 
of [Case~A] in Theorem~\ref{theo2.1}, 
$P_3(I_0-1, I_0)$ and $P_3(I_0)$ hold.

If by taking the minimum of $2 \leq j \leq 3$ the numbers $j=2,3$ 
are selected, we have $K^{12}(I_0)=K^{12}(I_0-1)$, 
$K^{23}(I_0)=K^{23}(I_0-1)+1$ and $K^{31}(I_0)=K^{31}(I_0-1)+1$. 

\noindent
And
\begin{align*}
&
  \sigma(I_0) = \sigma(I_0-1)
  +\frac{\tau_{K^{23}(I_0-1)+1}^{23}-T_{23}^{(M)}(\sigma(I_0-1))
  }{\dfrac{\lambda}{M} X_2^{(M)}(\sigma(I_0-1)) X_3^{(M)}
  (\sigma(I_0-1))}, \\
&
  \sigma(I) =\sigma(I_0-1)
  +\frac{\tau_{K^{31}(I_0-1)+1}^{31}-T_{31}^{(M)}(\sigma(I_0-1))
  }{\dfrac{\lambda}{M} X_3^{(M)}(\sigma(I_0-1)) X_1^{(M)}
  (\sigma(I_0-1))}.
\end{align*}

$P_1(I_0-1, I_0)$ and $P_1(I_0)$ hold in a similar way as [Step~5]. 
For the selected number $j=2,3$, similarly as in [Step~1] of [Case~A] 
in Theorem~\ref{theo2.1}, $P_j(I_0-1, I_0)$ 
and $P_j(I_0)$ hold for $2 \leq j \leq 3$.

Note that the proposition $P_1(I_0)$ leads
\begin{align*}
  \sum_{i=1}^{K^{12}(I_0)} 1
    =K^{12}(I_0)=K^{12}(I_0-1)=N_{12}^v
    \left(T_{12}^{(M)}(\sigma(I_0))\right),
\end{align*}
and that for $\sigma(I_0-1)<t<\sigma(I_0)$ the proposition 
$P_1(I_0-1, I_0)$ leads
\begin{align*}
  \sum_{i=1}^{K^{12}(I_0-1)} 1=K^{12}(I_0-1)
  =N_{12}^v\left(T_{12}^{(M)}(t)\right).
\end{align*}
For $\sigma(I_0-1)<t<\sigma(I_0)$, (\ref{eq2.10}) is replaced by
\begin{equation}
  \left\{\renewcommand{\arraystretch}{1.5}
  \begin{array}{@{}l@{}}
  X_1^{(M)}(t)
    =X_1^{(M)}(0)+N_{12}^v
    \left(T_{12}^{(M)}(t)\right)-N_{31}\left(T_{31}^{(M)}(t)\right),
    \\
  X_2^{(M)}(t)
    =X_2^{(M)}(0)+N_{23}
    \left(T_{23}^{(M)}(t)\right)-N_{12}^v\left(T_{12}^{(M)}(t)\right),
    \\
  X_3^{(M)}(t)
    =X_3^{(M)}(0)+N_{31}
    \left(T_{31}^{(M)}(t)\right)-N_{23}\left(T_{23}^{(M)}(t)\right).
  \end{array}\right.
  \label{eq2.12}
\end{equation}
At $\sigma(I_0)$, (\ref{eq2.11}) is replaced by
\begin{align*}
  \left\{\renewcommand{\arraystretch}{1.5}
    \begin{array}{@{}l@{}}
      X_1^{(M)}(\sigma(I_0))
        =X_1^{(M)}(0)+N_{12}^v
        \left(T_{12}^{(M)}(\sigma(I_0))\right)
        -N_{31}\left(T_{31}^{(M)}(\sigma(I_0))\right), \\
      X_2^{(M)}(\sigma(I_0))
        =X_2^{(M)}(0)+N_{23}\left(T_{23}^{(M)}(\sigma(I_0))\right)
        -N_{12}^v\left(T_{12}^{(M)}(\sigma(I_0))\right),\\
      X_3^{(M)}(\sigma(I_0))
        =X_3^{(M)}(0)+N_{31}\left(T_{31}^{(M)}(\sigma(I_0))\right)
        -N_{23}\left(T_{23}^{(M)}(\sigma(I_0))\right).
  \end{array}\right.
\end{align*}

Consequently there exists a solution of the system of (\ref{eq2.8}) in [Case~A'1] 
and $I_0-1$ in (\ref{eq2.9}) is replaced by $I_0$ in (\ref{eq2.11}).

\noindent
[Case B'1] We consider the case of $X_1^{(M)}(\sigma(l))>0$, 
$X_2^{(M)}(\sigma(l))>0$, $X_3^{(M)}(\sigma(l'))>$ 0 and 
$X_3^{(M)}(\sigma(l''))=0$ for $0 \leq l \leq I_0-1$, $0 \leq l'<k$, 
and $k \leq l'' \leq I_0-1$ $(0 \leq k \leq I_0-1)$. 
In this case the value of the random variable of species 3 has reached zero 
until $I_0-1$-th step after several times of [Case~B] 
in Theorem~\ref{theo2.1}.

For $j=2,3$, as to $\sigma(I_0-1)+\frac{\tau_{K^{jj+1} (I_0-1)+1}
^{jj +1} -T_{j j+1} ^{(M)}(\sigma(I_0-1))
}{\frac{\lambda}{M} X_j^{(M)}(\sigma(I_0-1)) X_{j+1}^{(M)}(\sigma(I_0-1))}$ 
the denominators are zero in [Case~B'1] and the 
numerators are positive because of $P_j(I_0-1)$. Thus we replace these terms by 
infinity just the same way in [Case~B] in Theorem~\ref{theo2.1}. 
We replace $\tau_{K^{12}(I_0-1)+1}^{12}$ by infinity just similarly as 
in [Case~A'1]. We determine $\sigma(I_0)$ as
\begin{align*}
  \sigma(I_0) &= \min \{\infty, \infty, \infty\} \\
   & =\infty.
\end{align*}
Thus we do not need the solution of the system at $\sigma(I_0)=\infty$.

\noindent
[Step 6] We consider the proposition for the number $j=1$.

Considering $P_1(I_0)$ for $\sigma(I_0-1)<t<\sigma(I_0)=\infty$
\begin{align*}
  \tau_{K^{12}(I_0-1)}^{12} 
    \leq T_{12}^{(M)}(t)
      <\tau_{K^{12}(I_0-1)+1}^{12}=\infty .
\end{align*}
Thus $P_1(I_0-1, I_0)$ holds.\hfill $\diamond$

Similarly as in [Step 4] in Theorem~\ref{theo2.1} $P_j(I_0-1, I_0)$ hold for 
$j=2,3$.

In ($\sigma(I_0-1), \infty$) (\ref{eq2.10}) is replaced by (\ref{eq2.12}).

It follows that there exists a solution of the system of (\ref{eq2.8}) 
in [Case~B'1].

For the mathematical induction from $I-1$ to $I$ we assume the solution
\begin{equation}
  \left\{\renewcommand{\arraystretch}{3}
    \begin{array}{@{}l@{}}
      X_1^{(M)}(\sigma(I-1))
        =X_1^{(M)}(0)+\dsum_{i=1}^{K^{12}(I-1)}(+1)+\dsum_{i=1}^{K^{31}(I-1)}(-1),
        \\
      X_2^{(M)}(\sigma(I-1))
        =X_2^{(M)}(0)+\dsum_{i=1}^{K^{23}(I-1)}(+1)+\dsum_{i=1}^{K^{12}(I-1)}(-1), 
        \\
      X_3^{(M)}(\sigma(I-1))
        =X_3^{(M)}(0)+\dsum_{i=1}^{K^{31}(I-1)}(+1)+\dsum_{i=1}^{K^{23}(I-1)}(-1),
  \end{array}\right.
  \label{eq2.13}
\end{equation}
with the propositions $P_j(I-1)$ $(j=1,2,3)$.

We construct the solution for $t \in(\sigma(I-1), \sigma(I))$ as
\begin{equation}
  \left\{\renewcommand{\arraystretch}{3}
    \begin{array}{@{}l@{}}
      X_1^{(M)}(t)=X_1^{(M)}(0)+\dsum_{i=1}^{K^{12}(I-1)}(+1)
        +\dsum_{i=1}^{K^{31}(I-1)}(-1), \\
      X_2^{(M)}(t)=X_2^{(M)}(0)+\dsum_{i=1}^{K^{23}(I-1)}(+1)
        +\dsum_{i=1}^{K^{12}(I-1)}(-1), \\
      X_3^{(M)}(t)=X_3^{(M)}(0)+\dsum_{i=1}^{K^{31}(I-1)}(+1)
        +\dsum_{i=1}^{K^{23}(I-1)}(-1),
    \end{array}\right. \label{eq2.14}
\end{equation}
and at $t=\sigma(I)$ as
\begin{equation}
  \left\{\renewcommand{\arraystretch}{3}
    \begin{array}{@{}l@{}}
      X_1^{(M)}(\sigma(I))
        =X_1^{(M)}(0)+\dsum_{i=1}^{K^{12}(I)}(+1)+\dsum_{i=1}^{K^{31}(I)}(-1),\\
      X_2^{(M)}(\sigma(I))=X_2^{(M)}(0)+\dsum_{i=1}^{K^{23}(I)}(+1)
        +\dsum_{i=1}^{K^{12}(I)}(-1), \\
      X_3^{(M)}(\sigma(I))=X_3^{(M)}(0)+\dsum_{i=1}^{K^{31}(I)}(+1)
      +\dsum_{i=1}^{K^{23}(I)}(-1),
    \end{array}\right. \label{eq2.15}
\end{equation}
where $\sigma(I)$ and $K^{j j+1}(I)$ are as follows in [Case~A']$\sim$[Case~D'].

\noindent
[Case A'] We consider the case of $X_j^{(M)}(\sigma(l))>0$ for 
$0 \leq l \leq I-1$ and $1 \leq j \leq 3$. 
This is the case that all random variables have positive values until $I-1$-th step.

In the present system of (\ref{eq2.8}) there are no jumps as to $N_{12}$ after 
$I_0-1$-th step and we implicitly assume $K^{12}(I_0-1)=\cdots=K^{12}(I-1)
=K^{12}(I)$. 
We replace $\tau_{K^{12}(I-1)+1}^{12}=\tau_{K^{12}(I_0-1)+1}^{12}$ 
by infinity in the system. Then we have
\begin{align*}
  \sigma(I)=\min _{2 \leq j \leq 3}
    \left\{\sigma(I-1)
      +\frac{\tau_{K^{j j+1}(I-1)+1}^{j j+1}-T_{j j+1}^{(M)}(\sigma(I-1))
      }{\dfrac{\lambda}{M} X_j^{(M)}(\sigma(I-1)) X_{j+1}^{(M)}(\sigma(I-1))}, \infty\right\}.
\end{align*}
By taking the minimum of $2 \leq j \leq 3$, we count up one in $K^{j j+1}(I)$ 
for the selected number and we do not count up one for the not selected number. 
If by taking the minimum of $2 \leq j \leq 3$ the number $j=2$ is selected, 
then we have $K^{12}(I)=K^{12}(I-1)$, $K^{23}(I)=K^{23}(I-1)+1$ and 
$K^{31}(I)=K^{31}(I-1)$. 
If by taking the minimum of $2 \leq j \leq 3$ the numbers $j=2,3$ are selected, 
then we have $K^{12}(I-1)=K^{12}(I_0-1)$, $K^{23}(I)=K^{23}(I-1)+1$ and 
$K^{31}(I)=K^{31}(I-1)+1$. 
The implicit assumption is satisfied to $I_0$-th step.

If by taking the minimum of $2 \leq j \leq 3$ the number $j=2$ is 
selected, we have
$K^{12}(I)=K^{12}(I-1)$, $K^{23}(I)=K^{23}(I-1)+1$ and 
$K^{31}(I)=K^{31}(I-1)$. Then
\begin{align*}
&
  \sigma(I)=\sigma(I-1)
    +\frac{\tau_{K^{23}(I-1)+1}^{23}-T_{23}^{(M)}(\sigma(I-1))
    }{\dfrac{\lambda}{M} X_2^{(M)}(\sigma(I-1)) X_3^{(M)}
    (\sigma(I-1))}, \\
&  
  \sigma(I)<\sigma(I-1)
    +\frac{\tau_{K^{31}(I-1)+1}^{31}-T_{31}^{(M)}(\sigma(I-1))
    }{\dfrac{\lambda}{M} X_3^{(M)}(\sigma(I-1)) 
    X_1^{(M)}(\sigma(I-1))} .
\end{align*}

If by taking the minimum of $2 \leq j \leq 3$ the numbers $j=2,3$ 
are selected, we have $K^{12}(I)=K^{12}(I-1)$, 
$K^{23}(I)=K^{23}(I-1)+1$ and $K^{31}(I)=K^{31}(I-1)+1$. 

\noindent
And
\begin{align*}
& 
  \sigma(I)=\sigma(I-1)
    +\frac{\tau_{K^{23}(I-1)+1}^{23}-T_{23}^{(M)}(\sigma(I-1))
    }{\dfrac{\lambda}{M} X_2^{(M)}(\sigma(I-1)) 
    X_3^{(M)}(\sigma(I-1))}, \\
& 
  \sigma(I)=\sigma(I-1)
    +\frac{\tau_{K^{31}(I-1)+1}^{31}-T_{31}^{(M)}(\sigma(I-1))
    }{\dfrac{\lambda}{M} X_3^{(M)}(\sigma(I-1)) X_1^{(M)}(\sigma(I-1))}.
\end{align*}

Similarly as in [Case A'1], 
in these above two cases $P_j(I-1, I)$ and $P_j(I)$ hold for 
$j=1,2,3$.

Note that the proposition $P_1(I)$ leads
\begin{align*}
  \sum_{i=1}^{K^{12}(I)} 1 
    = K^{12}(I)=K^{12}(I-1)=\cdots
    =K^{12}(I_0-1)=N_{12}^v
    (T_{12}^{(M)}(\sigma(I))),
\end{align*}
and that for $\sigma(I-1)<t<\sigma(I)$ the proposition 
$P_1(I-1, I)$ leads
\begin{align*}
  \sum_{i=1}^{K^{12}(I-1)} 1 
    = K^{12}(I-1)=\cdots
    =K^{12}(I_0-1)=N_{12}^v(T_{12}^{(M)}(t)).
\end{align*}

For $\sigma(I-1)<t<\sigma(I)$, (\ref{eq2.14}) is replaced by
\begin{equation}
  \left\{\renewcommand{\arraystretch}{1.5}
    \begin{array}{@{}l@{}}
      X_1^{(M)}(t)
        =X_1^{(M)}(0)+N_{12}^v\left(T_{12}^{(M)}(t)\right)
        -N_{31}\left(T_{31}^{(M)}(t)\right), \\
      X_2^{(M)}(t)
        =X_2^{(M)}(0)+N_{23}\left(T_{23}^{(M)}(t)\right)
        -N_{12}^v\left(T_{12}^{(M)}(t)\right),\\
      X_3^{(M)}(t)
        =X_3^{(M)}(0)+N_{31}\left(T_{31}^{(M)}(t)\right)
        -N_{23}\left(T_{23}^{(M)}(t)\right),
    \end{array}\right.
  \label{eq2.16}
\end{equation}
and, at $\sigma(I)$, (\ref{eq2.15}) is replaced by
\begin{equation}
  \left\{\renewcommand{\arraystretch}{1.5}
  \begin{array}{@{}l@{}}
    X_1^{(M)}(\sigma(I))
      =X_1^{(M)}(0)+N_{12}^v\left(T_{12}^{(M)}(\sigma(I))\right)
      -N_{31}\left(T_{31}^{(M)}(\sigma(I))\right), \\
    X_2^{(M)}(\sigma(I))
      =X_2^{(M)}(0)+N_{23}\left(T_{23}^{(M)}(\sigma(I))\right)
      -N_{12}^v\left(T_{12}^{(M)}(\sigma(I))\right), \\
    X_3^{(M)}(\sigma(I))
      =X_3^{(M)}(0)+N_{31}\left(T_{31}^{(M)}(\sigma(I))\right)
      -N_{23}\left(T_{23}^{(M)}(\sigma(I))\right) .
  \end{array}\right.
  \label{eq2.17}
\end{equation}

Consequently it is seen that there exists a solution of the system of 
(\ref{eq2.8}) in [Case~A'] and $I-1$ in (\ref{eq2.13}) 
is replaced by $I$ in (\ref{eq2.15}).

\noindent
[Case B'] We consider the case of $X_1^{(M)}(\sigma(l))>0$, 
$X_2^{(M)}(\sigma(l))>0$, $X_3^{(M)}(\sigma(l'))>0$ 
and $X_3^{(M)}(\sigma(I-1))=0$ for 
$0 \leq l \leq I-1$ and $0 \leq l'<I-1$. 
The random variable of species 3 has come to the value zero at 
$\sigma(I-1)$, before the random variable of species 1 comes to the 
value zero.

In this case we have that 
$K^{12}(I_0-1)=\cdots=K^{12}(I-1)$ by several times of [Case~A'].

For $j=2,3$, as to $\sigma(I-1)+\frac{\tau_{K^{j j+1}(I-1)+1}^{j j+1}
-T_{j j+1}^{(M)}(\sigma(I-1))
}{\frac{\lambda}{M} X_j^{(M)}(\sigma(I-1)) X_{j+1}^{(M)}(\sigma(I-1))}$ 
the denominators are zero in [Case~B'] and the numerators are positive 
because of $P_j(I-1)$. Thus we replace these terms by infinity. 
We also replace $\tau_{K^{12}(I-1)+1}^{12}$ by infinity. 
We determine $\sigma(I)$ as
\begin{align*}
  \sigma(I) & =\min \{\infty, \infty, \infty\} \\
            & =\infty.
\end{align*}
Thus we do not need the solution of the system at $\sigma(I)=\infty$.

Similarly as in [Case~B'1] we prove $P_j(I_0-1, I_0)$ for $j=1,2,3$.

In $(\sigma(I-1), \infty)$, (\ref{eq2.14}) is replaced by (\ref{eq2.16}).

It follows that there exists a solution of the system of (\ref{eq2.8}) 
in [Case~B'].

\noindent
[Case C'] We consider the case of $X_1^{(M)}(\sigma(l'))>0$, 
$X_1^{(M)}(\sigma(l''))=0$, $X_2^{(M)}(\sigma(l))> 0$, 
and $X_3^{(M)}(\sigma(l))>0$ for $0 \leq l \leq I-1$, 
$0 \leq l'<k$ and $k \leq l'' \leq I-1$ ($I_0-1<k \leq I-1$). 
In this case the value of random variable of species 1 has come to zero 
at $\sigma(k)$ and kept it in $[\sigma(k), \sigma(I-1)]$.

In this case we implicitly assume $K^{12}(I_0-1)=\cdots=K^{12}(I-1)$ 
and $K^{31}(k)=\cdots=K^{31}(I-1)$. 
We have $X_1^{(M)}(t)=0$ for $t \in[\sigma(k), \sigma(I-1)]$.

We replace $\tau_{K^{12}(I-1)+1}^{12}$ by infinity. 
As to $\sigma(I-1)+\frac{\tau_{K^{jj+1}(I-1)+1}^{j j+1}
-T_{j j+1}^{(M)}(\sigma(I-1))
}{\frac{\lambda}{M} X_j^{(M)}(\sigma(I-1)) X_{j+1}^{(M)}(\sigma(I-1))}$ 
the denominators are zero. The numerator for $j=3$ is positive because of $P_3(I-1)$ and the numerator for $j=1$ is infinite. Then we replace these two terms by infinity.

We determine $\sigma(I)$ as
\begin{align*}
  \sigma(I-1) 
    & =\min \left\{\sigma(I-1)
      +\frac{\tau_{K^{23}(I-1)+1}^{23}-T_{23}^{(M)}(\sigma(I-1))
      }{\dfrac{\lambda}{M} X_2^{(M)}(\sigma(I-1)) X_3^{(M)}
      (\sigma(I-1))}, \infty, \infty\right\} \\
    & =\sigma(I-1)
      +\frac{\tau_{K^{23}(I-1)+1}^{23}-T_{23}^{(M)}(\sigma(I-1))
      }{\dfrac{\lambda}{M} X_2^{(M)}(\sigma(I-1)) X_3^{(M)}
      (\sigma(I-1))}.
\end{align*}
We count up one in $K^{j j+1}(I)$ for the selected number $j=2$ and 
we do not count up one for the not selected number $j=3$. 
We have $K^{12}(I)=K^{12}(I-1)=K^{12}(I_0-1)$, 
$K^{23}(I)=K^{23}(I-1)+1$ and $K^{31}(I)=K^{31}(I-1)$. 
Thus the implicit assumption hold to $I$-th step.

In a similarly way as [Step~5] $P_1(I-1, I)$ and $P_1(I)$ hold. 
For the selected number $j=2$, similarly as in [Step~1] of 
Theorem~\ref{theo2.1}, we have $P_2(I-1, I)$ and $P_2(I)$. Similarly as 
in [Step~3] of Theorem~\ref{theo2.1} 
$P_3(I-1, I)$ and $P_3(I)$ hold.

In $(\sigma(I-1), \sigma(I))$, (\ref{eq2.14}) is replaced by 
(\ref{eq2.16}) and, at $\sigma(I)$, (\ref{eq2.15}) is also replaced 
by (\ref{eq2.17}).

Consequently we have a solution of the system in [Case C'] 
and $I-1$ in (\ref{eq2.13}) is replaced by $I$ in (\ref{eq2.15}).

\noindent
[Case D'] We consider the case of $X_1^{(M)}(\sigma(l'))>0$, 
$X_1^{(M)}(\sigma(l''))=0$, $X_2^{(M)}(\sigma(l))> 0$, 
$X_3^{(M)}(\sigma(l'''))>0$ and $X_3^{(M)}(\sigma(I-1))=0$ for 
$0 \leq l \leq I-1$, $0 \leq l'<k$, $k \leq l'' \leq I-1$, 
$0 \leq l'''<I-1$ $(I_0-1 \leq k<I-1)$. 
This is the first case that the value of random variable of species 
3 reaches zero after several times of [Case~C'].

In the present case we implicitly have $K^{12}(I_0-1)=\cdots
=K^{12}(I-1)$ and $K^{23}(k)=\cdots=K^{23}(I-1)$ 
by several times of [Case~C'].

We replace $\tau_{K^{12}(I-1)+1}^{12}$ by infinity. 
As to $\sigma(I-1)+\frac{\tau_{K^{j j+1}(I)+1}^{j j+1}
-T_{j j+1}^{(M)}(\sigma(I-1))
}{\frac{\lambda}{M} X_j^{(M)}(\sigma(I-1)) X_{j+1}^{(M)}
(\sigma(I-1))}$ ($j=1,2,3$) the denominators are zero in the present 
case. Because of $P_j(I-1)$ the numerators for $j=2,3$ are positive 
and the numerator for $j=1$ is infinite. 
Thus we replace all terms by infinity.

We determine $\sigma(I)$ as
\begin{align*}
  \sigma(I) & =\min \{\infty, \infty, \infty\} \\
    &=\infty.
\end{align*}
Thus we do not need the solution of the system at $\sigma(I)=\infty$.

Similarly as in [Step~4] in Theorem~\ref{theo2.1}, 
$P_j(I-1, I)$ hold for $j=2,3$. 
The proposition $P_1(I-1, I)$ holds similarly as in [Step~6].

In $(\sigma(I-1), \infty)$ (\ref{eq2.14}) is replaced 
by (\ref{eq2.16}).

Therefore there exists a solution in [Case~D'].

By mathematical induction there exists a solution of the system of 
(\ref{eq2.8}) in $[0, \infty)$.

Now we shall prove that the solution constructed above is unique.

We see that the random variables $X_j^{(M)}(t)$ $(1 \leq j \leq 3)$ are 
non-negative in a similar way as in the previous theorem. 
It follows that $\frac{\lambda}{M} \int_0^t X_k^{(M)}(s) 
X_{k+1}^{(M)}(s) ds$ are monotonously non-decreasing 
($1 \leq k \leq 3$).

We prove uniqueness of the solution by mathematical induction after 
$I_0$-th step.

In $[0, \sigma(I_0-1)]$ there exists a unique solution and it coincides 
with the solution constructed actually in the proof of existence of 
a solution by Theorem~\ref{theo2.1}. 
Note that the propositions hold in $[0, \sigma(I_0-1)]$.

The change of the system occurs at $s(I_0)$ such that
\begin{align*}
  s(I_0)
    =\min _{1 \leq j \leq 3}
      \left\{\inf \left\{t>\sigma(I_0-1): 
      T_{j j+1}^{(M)}(t)=\tau_{K^{j j+1}(I_0-1)+1}^{j j+1}
      \right\}\right\},
\end{align*}
where we set
\begin{align*}
  T_{j j+1}^{(M)}(t)
    =\frac{\lambda}{M} \int_0^t X_j^{(M)}(s) X_{j+1}^{(M)}(s) d s.
\end{align*}
[Case a'1] We consider the case of $X_j^{(M)}(\sigma(l))>0$ 
for $0 \leq l \leq I_0-1$ and $1 \leq j \leq 3$. 

Note that $\tau_{K^{12}(I_0-1)+1}^{12}=\infty$. 
For $t>\sigma(I_0-1)$ we have
\begin{align*}
  T_{12}^{(M)}(t) 
    & = T_{12}^{(M)}(\sigma(I_0-1))
      +\frac{\lambda}{M} X_1^{(M)}(\sigma(I_0-1)) 
      X_2^{(M)}(\sigma(I_0-1))(t-\sigma(I_0-1)) \\
    & <\infty.
\end{align*}
Thus
\begin{align*}
  \left\{
    t>\sigma(I_0-1): 
      T_{12}^{(M)}(t)=\tau_{K^{12}(I_0-1)+1}^{12}=\infty\right\}=\emptyset.
\end{align*}
and
\begin{align*}
  \inf \left\{
    t>\sigma(I_0-1): 
      T_{12}^{(M)}(t)=\tau_{K^{12}(I_0-1)+1}^{12}=\infty\right\}
    =\infty.
\end{align*}

It follows that
\begin{align*}
  s(I)=\min _{2 \leq j \leq 3}
  \left\{
    \sigma(I-1)+\frac{\tau_{K^{j j+1}(I_0-1)+1}^{j j+1}
      -T_{j j+1}^{(M)}(\sigma(I_0-1))
      }{\dfrac{\lambda}{M} X_j^{(M)}(\sigma(I_0-1)) 
      X_{j+1}^{(M)}(\sigma(I_0-1))}, \infty\right\}.
\end{align*}
[Case b'1] We consider the case of $X_1^{(M)}(\sigma(l))>0$, 
$X_2^{(M)}(\sigma(l))>0$, $X_3^{(M)}(\sigma(l'))>0$ 
and $X_3^{(M)}(\sigma(l''))=0$ for 
$0 \leq l \leq I_0-1$, $0 \leq l'<k$, 
and $k \leq l'' \leq I_0-1$ $(0 \leq k \leq I_0-1)$.

As $T_{j j+1}^{(M)}(t)=T_{j j+1}^{(M)}(\sigma(I_0-1))$ 
for $t>\sigma(I_0-1)$ $(j=2,3)$ and $P_j(I_0-1)$ hold, we have
\begin{align*}
  \left\{
    t>\sigma(I_0-1): 
      T_{j j+1}^{(M)}(t)=\tau_{K^{j j+1}(I_0-1)+1}^{jj+1}
  \right\}=\emptyset .
\end{align*}

It follows that
\begin{align*}
  s(I_0)=\min \{\infty, \infty, \infty\}.
\end{align*}

In $[0, \sigma(I-1)]$ we assume that there exists a unique solution of 
the system and that it coincides with the solution constructed actually 
in the proof of existence the system $(I \geq I_0)$. 
Note that the propositions hold in $[0, \sigma(I-1)]$.

The change of the system occurs at the time $s(I)$ such that
\begin{align*}
  s(I)=\min _{1 \leq j \leq 3}
    \left\{
      \inf \left\{t>\sigma(I-1): 
      T_{j j+1}^{(M)}(t)=\tau_{K^{j j+1}(I-1)+1}^{j j+1}
      \right\}\right\}.
\end{align*}
[Case a'] We consider the case of $X_j^{(M)}(\sigma(l))>0$ 
for $0 \leq l \leq I-1$ and $1 \leq j \leq 3$.

We have $K^{12}(I_0-1)=\cdots=K^{12}(I-1)$ and 
$\tau_{K^{12}(I-1)+1}^{12}=\cdots=\tau_{K^{12}(I_0-1)+1}^{12}=\infty$ 
in [Case~A'].

Similarly as in [Case~a'1] we have
\begin{align*}
  s(I)=\min _{2 \leq j \leq 3}
    \left\{
      \sigma(I-1)+\frac{\tau_{K^{j j+1}(I-1)+1}^{j j+1}
        -T_{j j+1}^{(M)}(\sigma(I-1))
      }{\dfrac{\lambda}{M} X_j^{(M)}(\sigma(I-1)) 
      X_{j+1}^{(M)}(\sigma(I-1))}, \infty\right\}.
\end{align*}
[Case b'] We consider the case of $X_1^{(M)}(\sigma(l))>0$, 
$X_2^{(M)}(\sigma(l))>0$, $X_3^{(M)}(\sigma(l'))>0$ 
and $X_3^{(M)}(\sigma(I-1))=0$ for $0 \leq l \leq I-1$ 
and $0 \leq l'<I-1$.

Similarly as in [Case b'1], we have
\begin{align*}
  s(I)=\min \{\infty, \infty, \infty\}.
\end{align*}
[Case c'] We consider the case of $X_1^{(M)}(\sigma(l'))>0$, 
$X_1^{(M)}(\sigma(l''))=0$, $X_2^{(M)}(\sigma(l))> 0$, and 
$X_3^{(M)}(\sigma(l))>0$ for $0 \leq l \leq I-1$, $0 \leq l'<k$ and 
$k \leq l'' \leq I-1$ $(I_0-1<k \leq I-1)$.

We have $T_{31}^{(M)}(t)=T_{31}^{(M)}(\sigma(I-1))<\tau_{K^{31}(I-1)+1}^{31}$ 
for $t>\sigma(I-1)$. It follows from $P_3(I-1)$ that
\begin{align*}
  \left\{t>\sigma(I-1): T_{31}^{(M)}(t)=\tau_{K^{31}(I-1)+1}^{31}\right\}
    =\emptyset.
\end{align*}

We have
\begin{align*}
  s(I)=\min \left\{\sigma(I-1)
    +\frac{\tau_{K^{23}(I-1)+1}^{23}-T_{23}^{(M)}(\sigma(I-1))
    }{\dfrac{\lambda}{M} X_2^{(M)}(\sigma(I-1)) X_3^{(M)}(\sigma(I-1))}, 
    \infty, \infty\right\}.
\end{align*}
[Case d'] We consider the case of $X_1^{(M)}(\sigma(l'))>0$, 
$X_1^{(M)}(\sigma(l''))>0$, $X_2^{(M)}(\sigma(l))> 0$, 
$X_3^{(M)}(\sigma(l'''))>0$ and $X_3^{(M)}(\sigma(I-1))=0$ for 
$0 \leq l \leq I-1,0 \leq l'<k$, $k \leq l'' \leq I-1$, 
$0 \leq l'''<I-1$ $(I_0-1 \leq k \leq I-1)$.

In this case we have that $K^{12}(I_0-1)=\cdots=K^{12}(I-1)$ 
and $K^{23}(k)=\cdots=$ $K^{23}(I-1)$.

As $T_{j j+1}^{(M)}(t)=T_{j j+1}^{(M)}(\sigma(I-1))
<\tau_{K^{j j+1}(I-1)+1}^{j j+1}$ for 
$t>\sigma(I-1)$ $(j=2,3)$ and $P_j(I-1)$ hold, we have
\begin{align*}
  \left\{t>\sigma(I-1): T_{j j+1}^{(M)}(t)
    =\tau_{K^{j j+1}(I-1)+1}^{j j+1}\right\}=\emptyset.
\end{align*}

We have
\begin{align*}
  s(I)=\min \{\infty, \infty, \infty\}.
\end{align*}

The jump time $\sigma(I_0)$ in [Case~A'1] and [Case~B'1] of the proof of existence of 
a solution coincides with $s(I_0)$ of [Case~a'1] and [Case~b'1]. 
The jump time $\sigma(I)$ in [Case~A']$\sim$[Case~D'] also coincides with $s(I)$ 
of [Case~a']$\sim$[Case~d']. There are no methods to construct the solution of 
the system of (\ref{eq2.2}) except the construction, since $\frac{\lambda}{M} 
\int_0^t X_k^{(M)}(s) X_{k+1}^{(M)}(s) d s$ are monotonously non-decreasing 
$(1 \leq k \leq 3)$. 
The constructed solution is unique, because $\sigma(I_0)$ and $\sigma(I)$ is 
determined by the factors to $I_0-1$-th and $I-1$-th step 
and jump times of standard Poisson processes.

By mathematical induction we prove that there exists a unique solution in 
$[0, \infty)$. \hfill $\square$

\begin{coro}\label{coro2.2}
  There exists a unique solution of equation \normalfont{(\ref{eq2.2})}, 
  when $t \in[0, t_0)$ for $t_0 \in[0, \infty)$.
\end{coro}

\section{A stochastic structure of the model}
\label{sec3}

From now on, we assume that $X_i^{(M)}(0)$ $(i=1,2,3)$ is independent of 
$N_{j j+1}(*)$ $(j=1,2,3)$. 
We define the reference family $(\mathcal{F}_t^{j j+1})_{t \geq 0}$
$(j=1,2,3)$ by
\begin{align*}
  \mathcal{F}_t^{j j+1} 
    =& \  \sigma\left(X_i^{(M)}(0): 1 \leq i \leq 3\right) \\
     & \vee \sigma\left(N_{j j+1}(s): 0 \leq s \leq t\right) \\
     & \vee \sigma\left(N_{i i+1}(u): u \geq 0,\  1 \leq i \leq 3, \  
     i \neq j\right), 
\end{align*}
where we put for each $t \in[0, \infty)$ the random time 
$T_{j j+1}^{(M)}(t)$ $(j=1,2,3)$ 
just similarly as in the previous section by
\[
T_{j j+1}^{(M)}(t)
  =\frac{\lambda}{M} \int_0^t X_j^{(M)}(s) X_{j+1}^{(M)}(s) d s.
\]

From equation (\ref{eq2.2}), for $t \in[0, \infty)$, 
we have the relation of $T^{(M)}(t)=(T_{12}^{(M)}(t), \allowbreak
T_{23}^{(M)}(t), T_{31}^{(M)}(t))$ as follows:
\begin{equation}
  \left\{
    \begin{aligned}
      & T_{12}^{(M)}(t)
        =\frac{\lambda}{M} \int_0^t\left(X_1^{(M)}(0)+N_{12}\left(T_{12}^{(M)}(s)\right)-N_{31}\left(T_{31}^{(M)}(s)\right)\right) \\
      & \hskip20mm~~~~ \left(X_2^{(M)}(0)+N_{23}\left(T_{23}^{(M)}(s)\right)
        -N_{12}\left(T_{12}^{(M)}(s)\right)\right) d s, \\
      & T_{23}^{(M)}(t)=\frac{\lambda}{M} \int_0^t\left(X_2^{(M)}(0)
        +N_{23}\left(T_{23}^{(M)}(s)\right)-N_{12}
        \left(T_{12}^{(M)}(s)\right)\right) \\
      & \hskip20mm~~~~ \left(X_3^{(M)}(0)+N_{31}\left(T_{31}^{(M)}(s)\right)
        -N_{23}\left(T_{23}^{(M)}(s)\right)\right) d s, \\
      & T_{31}^{(M)}(t)=\frac{\lambda}{M} \int_0^t\left(X_3^{(M)}(0)
        +N_{31}\left(T_{31}^{(M)}(s)\right)-N_{23}\left(T_{23}^{(M)}(s)
        \right)\right) \\
      &\hskip20mm~~~~ \left(X_1^{(M)}(0)+N_{12}\left(T_{12}^{(M)}(s)\right)
        -N_{31}\left(T_{31}^{(M)}(s)\right)\right) d s, \\
      & T^{(M)}(0)=0 .
  \end{aligned}\right.
  \label{eq3.1}
\end{equation}

\begin{theo}
  \label{theo3.1}
  When we fix the sample path $\omega \in \Omega, T^{(M)}(t)(\omega)$ is uniquely determined.
\end{theo}

\begin{proof}\normalfont 
  For each $t \in[0, \infty)$, we define
\end{proof}
\begin{equation}
  \left\{\renewcommand{\arraystretch}{1.5}
    \begin{array}{@{}l@{}}
      X_1^{(M)}(t)=X_1^{(M)}(0)+N_{12}\left(T_{12}^{(M)}(t)\right)
        -N_{31}\left(T_{31}^{(M)}(t)\right), \\
      X_2^{(M)}(t)=X_2^{(M)}(0)+N_{23}\left(T_{12}^{(M)}(t)\right)
        -N_{12}\left(T_{12}^{(M)}(t)\right), \\
      X_3^{(M)}(t)=X_3^{(M)}(0)+N_{12}\left(T_{31}^{(M)}(t)\right)
      -N_{23}\left(T_{31}^{(M)}(t)\right) .
  \end{array}\right.
  \label{eq3.2}
\end{equation}

From (\ref{eq3.1}) and (\ref{eq3.2}), we have
\begin{equation}
  \left\{\renewcommand{\arraystretch}{2}
    \begin{array}{@{}l@{}}
      T_{12}^{(M)}(t)=\dfrac{\lambda}{M} \dint_0^t X_1^{(M)}(s) X_2^{(M)}(s) ds, \\
      T_{23}^{(M)}(t)=\dfrac{\lambda}{M} \dint_0^t X_2^{(M)}(s) X_3^{(M)}(s) ds, \\
      T_{31}^{(M)}(t)=\dfrac{\lambda}{M} \dint_0^t X_3^{(M)}(s) X_1^{(M)}(s) ds.
  \end{array}\right.
  \label{eq3.3}
\end{equation}
It follows that
\begin{align*}
  \left\{\begin{aligned}
  X_1^{(M)}(t)=X_1^{(M)}(0)+ & N_{12}\left(\frac{\lambda}{M} 
  \int_0^t X_1^{(M)}(s) X_2^{(M)}(s) d s\right) \\
  & -N_{31}\left(\frac{\lambda}{M} \int_0^t X_3^{(M)}(s) X_1^{(M)}(s) 
  d s\right), \\
  X_2^{(M)}(t)=X_2^{(M)}(0) & +N_{23}\left(\frac{\lambda}{M} 
  \int_0^t X_2^{(M)}(s) X_3^{(M)}(s) d s\right) \\
  & -N_{12}\left(\frac{\lambda}{M} \int_0^t X_1^{(M)}(s) X_2^{(M)}(s) 
  d s\right), \\
  X_3^{(M)}(t)=X_3^{(M)}(0) & +N_{31}\left(\frac{\lambda}{M} 
  \int_0^t X_3^{(M)}(s) X_1^{(M)}(s) d s\right) \\
  & -N_{23}\left(\frac{\lambda}{M} \int_0^t X_2^{(M)}(s) X_3^{(M)}(s)
   d s\right).
  \end{aligned}\right.
\end{align*}

Therefore there exists a solution of the above equation 
and the solution is represented by (\ref{eq3.2}).

By the way there exists a unique solution of the above equation 
by Theorem~\ref{theo2.1}. 
If there exist two solutions $T^{(M)}(t)=(T_{12}^{(M)}(t),
T_{23}^{(M)}(t), T_{31}^{(M)}(t))$ and $T^{(M)*}(t)$ = 
$(T_{12}^{(M)*}(t)$, $T_{23}^{(M)*}(t), T_{31}^{(M)*}(t))$ of 
equation~(\ref{eq3.1}), then by (\ref{eq3.3})
\begin{align*}
& 
  T_{12}^{(M)}(t)
    =T_{12}^{(M) *}(t)
    =\frac{\lambda}{M} \int_0^t X_1^{(M)}(s) X_2^{(M)}(s) ds, \\
& 
  T_{23}^{(M)}(t)
    =T_{23}^{(M) *}(t)
    =\frac{\lambda}{M} \int_0^t X_2^{(M)}(s) X_3^{(M)}(s) d s, \\
& 
  T_{31}^{(M)}(t)
    =T_{31}^{(M) *}(t)
    =\frac{\lambda}{M} \int_0^t X_3^{(M)}(s) X_1^{(M)}(s) d s.
\end{align*}
Therefore $T^{(M)}(t)=T^{(M) *}(t)$.

Theorem is proved.\hfill $\square$

\begin{coro}\label{coro3.1}
  When we fix the sample path $\omega \in \Omega$, 
  $T^{(M)}(t)$ is uniquely determined for 
  $t \in[0, t_0]$ $(t_0 \in[0, \infty))$.
\end{coro}

\begin{proof}\normalfont
  Applying Corollarly~\ref{coro2.1} to Theorem~\ref{theo3.1} 
  we have the present corollarly. \hfill $\square$
\end{proof}

For any $v$, $v \geq 0$, we define a random field 
$\Phi_\omega^v$: $\mathbb{R}_{+}^3 \rightarrow \mathbb{R}_{+}^3$ 
as
\begin{align*}
& \Phi_\omega^v\left(\left(x_1, x_2, x_3\right)\right) \\
&\quad =\left(\renewcommand{\arraystretch}{2}
  \begin{array}{@{}l@{}}
\dfrac{\lambda}{M}\left(X_1^{(M)}(0)+N_{12}^v\left(x_1\right)
  -N_{31}\left(x_3\right)\right)
  \left(X_2^{(M)}(0)+N_{23}\left(x_2\right)
  -N_{12}^v\left(x_1\right)\right) \\
\dfrac{\lambda}{M}\left(X_2^{(M)}(0)+N_{23}\left(x_2\right)
  -N_{12}^v\left(x_1\right)\right)
  \left(X_3^{(M)}(0)+N_{31}\left(x_3\right)
  -N_{23}\left(x_2\right)\right) \\
\dfrac{\lambda}{M}\left(X_3^{(M)}(0)+N_{31}\left(x_3\right)
  -N_{23}\left(x_2\right)\right)
  \left(X_1^{(M)}(0)+N_{12}^v\left(x_1\right)
  -N_{31}\left(x_3\right)\right)
\end{array}\right)'.
\end{align*}
Put $S(t)=(S_1(t), S_2(t), S_3(t))$ to be the solution of
\begin{equation}
\label{eq3.4}  
  \left\{
    \begin{aligned}
      & S(t)(\omega) =\int_0^t \Phi_\omega^v(S(s)(\omega)) d s, \\
      & S(0) =0.
  \end{aligned}\right.
\end{equation}

\begin{theo}\label{theo3.2}
  When we fix the sample path $\omega \in \Omega$, 
  $S(t)(\omega)$ is uniquely determined.
\end{theo}

\begin{proof}\normalfont
  Applying Theorem~\ref{theo2.2} to $S(t)$ the present theorem is 
  concluded, similarly as in Theorem~\ref{theo3.1}.
  \hfill $\square$
\end{proof}

\begin{coro}\label{coro3.2}
  When we fix the sample path $\omega \in \Omega$, 
  $S(t)$ is uniquely determined for 
  $t \in[0, t_0]$ $(t_0 \in[0, \infty))$.
\end{coro}

\begin{proof}\normalfont
  Corollarly~\ref{coro2.2} and Theorem~\ref{theo3.2} lead 
  the present corollarly. \hfill $\square$
\end{proof}

\begin{lemma}\label{lem3.1}
  $S(t)$ is $\mathcal{F}_v^{12}$-measurable.
\end{lemma}

\begin{proof}\normalfont
  Since $\Phi_\omega^v(x)$ is represented by the generators of 
  $\mathcal{F}_v^{12}$, $\Phi_\omega^v(x)$ is 
  $\mathcal{F}_v^{12}$-measurable. 
  
  There exists a non-random function $F^v$ such that
  \begin{align*}
    S(t)=F^v\left(t ; X^{(M)}(0), N_{12}^v(u), N_{23}(u), 
      N_{31}(u), u \geq 0\right),
  \end{align*}
  where $X^{(M)}(0)=(X_1^{(M)}(0), X_2^{(M)}(0), X_3^{(M)}(0))$. 
  As $S(t)$ is represented by the generators of 
  $\mathcal{F}_v^{12}$, $S(t)$ is $\mathcal{F}_v^{12}$-measurable.
  \hfill $\square$
\end{proof}

Now, we prove the following lemma.

\begin{lemma}\label{lem3.2}
  For each $j$, $t$ $(1 \leq j \leq 3, t \in[0, \infty))$, 
  $T_{j j+1}^{(M)}(t)$ is a stopping time with respect to the reference 
  family $(\mathcal{F}_t^{j j+1})_{t \geq 0}$.
\end{lemma}

\begin{proof}\normalfont
  We consider the case of $j=1$, for example. 
  To be proved is that, for any $v \in[0, \infty)$,
  \[
    \left(T_{12}^{(M)}(t) \leq v\right) 
      \equiv\left\{\omega ; T_{12}^{(M)}(t)(\omega) \leq v\right\} 
      \in \mathcal{F}_v^{12}.
  \]
    
  We claim that $(T_{12}^{(M)}(t) \leq v)=(S_1(t) \leq v)$.

  For any $\omega \in(T_{12}^{(M)}(t) \leq v)$, $T_{12}^{(M)}(s)$ is 
  a monotonously non-decreasing function for $s \geq 0$ 
  (Theorem~\ref{theo3.1}). 
  It follows that $0 \leq T_{12}^{(M)}(u) \leq T_{12}^{(M)}(t)$ for 
  $0 \leq u \leq t$ and that $N_{12}^v(T_{12}^{(M)}(u))
  =N_{12}(T_{12}^{(M)}(u))$ for $0 \leq u \leq t$. 
  Thus the solution of (\ref{eq3.1}) satisfies (\ref{eq3.4}). 
  By uniqueness of the solution of (\ref{eq3.4}) in $[0, t]$ 
  (Corollary~\ref{coro3.2}) we have $T_{12}^{(M)}(u)=S_1(u)$ for 
  $0 \leq u \leq t$. Thus $T_{12}^{(M)}(t)=S_1(t)$.
  
  Hence $\omega \in(S_1(t) \leq v)$. 
  It concludes that $(T_{12}^{(M)}(t) \leq v) \subset(S_1(t) \leq v)$. 
  \hfill $\sharp$

  For any $\omega \in(S_1(t) \leq v)$, $S_1(s)$ is a monotonously 
  non-decreasing function for $s \geq 0$ (Theorem~\ref{theo3.2}). 
  It follows that $0 \leq S_1(u) \leq S_1(t)$ for $0 \leq u \leq t$ 
  and that $N_{12}(S_1(u))=N_{12}^v(S_1(u))$ for $0 \leq u \leq t$. 
  Thus the solution of (\ref{eq3.4}) satisfies (\ref{eq3.1}). 
  By uniqueness of the solution of (\ref{eq3.1}) in $[0, t]$ 
  (Corollary~\ref{coro3.1}) we have $S_1(u) = T_{12}^{(M)}(u)$ for 
  $0 \leq u \leq t$. Thus $S_1(t)=T_{12}^{(M)}(t)$.
  
  Hence $\omega \in(T_{12}^{(M)} \leq v)$. 
  We conclude $(S_1(t) \leq v) \subset(T_{12}^{(M)}(t) \leq v)$. 
  \hfill $\sharp$ 
  
  Therefore the proof is completed.\hfill $\square$
\end{proof}

The martingale parts of $N_{j j+1}(t)$ with respect to the reference 
family $\sigma(N_{j j+1}(t)$: $0 \leq s \leq t$) for 
$1 \leq j \leq 3$ are represented as
\[
  \widetilde{N}_{j j+1}(t)=N_{j j+1}(t)-t.
\]

Since $X_1^{(M)}(0)$, $X_2^{(M)}(0)$, $X_3^{(M)}(0)$, 
$N_{12}(*)$, $N_{23}(*)$ and $N_{31}(*)$ are mutually independent, 
$\widetilde{N}_{j j+1}(t)$ is an $\mathcal{F}_t^{j j+1}$-martingale.

Put
\begin{align*}
  \mathcal{G}_t^{(M)}
    &=\sigma\left( X_j^{(M)}(0): j=1,2,3\right) \\
    &\quad \vee \sigma\left(N_{j j+1}
      \left(T_{j j+1}^{(M)}(s)\right):\quad 
      0 \leq s \leq t,~~~ j=1,2,3\right),
\end{align*}
and
\[
  \mathcal{H}_t^{(M)}
    =\sigma\left(X_j^{(M)}(s): 0 \leq s \leq t,\ j=1,2,3\right).
\]

We shall recall the general theory in Corollary to 
Theorem~\ref{theo3.2} of Chapter~I of Ikeda-Watnabe \cite{3}. 
We assume that $(\Omega,(\mathcal{F}_t^{j j+1})_{t \geq 0})$ is 
a standard measurable space for each $j$, $1 \leq j \leq 3$, 
and let $P$ be a probability on 
$(\Omega,(\mathcal{F}_t^{j j+1})_{t \geq 0})$. 
Let $\mathcal{G}$ be a sub $\sigma$-field of 
$(\mathcal{F}_t^{j j+1})_{t \geq 0}$ and 
$P_{\mathcal{G}}(\omega, \cdot)$ be a regular conditional 
probability given $\mathcal{G}$.
Let $\xi(\omega)$ be a mapping from $\Omega$ into a measurable space 
$(S, \mathcal{B})$ such that it is $\mathcal{G} 
/ \mathcal{B}$-measurable. We assume that $\mathcal{B}$ is countably 
determined and $\{x\} \in \mathcal{B}$ for every $x \in S$. Then
\begin{equation}
  P_{\mathcal{G}}\left(\omega,\left\{\omega' ; 
    \xi\left(\omega'\right)=\xi(\omega)\right\}\right)=1 
    \quad
    a.a.\omega.
  \label{eq3.5}
\end{equation}
\begin{lemma}
  \label{lem3.3}
  $\mathcal{G}_t^{(M)} \subset \mathcal{F}_{T_{j j+1}^{(M)}(t)}
  ^{j j+1}$ for $t$, $t \geq 0$, and $j$, 
  $1 \leq j \leq 3$, where
\[
  \mathcal{F}_{T_{12}^{(M)}(t)}^{12}
    =\left\{S \in \mathcal{F}_{\infty}^{12}:
      \left(T_{12}^{(M)}(t) \leq u\right) 
      \cap S \in \mathcal{F}_u^{12} 
      \text { for any } u \geq 0\right\}.
\]
\end{lemma}
\begin{proof}\normalfont
  We consider the case of $\mathcal{G}_t^{(M)} \subset 
  \mathcal{F}_{T_{12}^{(M)}(t)}^{12}$.
  
  We define
  \[
    N_{12}^{[t]}(s)(\omega) 
      \equiv 
      \begin{cases}
        N_{12}(s)(\omega), 
          & \text { for } s \leq T_{12}^{(M)}(t)(\omega), \\[1mm]
        0, & \text { for } 
        s>T_{12}^{(M)}(t)(\omega).
      \end{cases}
  \]
  Since
  \[
    N_{12}^{[t]}(u)
      =N_{12}(u) \chi_{(u \leq T_{12}^{(M)}(t))},
  \]
    we have $(N_{12}^{[t]}(u) \leq a) \cap (T_{12}^{(M)}(t) \leq v) 
    \in \mathcal{F}_v^{12}$ for any $a \geq 0$. 
    Hence $N_{12}^{[t]}(u)$ is $\mathcal{F}_{T_{12}^{(M)}(t)}^{12}
    $-measurable. 
    We also have $(N_{23}(u) \leq a) \cap(T_{12}^{(M)}(t) \leq v) 
    \in \mathcal{F}_v^{12}$ 
    for any $a \geq 0$. 
    Hence $N_{23}(u)$ is 
    $\mathcal{F}_{T_{12}^{(M)}(t)}^{12}$-measurable. 
    Also $N_{31}(u)$ is $\mathcal{F}_{T_{12}^{(M)}(t)}^{12}$-measurable.
      
    We shall prove that $N_{12}(T_{12}^{(M)}(s)), 
    N_{23}(T_{23}^{(M)}(s))$ and $N_{31}(T_{31}^{(M)}(s))$ 
    is $\mathcal{F}_{T_{12}^{(M)}(t)}^{(12}$-\linebreak
    measurable, 
    for $0 \leq s \leq t$.

  \noindent  
  [Step 1] Put $F=N_{23}(T_{23}^{(M)}(s))$.

We claim that
\[
  E\left[F \mid \mathcal{F}_{T_{12}^{(M)}(t)}^{12}\right](\omega)=F(\omega).
\]

As the mapping in (\ref{eq3.5}), 
we take an $\mathcal{F}_{T_{12}^{(M)}(t)}^{12}$-measurable mapping
\[
  \xi\left(\omega'\right)
    =\left(N_{23}(u)\left(\omega'\right): u \geq 0\right).
\]
It follows that
\begin{align*}
  \lefteqn{E\left[F \mid \mathcal{F}_{T_{12}^{(M)}(t)}^{12}\right](\omega)} 
  \quad \\
    & =\int_{\Omega} P_{\mathcal{F}^{12}_{T_{12}^{(M)}(t)}}
    \left(\omega, d \omega'\right) 
      F\left(\omega'\right) \\
    & =\int_{\left\{\omega' ; \xi\left(\omega'\right)=\xi(\omega)\right\}} 
      P_{\mathcal{F}^{12}_{T_{12}^{(M)}(t)}}
      \left(\omega, d \omega'\right) 
      F\left(\omega'\right) \\
    & =\int_{\left\{\omega' ; \xi\left(\omega'\right)
      =\xi(\omega)\right\}} 
      P_{\mathcal{F}^{12}_{T_{12}^{(M)}(t)}}
      \left(\omega, d \omega'\right) 
        N_{23}\left(T_{23}^{(M)}(s)\left(\omega'\right), \omega\right) \\
    & =\int_{\Omega} P_{\mathcal{F}^{12}_{T_{12}^{(M)}(t)}}
    \left(\omega, d \omega'\right) f\left(T_{23}^{(M)}(s)
    \left(\omega'\right)\right),
\end{align*}
where $f(u)=N_{23}(u, \omega)$.

Similarly as in (\ref{eq3.1}), for $u, 0 \leq u \leq t$, we have
\begin{equation}
  \left\{\begin{aligned}
    T_{12}^{(M)}(u) 
    = & \frac{\lambda}{M} \int_0^u\left(X_1^{(M)}(0)+N_{12}^{[t]}
      \left(T_{12}^{(M)}(s)\right)-N_{31}\left(T_{31}^{(M)}(s)\right)\right) 
      \\
      & \hskip10mm 
      \left(X_2^{(M)}(0)+N_{23}\left(T_{23}^{(M)}(s)\right)-N_{12}^{[t]}
      \left(T_{12}^{(M)}(s)\right)\right) d s, \\
    T_{23}^{(M)}(u)
    = & \frac{\lambda}{M} \int_0^u\left(X_2^{(M)}(0)
      +N_{23}\left(T_{23}^{(M)}(s)\right)-N_{12}^{[t]}
      \left(T_{12}^{(M)}(s)\right)\right) \\
      & \hskip10mm 
      \left(X_3^{(M)}(0)+N_{31}\left(T_{31}^{(M)}(s)\right)
        -N_{23}\left(T_{23}^{(M)}(s)\right)\right) d s, \\
    T_{31}^{(M)}(u)
    = & \frac{\lambda}{M} \int_0^u\left(X_3^{(M)}(0)
    +N_{31}\left(T_{31}^{(M)}(s)\right)-N_{23}\left(T_{23}^{(M)}(s)\right)\right) 
    \\
     & \hskip10mm  \left(X_1^{(M)}(0)+N_{12}^{[t]}
     \left(T_{12}^{(M)}(s)\right)-N_{31}\left(T_{31}^{(M)}(s)\right)\right) d s, \\
  T^{(M)}(0)= &0 . 
  \end{aligned}\right.
\label{eq3.6}
\end{equation}
Hence there exists a non-random function $H$ from $[0, \infty)$ to $\mathbb{N}$ 
such that
\[
  f\left(T_{23}^{(M)}(s)\left(\omega'\right)\right)
    =H\left(s ; X^{(M)}(0), N_{12}^{[t]}\left(u, \omega'\right), 
      N_{23}\left(u, \omega'\right), 
      N_{31}\left(u, \omega'\right), u \geq 0\right).
\]
Therefore $f\left(T_{23}^{(M)}(s)\left(\omega'\right)\right)$ 
is $\mathcal{F}_{T_{12}^{(M)}(t)}^{12}$-measurable.
\begin{align*}
& \int_{\Omega} P_{\mathcal{F}^{12}_{T_{12}^{(M)}(t)}}\left(\omega, d \omega'\right) f\left(T_{23}^{(M)}(s)\left(\omega'\right)\right) \\
&\quad =f\left(T_{23}^{(M)}(s)(\omega)\right)(\omega) \\
&\quad =N_{23}\left(T_{23}^{(M)}(s)(\omega), \omega\right) \\
&\quad =F(\omega) .
\end{align*}
Hence the claim holds. 
It follows that $N_{23}(T_{23}^{(M)}(s))$ is 
$\mathcal{F}_{T_{12}^{(M)}(t)}^{12}$-measurable, for $0 \leq s \leq t$.

Similary, we prove that $N_{31}(T_{31}^{(M)}(s))$ is 
$\mathcal{F}_{T_{12}^{(M)}(t)}^{12}$-measurable, for $0 \leq s \leq t$.

\noindent
[Step 2] Put $G=N_{12}(T_{12}^{(M)}(s))$.

We claim that
\[
  E\left[G \mid \mathcal{F}_{T_{12}^{(M)}(t)}^{12}\right](\omega)=G(\omega).
\]

As the mapping in (\ref{eq3.5}), 
we take $\mathcal{F}_{T_{12}^{(M)}(t)}^{12}$-measurable mappings
\[
  \xi_1\left(\omega'\right)
    =\left(N_{12}^{[t]}(u)\left(\omega'\right): u \geq 0\right),
\]
and
\[
  \xi_2\left(\omega'\right)=T_{12}^{(M)}(t)\left(\omega'\right),
\]
and it is to be noted $T_{12}^{(M)}(t)$, which is the solution of (\ref{eq3.6}), 
is $\mathcal{F}_{T_{12}^{(M)}(t)}^{12}$-measurable. 

\noindent
We have
\begin{align*}
 \lefteqn{E\left[G \mid \mathcal{F}_{T_{12}^{(M)}(t)}^{12}\right](\omega)} 
 \quad \\
  & =\int_{\Omega} P_{\mathcal{F}^{12}_{ T_{12}^{(M)}(t)}}
    \left(\omega, d \omega'  \right) G\left(\omega'\right) \\
  & =\int_{\left\{\omega'; 
    \xi_1\left(\omega'\right)=\xi_1(\omega)\right\} 
    \cap\left\{\omega' ; \xi_2\left(\omega'\right)
    =\xi_2(\omega)\right\}} P_{\mathcal{F}^{12}_{ T_{12}^{(M)}(t)}}
    \left(\omega, d \omega'\right) G\left(\omega'\right) \\
  & =\int_{\left\{\omega'; \xi_1\left(\omega'\right)=\xi_1(\omega)\right\} 
    \cap\left\{\omega' ; 
    \xi_2\left(\omega'\right)=\xi_2(\omega)\right\}}
    P_{\mathcal{F}^{12}_{ T_{12}^{(M)}(t)}}
    \left(\omega, d \omega'\right) N_{12}
    \left(T_{12}^{(M)}(s)\left(\omega'\right), \omega\right) \\
  & =\int_{\Omega} P_{\mathcal{F}^{12}_{ T_{12}^{(M)}(t)}}
  \left(\omega,  d \omega'\right) g\left(T_{23}^{(M)}(s)
  \left(\omega'\right)\right),
\end{align*}
where $g(u)=N_{12}(u, \omega)$.

It is seen that $g(T_{12}^{(M)}(s)(\omega'))$ is 
$\mathcal{F}_{T_{12}^{(M)}(t)}^{12}$-measurable. Hence
\begin{align*}
  \lefteqn{\int_{\Omega} P_{\mathcal{F}^{12}_{ T_{12}^{(M)}(t)}} 
  \left(\omega, d \omega'\right)g\left( T_{12}^{(M)}(s)(\omega')\right)}
  \quad\\
    & =g\left(T_{12}^{(M)}(s)(\omega)\right)(\omega) \\
    & =N_{12}\left(T_{12}^{(M)}(s)(\omega), \omega\right) \\
    & =G(\omega).
\end{align*}
Hence the claim holds. It follows that $N_{12}(T_{12}^{(M)}(s))$ 
is $\mathcal{F}_{T_{12}^{(M)}(t)}^{12}$-measurable, for $0 \leq s \leq t$.

Therefore we see that
\[
  \mathcal{G}_t^{(M)} \subset \mathcal{F}_{T_{12}^{(M)}(t)}^{12}.
\]
Similarly, we prove $\mathcal{G}_t^{(M)} \subset 
\mathcal{F}_{T_{23}^{(M)}(t)}^{23}$ and $\mathcal{G}_t^{(M)} 
\subset \mathcal{F}_{T_{31}^{(M)}(t)}^{31}$. 
\hfill $\square$
\end{proof}

We set
\begin{align*}
&
  \mathcal{M}_{12}^{(M)}(*) 
    \equiv \tilde{N}_{12}\left(T_{12}^{(M)}(*)\right), \\
&
  \mathcal{M}_{23}^{(M)}(*) 
    \equiv \tilde{N}_{23}\left(T_{23}^{(M)}(*)\right), \\
&
  \mathcal{M}_{31}^{(M)}(*) 
    \equiv \widetilde{N}_{31}\left(T_{31}^{(M)}(*)\right).
\end{align*}
We denote $X^{(M)}(*)=(X_1^{(M)}(*), X_2^{(M)}(*), X_3^{(M)}(*))$.

\begin{theo}
  \label{theo3.3}
  The stochastic process $X^{(M)}(*)$ is $(\mathcal{G}_t^{(M)})
  _{t \geq 0   }$-semi-martingale such that
\[
  \left\{\renewcommand{\arraystretch}{1.5} \begin{array}{@{}l@{}}
    X_1^{(M)}(t)=X_1^{(M)}(0)
      +\left(\mathcal{M}_{12}^{(M)}(t)-\mathcal{M}_{31}^{(M)}(t)\right)
      +\left(T_{12}^{(M)}(t)-T_{31}^{(M)}(t)\right), \\
    X_2^{(M)}(t)=X_2^{(M)}(0)
      +\left(\mathcal{M}_{23}^{(M)}(t)-\mathcal{M}_{12}^{(M)}(t)\right)
      +\left(T_{23}^{(M)}(t)-T_{12}^{(M)}(t)\right), \\
    X_3^{(M)}(t)=X_3^{(M)}(0)
      +\left(\mathcal{M}_{31}^{(M)}(t)-\mathcal{M}_{23}^{(M)}(t)\right)
      +\left(T_{31}^{(M)}(t)-T_{23}^{(M)}(t)\right),
\end{array}\right.
\]
gives the Doob-Meyer decomposition and
\begin{enumerate}[{\protect\normalfont(i)}]
  \item 
    $\mathcal{M}_{j j+1}^{(M)}(t)$ are square-integrable 
    $(\mathcal{G}_t^{(M)})_{t \geq 0}$-martingales for $1 \leq j \leq 3$,
  \item 
    $T_{j j+1}^{(M)}(t)$ is continuous increasing 
    $(\mathcal{G}_i^{(M)})_{t \geq 0}$-adapted processes for $1 \leq j \leq 3$,
  \item 
    $\langle \mathcal{M}_{j j+1}^{(M)}(*)\rangle_t
    =T_{j j+1}^{(M)}(t)$ for $1 \leq j \leq 3$,
  \item 
  $\langle \mathcal{M}_{j j+1}^{(M)}(*), \mathcal{M}_{k k+1}^{(M)}(*)\rangle_t=0$ 
  \ for $1 \leq j$, $k \leq 3$, $j \neq k$.
\end{enumerate}
\end{theo}
\begin{coro}\label{coro3.3}
  The process $X_j^{(M)}(*)$ are $(\mathcal{H}_t^{(M)})_{t \geq 0}$-semi-martingale. 
\end{coro}
\begin{proof}\normalfont
  \mbox{}\linebreak

  \noindent
  [Step 1] For the case of the process $\tilde{N}_{12}(T_{12}^{(M)}(*))$, 
  to be proved is that for any $t$, $u$, $0 \leq t<u$,
  \[
    E\left[\widetilde{N}_{12}\left(T_{12}^{(M)}(u)\right)
      -\widetilde{N}_{12}\left(T_{12}^{(M)}(t)\right) 
      \mid \mathcal{G}_t^{(M)}\right]=0.
  \]
    
  By virtue of the optional sampling theorem due to Doob, 
  $\widetilde{N}_{12}(T_{12}^{(M)}(t))$ is a martingale with respect to 
  $\mathcal{F}_{T_{12}^{(M)}(t)}^{12}$.
  
  By Lemma~\ref{lem3.3},
  \begin{align*}
  & E\left[\widetilde{N}_{12}\left(T_{12}^{(M)}(u)\right)
    -\widetilde{N}_{12}\left(T_{12}^{(M)}(t)\right) 
    \mid \mathcal{G}_t^{(M)}\right]\\
  &\quad
    = E\left[E\left[\widetilde{N}_{12}\left(T_{12}^{(M)}(u)\right)
    -\widetilde{N}_{12}\left(T_{12}^{(M)}(t)\right) 
    \mid \mathcal{F}_{T_{12}^{(M)}(t)}^{12}\right] \mid \mathcal{G}_t^{(M)}\right]
    =0.
  \end{align*}
  Therefore $\tilde{N}_{12}(T_{12}^{(M)}(t))$ is a 
  $\mathcal{G}_t^{(M)}$-martingale.

  In a similar way, $\widetilde{N}_{23}(T_{23}^{(M)}(t))$ and 
  $\widetilde{N}_{31}(T_{31}^{(M)}(t))$ are $\mathcal{G}_t^{(M)}$-martingales.

  \noindent
  [Step 2] We claim that
  \[
  \left\langle\tilde{N}_{j j+1}\left(T_{j j+1}^{(M)}(*)\right)
  \right\rangle _t=T_{j j+1}^{(M)}(t),
  \]
and
\[
  \left\langle \widetilde{N}_{j j+1}\left(T_{j j+1}^{(M)}(*)\right), 
    \widetilde{N}_{k k+1}\left(T_{k k+1}^{(M)}(*)\right)\right\rangle _t=0,
\]
for $1 \leq j$, $k \leq 3$ and $j \neq k$.

In general, for the counting process whose martingale part is $M_t$ 
and whose bound-ed variational part is $A_t$
\[
  \left\langle M\right\rangle _t
    =\int_0^t\left(1-\Delta A_s\right) d A_s.
\]

The counting process $N_{j j+1}(T_{j j+1}^{(M)}(*))$ has the continuous bounded 
variational part. Therefore
\[
\left\langle \widetilde{N}_{j j+1}\left(T_{j j+1}^{(M)}(*)\right)
 \right\rangle _t=T_{j j+1}^{(M)}(t).
\]

There are no two more jumps of the mutually independent Poisson processes\linebreak 
$N_{j j+1}(t)$ and $N_{k k+1}(t)$ $(j \neq k)$ at the same time $t$. 
Hence we have no two more jumps of the processes $N_{j j+1}(T_{j j+1}^{(M)}(t))$ 
and $N_{j j+1}(T_{k k+1}^{(M)}(t))$ $(j \neq k)$ at the same time $t$. 
Thus $N_{j j+1}(T_{j j+1}^{(M)}(*))+N_{k k+1}(T_{k k+1}^{(M)}(*))$ 
is also a counting process whose bounded variational part is continuous. 
Hence
\[
  \left\langle \tilde{N}_{j j+1}\left(T_{j j+1}^{(M)}(*)\right)
    +\tilde{N}_{k k+1}\left(T_{k k+1}^{(M)}(*)\right)\right\rangle _t
      =T_{j j+1}^{(M)}(t)+T_{k k+1}^{(M)}(t).
\]
On the other hand,
\begin{align*}
&  
  \left\langle \widetilde{N}_{j j+1}\left(T_{j j+1}^{(M)}(*)\right)
    +\widetilde{N}_{k k+1}\left(T_{k k+1}^{(M)}(*)\right)\right\rangle _t 
    \\
&  
  \quad = \left\langle \widetilde{N}_{j j+1}\left(T_{j j+1}^{(M)}(*)\right)
  \right\rangle _t + \left\langle \widetilde{N}_{j j+1}
  \left(T_{j j+1}^{(M)}(*)\right)\right\rangle _t \\
&\qquad
  +2\left\langle \widetilde{N}_{j j+1}\left(T_{j j+1}^{(M)}(*)\right), 
  \widetilde{N}_{k k+1}\left(T_{k k+1}^{(M)}(*)\right)\right\rangle _t.
\end{align*}
Therefore
\[
  \left\langle \tilde{N}_{j j+1}\left(T_{j j+1}^{(M)}(*)\right), 
    \tilde{N}_{k k+1}\left(T_{k k+1}^{(M)}(*)\right)\right\rangle _t=0.
\]
\hfill $\square $
\end{proof}

\section{A weak law of large numbers of model which has a certain stochastic structure}
\label{sec4}

From now on, the norm $\|x\|$ of the vector $x=(x_1, x_2, \cdots , x_n)$ is to 
mean $\sum_{1 \leq i \leq n}|x_i|$. 
We take an integer $i$ as in the region $1 \leq i \leq n$, and if $i=n$, 
then $i+1=1$ and if $i=1$, then $i-1=n$. 
And we take another integer in a similar way.

By the same method as in the queuing model by Kogan, Liptser, Shiryayev and Smorodinski \cite{8,9}, we show the general theorem of the weak law of large numbers 
with respect to a model which has a certain stochastic structure.

Let $z(t)=(z_1(t), z_2(t), \cdots , z_n(t))$ $(t \in[0, \infty))$ 
be a solution of the differential equation
\begin{equation}
  \left\{
    \begin{aligned}
      \frac{d z_1(t)}{d t} 
        & =f^{12}\left(z_1(t), z_2(t)\right)-f^{n 1}\left(z_n(t), z_1(t)\right), 
        \\
      \frac{d z_2(t)}{d t} 
        & =f^{23}\left(z_2(t), z_3(t)\right)-f^{12}\left(z_1(t), z_2(t)\right), 
        \\
        & \cdots \cdots \cdots \\
      \frac{d z_i(t)}{d t} 
        & =f^{i i+1}\left(z_i(t), z_{i+1}(t)\right)-f^{i-1 i}
        \left(z_{i-1}(t), z_i(t)\right), \\
        & \cdots \cdots \cdots \\
      \frac{d z_n(t)}{d t} 
        & =f^{n 1}\left(z_n(t), z_1(t)\right)
        -f^{n-1 n}\left(z_{n-1}(t), z_n(t)\right),
  \end{aligned}\right.
  \label{eq4.1}
\end{equation}
with the property $\inf _{0 \leq s \leq t} z_i(s)>0$ for $1 \leq i \leq n$ and 
$\sum_{i=1}^n z_i(0)=1$. Here $f^{j j+1}=f^{j j+1}(x, y)$ is a non-negative 
function on $[0, \infty)$ with local Lipschitz condition for each variable 
$x$, $y$.

For each $M>0$, the stochastic process $Z^{(M)}(*)$ is an 
$(\mathcal{H}_t^{(M)})_{t \geq 0}$-semi-martingale such that
\begin{enumerate}[(i)]
  \item 
    $Z_i^{(M)}(t)=Z_i^{(M)}(0)+\mathfrak{m}_i^{(M)}(t)+\mathfrak{a}_i^{(M)}(t)$,
  \item 
    $\mathfrak{m}_i^{(M)}(t)
    =\mathcal{M}_{i i+1}^{(M)}(t)-\mathcal{M}_{i-1 i}^{(M)}(t)$,
  \item 
    $\mathfrak{a}_i^{(M)}(t)
      =\mathcal{A}_{i i+1}^{(M)}(t)-\mathcal{A}_{i-1 i}^{(M)}(t)$,
  \item 
    $\mathcal{M}_{j j+1}^{(M)}(t)$ is a square-integrable 
    $(\mathcal{H}_t^{(M)})_{t \geq 0}$-martingale, 
  \item 
    $\mathcal{A}_{j j+1}^{(M)}(t)$ is a continuous increasing 
    $(\mathcal{H}_t^{(M)})_{t \geq 0}$-adapted process,
  \item 
    $\mathcal{A}_{j j+1}^{(M)}(t)
    =\int_0^t M \chi_{\{\frac{z_j^{(M)}(s)}{M}>0\}} 
      \chi_{\{\frac{z_{j+1}^{(M)}(s)}{M}>0\}} f^{j j+1}(\frac{z_j^{(M)}(s)}{M}, 
      \frac{z_{j+1}^{(M)}(s)}{M}) d s$,
  \item 
    $\left\langle \mathcal{M}_{j j+1}^{(M)}(*)\right\rangle _t
    =\mathcal{A}_{j j+1}^{(M)}(t)$,
  \item 
    $\left\langle \mathcal{M}_{j j+1}^{(M)}(*), 
      \mathcal{M}_{k k+1}^{(M)}(*)V\right\rangle _t=0$ for $j \neq k$,
\end{enumerate}
where $Z_i^{(M)}(0)>0$ for $1 \leq i \leq n$ and $\sum_{i=1}^n Z_i^{(M)}(0)=1$.

Put
\[
  \left\{
    \begin{aligned}
      Z^{(M)}(t) & =\left(Z_1^{(M)}(t), Z_2^{(M)}(t), 
        \cdots, Z_n^{(M)}(t)\right), \\
      z(t) & =\left(z_1(t), z_2(t), \cdots, z_n(t)\right).
\end{aligned}\right.
\]
We set the reference family $\mathcal{H}_t^{(M)}=\sigma(Z_j^{(M)}(s): 
0 \leq s \leq t, 1 \leq j \leq n)$. 
We introduce the random time $T_i^{(M)}=\inf \{t: \frac{Z_i^{(M)}(s)}{M} 
\leq \frac{2}{M}\}$ and $T_0^{(M)}=\min _{1 \leq i \leq n} T_i^{(M)}$.

\begin{lemma}\label{lem4.1}
  $T_i^{(M)}$ is a stopping time with respect to the reference family 
  $(\mathcal{H}_s^{(M)})_{s \geq 0}$ for each $1 \leq i \leq n$. 
  $T_0^{(M)}$ is a stopping time with respect to the reference 
  family $(\mathcal{H}_s^{(M)})_{s \geq 0}$.
\end{lemma}

\begin{proof}\normalfont
  To be proved is
  \begin{equation}
    \left(T_i^{(M)} \leq s\right) 
      \equiv\left\{\omega ; T_i^{(M)}(\omega) \leq s\right\} 
        \in \mathcal{H}_s^{(M)}.
    \label{eq4.2}
  \end{equation}
  We decompose $(T_i^{(M)} \leq s)$ into
  \begin{align*}
    \left(T_i^{(M)} \leq s\right)
      = & \left\{\left(T_i^{(M)} \leq s\right) 
        \cap \left(\frac{Z_i^{(M)}(0)}{M} \leq \frac{2}{M}\right)\right\} \\
        & \cup\left\{\left(T_i^{(M)} \leq s\right) 
        \cap\left(\frac{Z_i^{(M)}(0)}{M}>\frac{2}{M}\right)\right\} .
  \end{align*}
  We have
  \[
    \left(T_i^{(M)} \leq s\right) 
      \cap \left(\frac{Z_i^{(M)}(0)}{M} 
      \leq \frac{2}{M}\right)=\left(\frac{Z_i^{(M)}(0)}{M} \leq \frac{2}{M}\right) 
      \in \mathcal{H}_0^{(M)} \subset \mathcal{H}_s^{(M)}.
  \]
  The second term is decomposed into
  \[
    \left(T_i^{(M)} \leq s\right) 
    \cap\left(\frac{Z_i^{(M)}(0)}{M}>\frac{2}{M}\right)
    =\cup_{r \leq s}\left(\frac{Z_i^{(M)}(r)}{M} 
    \leq \frac{2}{M}\right) \cap\left(\frac{Z_i^{(M)}(0)}{M}>\frac{2}{M}\right).
  \]
  Since $(\frac{Z_i^{(M)}(r)}{M} \leq \frac{2}{M}) \in \mathcal{H}_r^{(M)} 
  \subset \mathcal{H}_s^{(M)}$ and $(\frac{Z_i^{(M)}(0)}{M}>\frac{2}{M}) 
  \in \mathcal{H}_0^{(M)} \subset \mathcal{H}_s^{(M)}$, 
  the second term $(T_i^{(M)} \leq s) \cap(\frac{Z_i^{(M)}(0)}{M}>\frac{2}{M}) 
  \in \mathcal{H}_s^{(M)}$.

  Therefore (\ref{eq4.2}) holds.

  It follows from the general theory that 
  $T_0^{(M)}=\min _{1 \leq i \leq n} T_i^{(M)}$ is also a stopping time.

  \hfill \mbox{\qquad} $\square $
\end{proof}
\begin{theo}\label{theo4.1}
  We assume
  \[
    \lim _{M \rightarrow \infty}
    \left\|\frac{Z^{(M)}(0)}{M}-z(0)\right\| = 0 \quad \text { in probability. }
  \]
  Then for any $t \in(0, \infty)$
  \[
    \lim _{M \rightarrow \infty} \sup _{0 \leq s \leq t}
    \left\|\frac{Z^{(M)}(s)}{M}-z(s)\right\|=0 \quad \text { in probability. }
  \]
\end{theo}

\begin{proof}\normalfont
  We have
  \begin{align*}
   \lefteqn{\frac{Z_j^{(M)}(t)}{M}}\quad \\
      & =\frac{Z_j^{(M)}(0)}{M}
          +\frac{1}{M}\left(\mathcal{M}_{j j+1}^{(M)}(t)
          -\mathcal{M}_{j-1 j}^{(M)}(t)\right) \\
      & \quad+\int_0^t\left\{\chi_{\Big\{\frac{z_i^{(M)}(s)}{M}>0\Big\}} 
      \chi_{\Big\{\frac{z_{i+1}^{(M)}(s)}{M}>0\Big\}} 
      f^{j j+1}\left(\frac{Z_j^{(M)}(s)}{M}, \frac{Z_{j+1}^{(M)}(s)}{M}\right)
      \right. \\
      & \left.\quad-\chi_{\Big\{\frac{z_{i-1}^{(M)}(s)}{M}>0\Big\}} 
      \chi_{\Big\{\frac{z_i^{(M)}(s)}{M}>0\Big\}} f^{j-1 j}
      \left(\frac{Z_{j-1}^{(M)}(s)}{M}, \frac{Z_j^{(M)}(s)}{M}\right)\right\}
       ds.
  \end{align*}
  From the previous lemma, for any $t \in[0, \infty)$, 
  $\frac{Z_j^{(M)}(t \wedge T_0^{(M)})}{M}$ $(j=1,2, \cdots , n)$ 
  are decomposed into
  \begin{align*}
    \lefteqn{\frac{Z_j^{(M)}\left(t \wedge T_0^{(M)}\right)}{M}}\quad \\
      & =\frac{Z_j^{(M)}(0)}{M}+\frac{1}{M}\left(\mathcal{M}_{j j+1}^{(M)}
      \left(t \wedge T_0^{(M)}\right)-\mathcal{M}_{j-1 j}^{(M)}
      \left(t \wedge T_0^{(M)}\right)\right) \\
      & \quad+\int_0^{t \wedge T_0^{(M)}}\left\{f^{j j+1}
      \left(\frac{Z_j^{(M)}(s)}{M}, \frac{Z_{j+1}^{(M)}(s)}{M}\right)
      -f^{j-1 j}\left(\frac{Z_{j-1}^{(M)}(s)}{M}, 
      \frac{Z_j^{(M)}(s)}{M}\right)\right\} d s.
  \end{align*}
  By the assumption of the local Lipschitz condition, for $0<x_1<1$, $0<x_2<1$, 
  $0<y_1<1$ and $0<y_2<1$, 
  there exists a constant $C_{\mathit {Lipschitz }}$ such that
  \begin{align*}
    & \sup _{0<x_1<1,0<x_2<1} 
      \frac{\left|f^{j j+1}\left(x_1, y_1\right)
      -f^{j j+1}\left(x_2, y_2\right)\right|}{\left|x_1-x_2\right|} 
      \leq C_x^j, \\
    & \sup _{0<y_1<1,0<y_2<1} \frac{\left|f^{j j+1}\left(x_1, y_1\right)
      -f^{j j+1}\left(x_2, y_2\right)\right|}{\left|y_1-y_2\right|} 
      \leq C_y^j, \\
    & C_{Lipschitz}
      =2 \max \left\{C_x^1, C_x^2, \cdots, C_x^n, C_y^1, C_y^2, 
      \cdots, C_y^n\right\}.
  \end{align*}
  We estimate:
  \begin{align*}
    \lefteqn{\left|\frac{Z_j^{(M)}\left(t \wedge T_0^{(M)}\right)}{M}-z_j
      \left(t \wedge T_0^{(M)}\right)\right| }\quad \\
      & \leq\left|\frac{Z_j^{(M)}(0)}{M}-z_j(0)\right| \\
      & \quad+\frac{1}{M}\left\{\left|\mathcal{M}_{j j+1}^{(M)}
      \left(t \wedge T_0^{(M)}\right)-\mathcal{M}_{j-1 j}^{(M)}
      \left(t \wedge T_0^{(M)}\right)\right|\right\} \\
      & \quad+C_{\mathit{Lipschitz }} 
      \int_0^t\left|\frac{Z_j^{(M)}(s)}{M}-z_j(s)\right| d s.
  \end{align*}
  Put
  \[
    U_t^{(M)}=\left\|\frac{Z_j^{(M)}(t)}{M}-z_j(t)\right\|.
  \]
  We get the following estimation:
  \begin{align*}
    \lefteqn{\left\|U_{t \wedge T_0^{(M)}}^{(M)}\right\|}\quad \\
      & \leq\left\|U_0^{(M)}\right\|
        +\frac{1}{M}\left\|\mathcal{M}^{(M)}\left(t \wedge T_0^{(M)}\right)\right\|
        +C_{\mathit{Lipschitz}} \int_0^t\left\|U_s^{(M)}\right\| d s \\
        & \leq\left(\left\|U_0^{(M)}\right\|+\sup _{0 \leq s \leq t} 
        \frac{1}{M}\left\|\mathfrak{M}^{(M)}(s)\right\|\right) 
        e^{C_{\mathit{Lipschitz }} t}.
  \end{align*}
  Hence we get for any real number $\epsilon>0$,
  \begin{align*}
    & P\left(\arraycolsep=0pt
    \begin{array}{l}
      \displaystyle\sup _{0 \leq s \leq t \wedge T_0^{(M)}}
      \left\|U_s^{(M)}\right\|>\epsilon
    \end{array}
      \right) \\
    & \quad 
      \leq P\left(\sup _{0 \leq s \leq t}
      \left(\left\|U_0^{(M)}\right\|
      +\frac{1}{M}\left\|\mathfrak{M}^{(M)}(s)\right\|\right) 
      e^{C_{\mathit{Lipschits}}t}>\epsilon\right),
  \end{align*}
  and so,
  \begin{align*}
    \lefteqn{P\left(
    \arraycolsep=0pt
    \begin{array}{l}  
    \displaystyle \sup _{0 \leq s \leq t}
    \left\|U_s^{(M)}\right\|>\epsilon
    \end{array}
    \right)}\quad \\
     & \leq P\left(T_0^{(M)}<t\right)
     +P\left(\arraycolsep=0pt
    \begin{array}{l}
      \displaystyle \sup _{0 \leq s \leq t \wedge T_0^{(M)}}
     \left\|U_s^{(M)}\right\|>\epsilon, T_0^{(M)} \geq t
    \end{array}
     \right) \\
     & \leq P\left(T_0^{(M)}<t\right)
     +P\left(\arraycolsep=0pt
    \begin{array}{l}
      \displaystyle \sup _{0 \leq s \leq t \leq T_0^{(M)}}
     \left(\left\|U_0^{(M)}\right\|
     +\frac{1}{M}\left\|\mathfrak{M}^{(M)}(s)\right\|\right)
     >\epsilon e^{-C_{\mathit{Lipschitz }}t}
    \end{array}
     \right) .
  \end{align*}

  For any real number $\delta>0$ we claim
  \begin{equation}
    \lim _{M \rightarrow \infty} P\,\Bigg(\sup _{0 \leq s \leq t \leq T_0^{(M)}}\left(\left\|U_0^{(M)}\right\|+\frac{1}{M}\left\|\mathfrak{M}^{(M)}(s)\right\|\right)>\delta\Bigg)=0,
    \label{eq4.3}
  \end{equation}
  \begin{equation}
    \lim _{M \rightarrow \infty} P\left(T_0^{(M)}<t\right)=0.
    \label{eq4.4}
  \end{equation}
  We estimate (\ref{eq4.3}):
  \begin{align*}
    \lefteqn{P\,\Bigg(\sup _{0 \leq s \leq t \leq T_0^{(M)}}
    \left(\left\|U_0^{(M)}\right\|+\frac{1}{M}\left\|\mathfrak{M}^{(M)}
    (s)\right\|\right)>\delta\Bigg)} \quad \\
    & \leq P\left(\left\|U_0^{(M)}\right\|>\frac{\delta}{2}\right)
    +P\,\Bigg(\sup _{0 \leq s \leq t \leq T_0^{(M)}} 
    \frac{1}{M}\left\|\mathfrak{M}^{(M)}(s)\right\|>\frac{\delta}{2}
    \Bigg) \\
    & \leq P\left(\left\|U_0^{(M)}\right\|>\frac{\delta}{2}\right)
    +\sum_{1 \leq j \leq n} 
    P\,\Bigg(\sup _{0 \leq s \leq t} \frac{1}{M}
    \left|\mathcal{M}_{j j+1}^{(M)}(s)-\mathcal{M}_{j-1 j}^{(M)}(s)\right|>\frac{\delta}{2 n}, t \leq T_0^{(M)}\Bigg).
  \end{align*}
  By Chebyshev's inequality and the martingale inequality, we have
  \begin{align*}
    \lefteqn{P\,\Bigg(\sup _{0 \leq s \leq t \leq T_0^{(M)}} \frac{1}{M}\left|\mathcal{M}_{j j+1}^{(M)}(s)-\mathcal{M}_{j-1 j}^{(M)}(s)\right|>\frac{\delta}{2 n}\Bigg)} \quad \\
    & \leq \frac{2 n}{\epsilon} E\,\Bigg[
      \sup _{0 \leq s \leq t \leq T_0^{(M)}} \frac{1}{M}
      \left|\mathcal{M}_{j j+1}^{(M)}(s)-\mathcal{M}_{j-1 j}^{(M)}(s)
      \right|\Bigg] \\
    & \leq \frac{2 n}{\epsilon} \frac{C}{M^2} 
    E\left[\left\langle \mathcal{M}_{j j+1}^{(M)}(*)\right\rangle _t
    +\left\langle\mathcal{M}_{j-1 j}^{(M)}(*)\right\rangle _t\right] 
    \\
    & =\frac{2 n C}{\epsilon M} 
    E\left[\int_0^t f^{j j+1}\left(\frac{Z_j^{(M)}(v)}{M}, 
    \frac{Z_{j+1}^{(M)}(v)}{M}\right) d v
    +\int_0^t f^{j-1 j}\left(\frac{Z_{j-1}^{(M)}(v)}{M}, 
    \frac{Z_j^{(M)}(v)}{M}\right) d v\right] \\
    & \leq \frac{2 n C}{\epsilon M} 2 t 
    \max _{1 \leq j \leq n} \sup_{0<x<1,0<y<1} f^{j j+1}(x, y),
  \end{align*}
  where $C$ is a positive constant for the martingale inquality. 
  By letting $M$ tend to infinity, we see that (\ref{eq4.3}) holds.

  Now we estimate (\ref{eq4.4}). 
  We define the $\{1,2, \cdots , n\}$-valued function $i_s^{(M)}$ 
  such that $\frac{Z_{i_s^{(M)}}^{(M)}(s)}{M}=\min _{1 \leq l \leq n}
  \{\frac{Z_l^{(M)}(s)}{M}\}$ for $s \in[0, \infty)$. 
  Here we have the the relation
  \[
    \left\{T_0^{(M)}<t\right\} \subset\left\{T_0^{(M)} \leq t\right\} 
    \subset\left\{\inf _{0 \leq s \leq t \wedge T_0^{(M)}} 
    \frac{Z_{i_s^{(M)}}^{(M)}(s)}{M} \leq \frac{2}{M}\right\}.
  \]
  We estimate the third term: for any $s$, 
  $s \leq t \wedge T_0^{(M)}$,
  \begin{align*}
    \frac{Z_{i_s^{(M)}}^{(M)}(s)}{M} 
      & \geq z_{i_s^{(M)}}(s)
      -\left|z_{i_s^{(M)}}(s)-\frac{Z_{i_s^{(M)}}^{(M)}(s)}{M}\right|
       \\
      & \geq \inf _{0 \leq s \leq t} z_{i_s^{(M)}}(s)
      -\sup _{0 \leq s \leq t \wedge T_0^{(M)}}
      \left\|U_s^{(M)}\right\| \\
      & \geq r-\sup _{0 \leq s \leq t \wedge T_0^{(M)}}
      \left\|U_s^{(M)}\right\|,
  \end{align*}
  where $r=\inf _{0 \leq s \leq t} \min _{1 \leq i \leq n} z_i(s)$. 
  Hence we get
  \[
    \inf _{0 \leq s \leq t \wedge T_0^{(M)}} 
      \frac{Z_{i_s^{(M)}}^{(M)}(s)}{M} \geq r
      -\sup _{0 \leq s \leq t \wedge T_0^{(M)}}\left\|U_s^{(M)}\right\|.
  \]
  We have the relation
  \[
    \left\{T_0^{(M)}<t\right\} 
      \subset\Bigg\{r-\sup _{0 \leq s \leq t \wedge T_0^{(M)}}
      \left\|U_s^{(M)}\right\| \leq \frac{2}{M}\Bigg\}.
  \]
  Therefore we have the estimation:
  \[
    P\left(T_0^{(M)}<t\right) 
      \leq P\,\Bigg(\sup _{0 \leq s \leq t \wedge T_0^{(M)}}
      \left\|U_s^{(M)}\right\| \geq r-\frac{2}{M}\Bigg).
  \]
  It follows by (\ref{eq4.3}) that
  \[
    \lim _{M \rightarrow \infty} 
    P\,\Big(\sup _{0 \leq s \leq t \wedge T_0^{(M)}}
    \left\|U_s^{(M)}\right\|>\epsilon\Big)=0.
  \]
  This fact concludes (\ref{eq4.4}).

  Therefore
  \[
    \lim _{M \rightarrow \infty} P\Big(
      \sup _{0 \leq s \leq t}\left\|U_s^{(M)}\right\|>\epsilon
      \Big)=0,
  \]
  which complete the proof of Theorem~\ref{theo4.1}.
  \hfill $\square$
\end{proof}

\section{Application of the weak law of large numbers to\\ 
paper-scissors-stone model}
\label{sec5}

Let $u_1(t)$, $u_2(t)$, $u_3(t)$ be the solution of the deterministic 
system expressed by the defferential equation
\begin{equation}
  \left\{\renewcommand{\arraystretch}{2}
    \begin{array}{@{}l@{}}
      \dfrac{d u_1(t)}{d t}
        =\lambda\left(u_1(t) u_2(t)-u_3(t) u_1(t)\right), \\
      \dfrac{d u_2(t)}{d t}
        =\lambda\left(u_2(t) u_3(t)-u_1(t) u_2(t)\right), \\
      \dfrac{d u_3(t)}{d t}
        =\lambda\left(u_3(t) u_1(t)-u_2(t) u_3(t)\right).
  \end{array}\right.
  \label{eq5.1}
\end{equation}

Now, we shall discuss the convergence of $\frac{X_1^{(M)}(t)}{M}$, 
$\frac{X_2^{(M)}(t)}{M}$, $\frac{X_3^{(M)}(t)}{M}$ to $u_1(t)$, 
$u_2(t)$, $u_3(t)$, when $M$ tends to infinity.

By applying the previous general theorem to our model, we have the following theorem.

\begin{theo}
  \label{theo5.1}
  We assume the convergence in probability and conditions as
  \begin{align*}
    \left\{
    \begin{array}{@{}l@{}}
      \dlim _{M \rightarrow \infty}
        \left|\dfrac{X_1^{(M)}(0)}{M}-u_1(0)\right|=0, \\[5mm]
      \dlim _{M \rightarrow \infty}
        \left|\dfrac{X_2^{(M)}(0)}{M}-u_2(0)\right|=0, \\[5mm]
      \dlim _{M \rightarrow \infty}
        \left|\dfrac{X_3^{(M)}(0)}{M}-u_3(0)\right|=0, \\
      0<u_1(0)<1, \\
      0<u_2(0)<1, \\
      0<u_3(0)<1, \\
      u_1(0)+u_2(0)+u_3(0)=1 .
  \end{array}\right.\\
  \end{align*}
  Then for any $t \in(0, \infty)$
  \begin{align*}
    \left\{
    \begin{array}{@{}l@{}}\displaystyle
      \dlim _{M \rightarrow \infty} \sup _{0 \leq s \leq t}
        \left|\dfrac{X_1^{(M)}(s)}{M}-u_1(s)\right|=0, \\[5mm]
      \dlim _{M \rightarrow \infty} \sup _{0 \leq s \leq t}
        \left|\dfrac{X_2^{(M)}(s)}{M}-u_2(s)\right|=0, \\[5mm]
      \dlim _{M \rightarrow \infty} \sup _{0 \leq s \leq t}
        \left|\dfrac{X_3^{(M)}(s)}{M}-u_3(s)\right|=0.
      \end{array}\right.
  \end{align*}
\end{theo}

\begin{proof}\normalfont
  The system of the ordinary differential equation has two constants 
  of motion that $u_1(t)+u_2(t)+u_3(t)=1$ and 
  $u_1(t) u_2(t) u_3(t)=u_1(0) u_2(0) u_3(0)$. 
  The condition $0<u_i(0)<1$ for $i=1,2,3$ concludes that 
  $\inf _{0 \leq s \leq t} u_i(t)>0$ for any 
  $t \geq 0$ $(i=1,2,3)$.

  It is sufficient that we prove the Lipschitz condition of 
  the previous theorem in our model.

  We set $f^{j j+1}(x, y)=h(x, y) \equiv \lambda x y$ and 
  $f(y', x, y)=h(x, y)-h(y', x)$. 
  For $0<y_1<1$, $0<y_2<1$, $0<y_1'<1$ and $0<y_2'<1$, 
  we get the estimate:
  \[
    \sup _{0<x_1<1,0<x_2<1} 
    \frac{|f(y_1', x_1, y_1)-f(y_2', x_2, y_2)|}{|x_1-x_2|} 
      \leq 8 \lambda .
  \]
  Hence we take the Lipschitz constant $C_{\text {Lipschitz }}$ 
  as $8 \lambda$.

  \hfill $\square$
\end{proof}

\section{A central limit theorem of model which has 
a certain stochastic structure}
\label{sec6}

Similarly as in the queuing model by Kogan, Liptser, Shiryayev and 
Smorodinski \cite{8,9}, we show the following central limit theorem 
with respect to the model in section~\ref{sec3}. 
This theorem is preliminary for the central limit theorem of the 
paper-scissors-stone model.

\begin{theo}
  \label{theo6.1}
  Let $z(t)=(z_1(t), z_2(t), \cdots, z_n(t))$ $(t \in[0, \infty))$ 
  be a solution of the differential equation~({\normalfont\ref{eq4.1}}), 
  that has the vector form as
  \[
    \frac{d z(t)}{d t}=f_0(z(t)),
  \]
  with the property $\inf _{0 \leq s \leq t} z_i(s)>0$ 
  for $1 \leq i \leq n$, where
  \[
    f_0^i(x_1, x_2, \cdots, x_n)
      =f^{i i+1}(x_i, x_{i+1})-f^{i-1 i}(x_{i-1} ; x_i) 
     ~~ \text { for } 1 \leq i \leq n .
  \]
  Here, $f^{j j+1}=f^{j j+1}(x, y)$ is a non-negative continuously 
  differentiable function on $[0, \infty)$ with local Lipschitz 
  condition of the derivatives 
  $f_x^{j j+1}=\frac{\partial f^{j j+1}}{\partial x}(x, y)$ 
  and $f_y^{j j+1}=\frac{\partial f^{i i+1}}{\partial y}(x, y)$ 
  with respect to each variable $x$, $y$. 
  Moreover, we impose the normalization of 
  $z_1(0)+z_2(0)+\cdots+z_n(0)=1$.

  For each $M>0$, the stochastic process $Z^{(M)}(*)$ has the same 
  stochastic structure as in Theorem~{\normalfont\ref{theo4.1}}. 
  Moreover we assume
  \[
    \lim _{M \rightarrow \infty} 
      \frac{Z^{(M)}(0)}{M}=z(0)\quad
       \textit {in \ probability. }
  \]
  Put
  \[
    V^{(M)}(t)=\sqrt{M}\left(\frac{Z^{(M)}(t)}{M}-z(t)\right).
  \]
  Let the sequence of random variables $\{V^{(M)}(0)\}_{M \geq 1}$ 
  converges weakly to a distribution $F$.
  
  Then the sequence of the probability distributions of the process 
  $V^{(M)}(t)$ converges weakly to the distribution of an 
  $\mathbb{R}^n$-valued Gaussian diffusion process 
  $V=(V(t))_{t \geq 0}$ defined by the stochastic differential 
  equation
  \[
    d V(t)=b(t) V(t) dt + c^{\frac{1}{2}}(t) d W(t),
  \]
  with an $\mathbb{R}^n$-valued Wiener process 
  $W=(W_t)_{t \geq 0}$, with the initial condition 
  $V(0)$ having the distribution $F$ and with a matrix
  \begin{align*}
    & b(t)
    =\left(\frac{\partial}{\partial x_j} f_0^i
    \left(z_1(t), z_2(t), \cdots , z_n(t)\right)
    \right)_{1 \leq i, j \leq n} \\
    & \quad
     \left(\renewcommand{\arraystretch}{2}
      \begin{array}{@{}cccccc@{}}
        \dfrac{\partial f^{12}}{\partial x}
          -\dfrac{\partial f^{n 1}}{\partial y} 
          & \dfrac{\partial f^{12}}{\partial y} & 0 & \ldots & 0 
          & -\dfrac{\partial f^{n 1}}{\partial x} \\
        -\dfrac{\partial f^{12}}{\partial x} 
          & \dfrac{\partial f^{23}}{\partial x}
          -\dfrac{\partial f^{12}}{\partial y} 
          & \dfrac{\partial f^{23}}{\partial y} & 0 & \ldots & 0 \\
        \cdots &\cdots &\cdots & \cdots & \cdots & \cdots \\
        \dfrac{\partial f^{n 1}}{\partial x} & 0 & \ldots &  0 
        & -\dfrac{\partial f^{n-1 n}}{\partial x} 
        & \dfrac{\partial f^{n 1}}{\partial x}
          -\dfrac{\partial f^{n-1 n}}{\partial y}
        \end{array}\right), \\
    & c(t)= \\
    & \left(\renewcommand{\arraystretch}{1.3}
      \begin{array}{@{}cccccc@{}}
        f^{12}+f^{n 1} & -f^{12} & 0 & \cdots & 0 & -f^{n 1} \\
        -f^{12} & f^{23}+f^{12} & -f^{23} & 0 & \cdots & \cdots \\
        \cdots& \cdots &\cdots & \cdots & \cdots \\
        -f^{n 1} & 0 & \cdots & 0 & -f^{n-1 n} & f^{n 1}+f^{n-1 n}
      \end{array}\right),
  \end{align*}
  where
  \begin{align*}
    & c_{j k}(t)=c_{k j}(t)= \\
    & \left\{\renewcommand{\arraystretch}{1.3}
      \begin{array}{l}
      0, \quad \text { \textit{for} } \quad 
        2 \leq|j-k| \leq n-2,\ 1 \leq j,\ k \leq n, \\
      -f^{j j+1}=-f^{j j+1}(z_j(t), z_{j+1}(t)), 
        \quad \text { \textit{for} } 
        \quad |j-k|=1,\ n-1,\ 1 \leq j,\ k \leq n, \\
      f^{j j+1}+f^{j-1 j}=f^{j j+1}(z_j(t), z_{j+1}(t))
        +f^{j-1 j}(z_{j-1}(t), z_j(t)), \\
        \text { \textit{for} } k=j,\ 1 \leq j,\ k \leq n .
      \end{array}\right.
    \end{align*}
\end{theo}

\begin{proof}\normalfont
  We set 
  \begin{align*}
    & V_i^{(M)}(t) \\
    &\quad =V_i^{(M)}(0) \\
    &\qquad
     +\int_0^t \sqrt{M}\left(\chi_{\left\{\frac{z_i^{(M)}(s)}{M}>0\right\}} \chi_{\left\{\frac{z_{i+1}^{(M)}(s)}{M}>0\right\}} f^{i i+1}\left(\frac{Z_i^{(M)}(s)}{M}, \frac{Z_{i+1}^{(M)}(s)}{M}\right)\right. \\
     &\qquad
      \left.-f^{i i+1}\left(z_i(s), z_{i+1}(s)\right)\right) d s \\
     &\qquad
      -\int_0^t \sqrt{M}\left(\chi_{\left\{\frac{z_{i-1}^{(M)}(s)}{M}>0\right\}} \chi_{\left\{\frac{z_i^{(M)}(s)}{M}>0\right\}} f^{i-1 i}\left(\frac{Z_{i-1}^{(M)}(s)}{M}, \frac{Z_i^{(M)}(s)}{M}\right)\right. \\
     &\qquad \left.-f^{i-1 i}\left(z_{i-1}(s), z_i(s)\right)\right) d s \\
     &\qquad
      +\frac{1}{\sqrt{M}}\left(\mathcal{M}_{i i+1}^{(M)}(t)-\mathcal{M}_{i-1 i}^{(M)}(t)\right), \\
    &B_i^{(M)}(t)=\int_0^t \sqrt{M}\left(\chi_{\big\{\frac{z_i^{(M)}(s)}{M}>0\big\}} \chi_{\big\{\frac{z_{i+1}^{(M)}(s)}{M}>0\big\}} f^{i i+1}\left(\frac{Z_i^{(M)}(s)}{M}, \frac{Z_{i+1}^{(M)}(s)}{M}\right)\right. \\
      &\qquad\qquad -f^{i i+1}\left(z_i(s), z_{i+1}(s)\right) d s\\
      &\qquad\qquad -\int_0^t \sqrt{M}
        \left(\chi_{\big\{\frac{z_{i-1}^{(M)}(s)}{M}>0\big\}} 
        \chi_{\big\{\frac{z_i^{(M)}{(s)}}{M} >0\big\}} f^{i-1 i}
        \left(\frac{Z_{i-1}^{(M)}(s)}{M}, 
        \frac{Z_i^{(M)}(s)}{M}\right)\right. \\
      &\qquad\qquad -f^{i-1 i}\left(z_{i-1}(s), z_i(s)\right) d s,\\
    & \mathfrak{m}_i^{(M)}(t)
      = \frac{1}{\sqrt{M}}\left(\mathcal{M}_{i i+1}^{(M)}(t)
      -\mathcal{M}_{i-1 i}^{(M)}(t)\right), \\
    &\left\langle \mathfrak{m}_i^{(M), a}(*)\right\rangle _t
      =\chi_{\left\{\frac{1}{\sqrt{N}} \leq a\right\}} 
      \frac{1}{M}\left(\mathcal{A}^{(M) i i+1}(t)
      +\mathcal{A}^{(M) i-1 i}(t)\right), \\
    &\left\langle \mathfrak{m}_i^{(M), a}(*), \mathfrak{m}_j^{(M), 
    a}(*)\right\rangle _t
    =-\chi_{\left\{\frac{1}{\sqrt{N}} \leq a\right\}} 
    \frac{1}{M} \mathcal{A}^{(M) i+1}(t) \ \text { \textit{for} } \ 
    |j-i|=1, n-1, \\
  &\left\langle \mathfrak{m}_i^{(M), a}(*), 
  \mathfrak{m}_j^{(M), a}(*)\right\rangle _t
  =0 \ \text{ \textit{for} }\ 2 \leq|j-i| \leq n-2 .
  \end{align*}
  We present the following conditions which are known 
  in \cite{9,8,1,13}.
  
  \noindent
  For $t \in[0, \infty)$
  \begin{enumerate}[(A)]
    \item 
      $\lim _{M \rightarrow \infty} 
      \sup _{t \leq T}\left\|\Delta V^{(M)}(t)\right\|=0$ 
      in probability,
    \item 
      $\lim _{M \rightarrow \infty} 
      \sup _{t \leq T}\left\|B^{(M)}(t)-\int_0^t b(s) V(s)ds\right\|
      =0$ in probability,
    \item 
      $\lim _{M \rightarrow \infty} \sup _{t \leq T}
      \left| \left\langle\mathfrak{m}_j^{(M), a}(*), 
      \mathfrak{m}_k^{(M), a}(*)\right\rangle _t
      -\int_0^t c_{j k}(s) d s\right|=0$ in probability,
  \end{enumerate}
  for each $T>0$, $a \in(0,1]$ and $j$, $k=1,2, \cdots, n$ 
  and so-called condition of the linear growth
  \begin{enumerate}[(I)]
    \item 
      $\|b(t, V(t))\| \leq L(t)
      \left(1+\sup _{0 \leq s \leq t}\|V(s)\|\right)$,
    \item 
      $\sum_{j=1}^k\left|c_{j j}(t, V(t))\right| \leq L(t)
      \left(1+\sup _{0 \leq s \leq t}\|V(s)\|^2\right)$,
    \item 
      $\int_0^t L(s) d s<\infty$ for $t \in[0, \infty)$.
  \end{enumerate}
  It follows from [9] $V^{(M)}(t)$ converges weakly in distribution 
  to
  \[
    d V(t)=b(t, V(t)) d t+c^{\frac{1}{2}}(t, V(t)) d W(t)
  \]
  with a vector-valued Wiener process $W(*)$ consisting of independent 
  components, as $M$ tends to infinity.
  
  Now we shall prove these conditions.
  
  Condition of ``linear growth'' is clear because of the local 
  Lipschitz property of the functions. We prove three conditions (A), 
  (B) and (C) in the following steps.

  \noindent
  [Step 1] We claim that condition (A) holds.
  
  For any $t>0$,
  \[
    \left\|\Delta V^{(M)}(t)\right\|
      =\sqrt{M}\left\|\frac{\Delta Z^{(M)}(t)}{M}\right\| 
      \leq \frac{1}{\sqrt{M}}.
  \]
  Hence condition (A) holds, since for any $\epsilon>0$,
  \[
    P\left(\left\|\Delta V^{(M)}(t)\right\|>\epsilon\right) 
      \leq \frac{1}{\epsilon} 
      E\left[\left\|\Delta V^{(M)}(t)\right\|\right] 
      \leq \frac{1}{\epsilon \sqrt{M}}.
  \]
  [Step 2] We claim that
  \begin{equation}
    \lim _{M \rightarrow \infty} 
      P\left(\int_0^t \chi_{\big\{\frac{z_i^{(M)}(s)}{M}=0\big\}} d s>0\right)=0,
      \label{eq6.1}
  \end{equation}
  for any $t \in[0, \infty)$. 
  
  We have the estimate
  \[
    P\left(\int_0^t \chi_{\big\{\frac{z_i^{(M)}(s)}{M}=0\big\}} d s>0\right) 
      \leq P\left(\inf _{0 \leq s \leq t} \frac{Z_i^{(M)}(s)}{M}=0\right).
  \]
  Since
  \[
    \inf _{0 \leq s \leq t} \frac{Z_i^{(M)}(s)}{M} 
      \leq \inf _{0 \leq s \leq t} z_i(s)
      -\sup _{0 \leq s \leq t}\left|\frac{Z_i^{(M)}(s)}{M}-z_i(s)\right|,
  \]
  we have
  \[
    P\left(\int_0^t \chi_{\big\{\frac{z_i^{(M)}(s)}{M}=0\big\}} d s>0\right) 
      \leq P\left(\sup _{0 \leq s \leq t}\left|\frac{Z_i^{(M)}(s)}{M}-z_i(s)
      \right| 
        \geq \inf _{0 \leq s \leq t} z_i(s)\right) .
    \]
    From the weak law of large numbers and from the assumption of 
    $\inf _{0 \leq s \leq t} z_i(s)>0$, the claim (\ref{eq6.1}) holds.
  
  \noindent
  [Step 3] We claim that $B^{(M)}(t)$ is replaced by $\overline{B^{(M)}(t)}$ in 
  condition (B) and that $\langle\mathfrak{m}_j^{(M), a}(*)$, \linebreak
  $\mathfrak{m}_k^{(M), a}(*)\rangle _t$ is replaced by 
  $\overline{\langle \mathfrak{m}_j^{(M), a}(*), \mathfrak{m}_k^{(M), a}(*)
  \rangle _t}$ in condition (C), where
  \begin{align*}
    \overline{B_t^{i(M)}} 
      = & \int_0^t \sqrt{M}\left(f^{i i+1}\left(\frac{Z_i^{(M)}(s)}{M}, 
        \frac{Z_{i+1}^{(M)}(s)}{M}\right)
        -f^{i i+1}\left(z_i(s), z_{i+1}(s)\right) d s\right. \\
        & -\int_0^t \sqrt{M}\left(f^{i-1 i}\left(\frac{Z_{i-1}^{(M)}(s)}{M}, 
        \frac{Z_i^{(M)}(s)}{M}\right)
        -f^{i-1 i}\left(z_{i-1}(s), z_i(s)\right) d s\right., 
  \end{align*}
  \begin{align*}
    &  \overline{\left\langle\mathfrak{m}_j^{(M), a}(*), \mathfrak{m}_k^{(M), a}(*)
    \right\rangle _t}
    \\
    & \quad
    =\left\{\renewcommand{\arraystretch}{2}
    \begin{array}{@{}l@{}}
      \chi_{\big\{\frac{1}{\sqrt{M}} \leq a\big\}}
      \left(\dint_0^t f^{j j+1}\left(\dfrac{Z_j^{(M)}(s)}{M}, 
      \dfrac{Z_{j+1}^{(M)}(s)}{M}\right) d s\right. \\
      \left.\quad+\dint_0^t f^{j-1 j}\left(\dfrac{Z_{j-1}^{(M)}(s)}{M}, 
      \dfrac{Z_j^{(M)}(s)}{M}\right) d s\right),
      \ \text { for } \ j=k, \\
      -\chi_{\big\{\frac{1}{\sqrt{M}} \leq a\big\}} 
      \dint_0^t f^{j j+1}\left(\dfrac{Z_j^{(M)}(s)}{M}, 
      \dfrac{Z_{j+1}^{(M)}(s)}{M}\right) d s, 
      \ \text { for } \ |k-j|=1, n-1, \\
      0, \text { for } 2 \leq|k-j| \leq n-2 .
  \end{array}\right.
  \end{align*}
  We consider the case of $\overline{B^{(M)}(t)}$. Since 
  \begin{align*}
    & \sup _{t \leq T} \left|B_i^{(M)}(t)-\int_0^t b(s) V(s) d s\right| \\
    &\quad 
      \leq  \sup _{t \leq T}
      \left|\overline{B_i^{(M)}(t)}-\int_0^t b(s) V(s) d s\right| \\
    &\quad
      +\sup _{t \leq T} \| \int_0^t \sqrt{M} 
        \chi_{\big\{\frac{z_i^{(M)}(s)}{M}=0 
          \text { \textit{or} } \frac{z_{i+1}^{(M)}(s)}{M}=0\big\}} 
      f^{i i+1}\left(\frac{Z_i^{(M)}(s)}{M}, \frac{Z_{i+1}^{(M)}(s)}{M}\right)ds 
      \\
    & \quad
      -\int_0^t \sqrt{M} \chi_{\big\{\frac{z_{i-1}^{(M)}(s)}{M}=0 
      \text { \textit{or} } \frac{z_i^{(M)}(s)}{M}=0\big\}} 
      f^{i-1 i}\left(\frac{Z_{i-1}^{(M)}(s)}{M}, \frac{Z_i^{(M)}(s)}{M}\right)ds \| 
      \\
    &\quad
    \leq \sup _{t \leq T}
    \left\|\overline{B^{(M)}(t)}-\int_0^t b(s) V(s) d s\right\| \\
    &\quad
      +C_0 \sqrt{M} \sup _{t \leq T} \left\lvert\, 
      \int_0^t \chi_{\big\{\frac{z_i^{(M)}(s)}{M}=0\big\}} d s
      +\int_0^t \chi_{\big\{\frac{z_{i+1}^{(M)}(s)}{M}=0\big\}} d s\right. 
      \\& \left.\quad+\int_0^t \chi_{\big\{\frac{z_{i-1}^{(M)}(s)}{M}=0\big\}} d s+\int_0^t \chi_{\big\{\frac{z_i^{(M)}(s)}{M}=0\big\}} d s \right\rvert\, 
      \\
    &\quad \leq 
    \sup _{t \leq T}\left\|\overline{B^{(M)}(t)}-\int_0^t b(s) V(s) d s\right\| 
    \\
    &\quad 
    +  C_0 \sqrt{M}\left\{\int_0^T \chi_{\big\{\frac{z_i^{(M)}(s)}{M}=0\big\}}ds
    +\int_0^T \chi_{\big\{\frac{z_{i+1}^{(M)}(s)}{M}=0\big\}} d s\right. \\
    & \left.\quad+\int_0^T \chi_{\big\{\frac{z_{i-1}^{(M)}(s)}{M}=0\big\}} ds
    +\int_0^T \chi_{\big\{\frac{z_i^{(M)}(s)}{M}=0\big\}} d s\right\}, 
  \end{align*}
  where $C_0=\max _{1 \leq j \leq n} \sup _{0<x<1,0<y<1} f^{j j+1}(x, y)$, 
  we have
  \begin{align*}
    & P\left(\sup _{t \leq T}\left\|B^{(M)}(t)-\int_0^t b(s) V(s) d s\right\|
    >\epsilon\right) \\
    &\quad 
      \leq P\left(\sup _{t \leq T}\left\|\overline{B^{(M)}(t)}
      -\int_0^t b(s) V(s) d s\right\|>\frac{\epsilon}{2}\right) \\
    & \quad
      +\sum_{j=1}^n P\left(C \sqrt{M} 
      \int_0^T \chi_{\big\{\frac{z_j^{(M)}(s)}{M}=0\big\}} d s 
      > \frac{\epsilon}{8 n}\right) \\
    &\quad
     \leq P\left(\sup _{t \leq T}\left\|\overline{B^{(M)}(t)}
     -\int_0^t b(s) V(s) d s\right\|>\frac{\epsilon}{2}\right) \\
    & \quad
    +\sum_{j=1}^n P\left(\int_0^T \chi_{\big\{\frac{z_j^{(M)}(s)}{M}=0\big\}} 
    d s>0\right).
  \end{align*}
  
  When we take the limit of $M \rightarrow \infty$,
  \begin{align*}
    & \lim _{M \rightarrow \infty} 
      P\left(\sup _{t \leq T}\left\|B^{(M)}(t)
      -\int_0^t b(s) V(s) d s\right\|>\epsilon\right) \\
      &\quad
       \leq \lim _{M \rightarrow \infty}
       P\left(\sup _{t \leq T}\left\|\overline{B^{(M)}(t)}
       -\int_0^t b(s) V(s) d s\right\|>\frac{\epsilon}{2}\right).
  \end{align*}
  
  The proof with respect to condition (C) can be done in a similar way. 
  
  Therefore the claim holds.

  \noindent
  [Step 4] We claim that condition (B) holds.
  
  Considering [Step 3], we have the following estimate:
  {\mathindent=0mm
  \begin{align*}
    & \sup _{t \leq T}\left\|\overline{B^{(M)}(t)}-\int_0^t b(s) V(s) d s\right\| 
    \\
    &\quad
     \leq \int_0^T \| \sqrt{M}\left\{\left(f^{i i+1}\left(\frac{Z_i^{(M)}(s)}{M}, 
    \frac{Z_{i+1}^{(M)}(s)}{M}\right)-f^{i i+1}\left(z_i(s), z_{i+1}(s)\right)\right)
    \right. \\
    &\qquad
     \left.-\left(f^{i-1 i}\left(\frac{Z_{i-1}^{(M)}(s)}{M}, \frac{Z_i^{(M)}(s)}{M}
    \right)-f^{i-1 i}\left(z_i(s) z, z_{i+1}(s)\right)\right)\right\} 
    \\
    &\qquad
     -\sum_{j=1}^n\left\{\left(\frac{\partial f_x^{i i+1}}{\partial x_j}
    -\frac{\partial f_x^{i-1 i}}{\partial x_j}\right)\left(z_1(s), z_2(s), 
    \cdots , z_n(s)\right) V_j^{(M)}(s)\right\} \| d s \\
    &\quad
     \leq \int_0^T \| V_i^{(M)}(s) f_x^{i i+1}
    \left(z_i(s)+\theta^{i i+1}\left(\frac{Z_i^{(M)}(s)}{M} z_i(s)\right), 
    z_{i+1}(s)+\theta^{i i+1}\left(\frac{Z_{i+1}^{(M)}(s)}{M}-z_{i+1}(s)\right)
    \right) \\
    &\qquad
     +V_{i+1}^{(M)}(s) f_y^{i i+1}\left(z_i(s)+\theta^{i i+1}
    \left(\frac{Z_i^{(M)}(s)}{M} z_i(s)\right), z_{i+1}(s)+\theta^{i i+1}
    \left(\frac{Z_{i+1}^{(M)}(s)}{M}-z_{i+1}(s)\right)\right) \\
    &\qquad
     -V_{i-1}^{(M)}(s) f_x^{i-1 i}\left(z_{i-1}(s)+\theta^{i-1 i}
    \left(\frac{Z_{i-1}^{(M)}(s)}{M} z_{i-1}(s)\right), z_i(s)+\theta^{i-1 i}
    \left(\frac{Z_i^{(M)}(s)}{M}-z_i(s)\right)\right) \\
    &\qquad
     -V_i^{(M)}(s) f_y^{i-1 i}\left(z_{i-1}(s)+\theta^{i-1 i}
    \left(\frac{Z_{i-1}^{(M)}(s)}{M} z_{i-1}(s)\right), z_i(s)+\theta^{i-1 i}
    \left(\frac{Z_i^{(M)}(s)}{M}-z_i(s)\right)\right) \\
    &\qquad
     -V_i^{(M)}(s) f_x^{i i+1}\left(z_i(s), z_{i+1}(s)\right)
    -V_{i+1}^{(M)}(s) f_y^{i i+1}\left(z_i(s), z_{i+1}(s)\right) \\
    &\qquad
     +V_{i-1}^{(M)}(s) f_x^{i-1 i}\left(z_{i-1}(s), z_i(s)\right)
    +V_i^{(M)}(s) f_y^{i-1 i}\left(z_{i-1}(s), z_i(s)\right) \| d s \\
    &\quad
     \leq \sup _{t \leq T}\left\|V^{(M)}(t)\right\| \sup _{t \leq T}
    \left\|\frac{Z^{(M)}(t)}{M}-z(t)\right\| 4 C T,
  \end{align*}
  }%
  where $C$ is a positive constant of the maximum of the Lipschitz 
  constants such that for $0<x_1<1$, $0<x_2<1$, $0<y_1<1$, $0<y_2<1$,
  \begin{align*}
  & 
    \sup _{0<x_1<1,0<x_2<1} 
      \frac{\left|f_x^{j j+1}(x_1, y_1)
      -f_x^{j j+1}(x_2, y_2)\right|}{\left|x_1-x_2\right|} 
      \leq C_{x x}^j, \\
  & 
    \sup _{0<y_1<1,0<y_2<1} 
      \frac{\left|f_x^{j j+1}\left(x_1, y_1\right)-f_x^{j j+1}
      \left(x_2, y_2\right)\right|}{\left|y_1-y_2\right|} 
      \leq C_{x y}^j, \\
  & 
    \sup _{0<x_1<1,0<x_2<1} 
      \frac{\left|f_y^{j j+1}\left(x_1, y_1\right)
      -f_y^{j j+1}\left(x_2, y_2\right)\right|}{\left|x_1-x_2\right|} 
      \leq C_{y x}^j \\
  & 
    \sup _{0<y_1<1,0<y_2<1} 
    \frac{\left|f_y^{j j+1}\left(x_1, y_1\right)
    -f_y^{j j+1}\left(x_2, y_2\right)\right|}{\left|y_1-y_2\right|} 
    \leq C_{y y}^j,\\
  & 
    C=\max _{1 \leq j \leq n}
    \left\{C_{x x}^j, C_{x y}^j, C_{y x}^j, C_{y y}^j\right\},
  \end{align*}
  and where $\theta^{j j+1} \in[0,1]$ $(1 \leq j \leq n)$ 
  are parameters in the mean value theorem.

  Hence
  \begin{align*}
    & P\left(\sup _{t \leq T}
      \left\|\overline{B^{(M)}(t)}
      -\int_0^t b(s) V(s) d s\right\| \geq \epsilon\right) \\
    & \quad
      \leq P\left(\sup _{t \leq T}\left\|V^{(M)}(t)\right\| 
      \geq l\right)
      +P\left(\sup _{t \leq T}\left\|
        \frac{Z^{(M)}(t)}{M}-z(t)\right\| 
        \geq \frac{\epsilon}{4 l C T}\right).
\end{align*}
  If
  \begin{equation}
    \lim _{l \rightarrow \infty} 
    \overline{\lim}_{M \rightarrow \infty} 
    P\left(\sup _{t \leq T}\left\|V^{(M)}(t)\right\| \geq l\right)=0,
    \label{eq6.2}
  \end{equation}
  then, from the weak law of large numbers (Theorem~\ref{theo4.1}), 
  for any $\delta>0$ there exists an integer $l$ such that
  \begin{align*}
    & 
      P\left(\sup _{t \leq T}\left\|V^{(M)}(t)\right\| 
      \geq l\right)<\delta, \\
    & 
      P\left(\sup _{t \leq T}\left\|\frac{Z^{(M)}(t)}{M}-z(t)\right\| 
      \geq \frac{\epsilon}{4 l C T}\right)<\delta .
  \end{align*}
  Therefore we get
  \[
    \overline{\lim }_{M \rightarrow \infty} 
    P\left(\sup _{t \leq T}\left\|\overline{B^{(M)}(t)}
    -\int_0^t b(s) V(s) d s\right\| \geq \epsilon\right)=0 .
  \]
  Now, we shall prove (\ref{eq6.2}). We have
  \begin{align*}
    \lefteqn{\left|V_i^{(M)}(t)\right|}\quad \\
      & \leq\left|V_i^{(M)}(0)\right|
        +\int_0^t C_{x, y}\left|V_i^{(M)}(s)\right| d s \\
      & \quad
        +C_0 \sqrt{M} 
        \int_0^t\left(\chi_{\big\{\frac{z_i^{(M)}(s)}{M}=0\big\}}
        +\chi_{\big\{\frac{z_{i+1}^{(M)}(s)}{M}=0\big\}}\right) ds 
        \\
      & \quad
        +\frac{1}{\sqrt{M}} 
        \sup _{0 \leq s \leq t}\left|\mathcal{M}_{i i+1}^{(M)}(s)
        -\mathcal{M}_{i-1 i}^{(M)}(s)\right|,
  \end{align*}
  where $C_0=\max _{1 \leq j \leq n} \sup _{0<x<1,0<y<1} 
  f^{j j+1}(x, y)$ and where 
  $C_{x, y}=\sup _{0 \leq x \leq 1,0 \leq y \leq 1}$\linebreak
  $f_x^{i i+1}(x, y)$
   + $\sup _{0 \leq x \leq 1,0 \leq y \leq 1} f_y^{i-1 i}(x, y)$.
   
  By Gromwell's inequality,
  \begin{align*}
    \lefteqn{\left|V_i^{(M)}(t)\right|}\quad \\
    & 
      \leq\left\{\left|V_i^{(M)}(0)\right|\right. \\
    & \quad
        +C_0 \sqrt{M} \int_0^t
        \Big(\chi_{\big\{\frac{z_i^{(M)}(s)}{M}=0\big\}}
        +\chi_{\big\{\frac{z_{i+1}^{(M)}(s)}{M}=0\big\}}\Big) d s 
        \\
    & \left.\quad
      +\frac{1}{\sqrt{M}} 
      \sup _{0 \leq s \leq t}
      \left|\mathcal{M}_{i i+1}^{(M)}(s)
      -\mathcal{M}_{i-1 i}^{(M)}(s)\right|\right\} 
      \cdot \exp \left\{C_{x, y} t\right\}.
  \end{align*}
  (\ref{eq6.2}) is estimated as
  \begin{align*}
    \lefteqn{
      P\left(\left\|V^{(M)}(t)\right\| \geq l\right) 
    \leq  P\left(C_1\left\|V^{(M)}(0)\right\| 
    \geq \frac{l}{3}\right) } \quad  \\
    & \qquad
      +\sum_{i=1}^n P\left(\int_0^t 
      \chi_{\big\{\frac{z_i^{(M)}(s)}{M}=0\big\}} d s>0\right) 
      \\
    & \qquad
      +\sum_{i=1}^n P\left(C_1 \frac{1}{\sqrt{M}} 
      \sup _{0 \leq s \leq t}
      \left|\mathcal{M}_{i i+1}^{(M)}(s)
      -\mathcal{M}_{i-1 i}^{(M)}(s)\right| \geq \frac{l}{3 n}\right),
  \end{align*}
  where $C_1=\exp \{C_{x, y} t\}$. 
  From the assumption of the theorem, the first term is convergent to 
  zero in probability as $M$ tends to infinity. From (\ref{eq6.1}), 
  the second term is convergent to zero in probability as $M$ tends to 
  infinity. From Chebyshev's inequality and the martingale inequality, 
  the third term is convergent to zero in probability, as $l$ tends to 
  infinity, since
  \begin{align*}
    & P\left(C_1 \frac{1}{\sqrt{M}} \sup _{0 \leq s \leq t}
    \left|\mathcal{M}_{i i+1}^{(M)}(s)
    -\mathcal{M}_{i-1 i}^{(M)}(s)\right| \geq \frac{l}{3 n}\right) 
    \\
    &\quad
     \leq \frac{3 n C_1 C_2}{l} E
     \left[\frac{1}{M}\left\langle \mathcal{M}_{i i+1}^{(M)}(*)\right\rangle _t
     + \left\langle\mathcal{M}_{i-1 i}^{(M)}(*)\right\rangle _t\right] 
     \\
     &\quad
      \leq \frac{3 n C_1 C_2}{l} 2 t \max _{1 \leq j \leq n} 
      \sup _{0<x<1,0<y<1} f^{j j+1}(x, y).
  \end{align*}
  where $C_2$ is a constant of the martingale inequality.
  
  Therefore the claim holds.

  \noindent
  [Step 5] We claim that condition (C) holds.
  
  By [Step 3], we prove that
  \[
    \lim _{M \rightarrow \infty} 
    \sup _{t \leq T} 
    \overline{\mid \left\langle \mathfrak{m}_j^{(M), a}(*), 
      \mathfrak{m}_k^{(M), a}(*)\right\rangle _t}
    -\int_0^t c_{j k}(s) d s \mid=0 
    \quad \text { \textit{in probability}. }
  \]
  We take the integer $M$ as $M>\frac{1}{a^2}$.

  There are no interactions between $j$ and $k$ for $2 \leq|j-k| 
  \leq n-2$. Hence condition (C) holds for this case.

  We consider the case of diagonal element of the quadratic variational 
  part.
  \begin{align*}
    & \sup _{t \leq T}\left| \left\langle 
      \frac{\mathcal{M}_{i i+1}^{(M)}(*)
      -\mathcal{M}_{i-1 i}^{(M)}(*)}{\sqrt{M}}\right\rangle _t
      -\int_0^t c_{i i}(s) d s\right| \\
    &\quad
     =\sup _{t \leq T}\left|\frac{1}{M}\left\langle
       \mathcal{M}_{i i+1}^{(M)}(*)\right\rangle _t
       +\frac{1}{M}\left\langle \mathcal{M}_{i-1 i}^{(M)}(*)
       \right\rangle _t-\int_0^t c_{i i}(s) d s\right| 
       \\
    &\quad
     \leq \sup _{t \leq T} \left|
      \int_0^t\left\{ f^{i i+1}\left(\frac{Z_i^{(M)}(s)}{M}, 
      \frac{Z_{i+1}^{(M)}(s)}{M}\right)-f^{i i+1}
      \left(z_i(s), z_{i+1}(s)\right) d s\right.\right. \\
    & \left.\qquad
      +\int_0^t\left\{f^{i-1 i}\left(\frac{Z_{i-1}^{(M)}(s)}{M}, 
      \frac{Z_i^{(M)}(s)}{M}\right)
      -f^{i-1 i}\left(z_{i-1}(s), z_i(s)\right) d s\right\} 
      d s \right| \\
    &\quad
     \leq 2 C T \sup _{t \leq T}
     \left|\frac{Z_i^{(M)}(s)}{M}-z_i(s)\right|,
  \end{align*}
  where $C$ is a constant of the maximum value of the Lipschitz 
  constants such that for $0<x_1<1$, $0<x_2<10<y_1<1$, $0<y_2<1$,
  \begin{align*}
    & \sup _{0<x_1<1,0<x_2<1} 
      \frac{\left|f^{j j+1}\left(x_1, y_1\right)-f^{j j+1}
      \left(x_2, y_2\right)\right|}{\left|x_1-x_2\right|} 
      \leq C_x^j, \\
    & \sup _{0<y_1<1,0<y_2<1} 
      \frac{\left|f^{j j+1}\left(x_1, y_1\right)-f^{j j+1}
      \left(x_2, y_2\right)\right|}{\left|y_1-y_2\right|} \leq C_y^j, 
      \\
    & C=\max \left\{C_x^1, C_x^2, \cdots, C_x^n, C_y^1, C_y^2, 
      \cdots, C_y^n\right\}.
  \end{align*}
  This term is convergent to zero in probability, 
  from the weak law of large numbers of Theorem~\ref{theo4.1}.

  Moreover,
  \begin{align*}
    & \sup _{t \leq T}\left|
      \left\langle \frac{\mathcal{M}_{i i+1}^{(M)}(*)
      -\mathcal{M}_{i-1 i}^{(M)}(*)}{\sqrt{M}}, 
      \frac{\mathcal{M}_{i+1 i+2}^{(M)}(*)
      -\mathcal{M}_{i i+1}^{(M)}(*)}{\sqrt{M}}\right\rangle _t
      -\int_0^t c_{i i+1}(s) d s\right| \\
    & \quad
      =\sup _{t \leq T}\left|-\frac{1}{M}
      \left\langle \mathcal{M}_{i i+1}^{(M)}(*)\right\rangle _t
      +\int_0^t c_{i i+1}(s) d s\right| \\
    &  \quad
      \leq \sup _{t \leq T} 
      \left| \int_0^t\left\{ f^{i i+1}
      \left(\frac{Z_i^{(M)}(s)}{M}, \frac{Z_{i+1}^{(M)}(s)}{M}\right)
      -f^{i j+1}(z_i(s), z_{i+1}(s)) d s \right|\right. \\
    &  \quad
      \leq C T \sup _{t \leq T}
      \left|\frac{Z_i^{(M)}(s)}{M}-z_i(s)\right|.
  \end{align*}
  This term is also convergent to zero in probability, 
  from the weak law of large numbers of Theorem~\ref{theo4.1}.
  
  Therefore the claim holds.\hfill $\square$
\end{proof}

\begin{rem}\label{rem6.1}
  \normalfont
  It is easy to see that the matrix $c(t)$ has eigenvalue zero and the 
  eigenvector $(1,1, \cdots, 1)$. 
  Hence we consider the eigenvector $(*, *, \cdots, *, 0)$ which is 
  independent of $(1,1, \cdots, 1)$. 
  In the restricted $(n-1) \times(n-1)$ matrix of $c(t)$ all 
  determinants of the leading minor matrix are positive. 
  Thus the restricted $(n-1) \times(n-1)$ matrix is positive definite. 
  Consequently, the matrix $c(t)$ is positive semi-definite.
\end{rem}

\section{Application of the central limit theorem to \\
aper-scissors-stone model}
\label{sec7}

We set
\[
  Y^{(M)}(t)=\sqrt{M}\left(\frac{X^{(M)}(t)}{M}-u(t)\right),
\]
for $t \in[0, \infty)$. We shall show that a sequence of the process 
$(Y^{(M)}(t))_{t \geq 0}$ admits the central limit theorem in our model.

We apply Theorem~\ref{theo6.1} to our model. 
Then we get the following theorem.

\begin{theo}\label{theo7.1}
  We assume
  \[
    \lim _{M \rightarrow \infty} 
      \frac{X^{(M)}(0)}{M}=u(0)
      \quad \text { \textit{in probability}, }
  \]
    as well as the case of the weak law of large numbers.

    Let the sequence of random variables $\{Y^{(M)}(0)\}_{M \geq 1}$ 
    converge weakly to a distribution $G$.
  
    Then the sequence of the probability distributions of the processes 
    $Y^{(M)}(t)$ converges weakly to the distribution of an 
    $\mathbb{R}^3$-valued Gaussian diffusion process $Y= 
    (Y(t))_{t \geq 0}$ defined by the stochastic equation 
    in the vector form
    \[
      d Y(t)=b(t) Y(t) d t+c^{\frac{1}{2}}(t) d W(t),
    \]
    with an $\mathbb{R}^3$ valued Wiener process $W=(W_t)_{t \geq 0}$, 
    with the initial condition $Y(0)$ having the distribution $G$ and 
    with a matrix
    \begin{align*}
      & 
        b(t)=\left(
          \begin{array}{@{}ccc@{}}
            \lambda(u_2(t)-u_3(t)) & \lambda u_1(t) & -\lambda u_1(t) 
              \\
            -\lambda u_2(t) & \lambda(u_3(t)-u_1(t)) & \lambda u_2(t) 
              \\
            \lambda u_3(t) & -\lambda u_3(t) & \lambda(u_1(t)-u_2(t))
          \end{array}\right),
    \end{align*}
    \begin{align*}
      & 
        c(t)= \\
      & 
        \left(\arraycolsep=3pt
        \begin{array}{@{}ccc@{}}
          \lambda(u_1(t) u_2(t)+u_3(t) u_1(t)) 
          & -\lambda u_1(t) u_2(t) & -\lambda u_3(t) u_1(t) 
            \\
          -\lambda u_1(t) u_2(t) 
          & \lambda(u_2(t) u_3(t)+u_1(t) u_2(t)) 
          & -\lambda u_2(t) u_3(t) 
            \\
          -\lambda u_3(t) u_1(t) & -\lambda u_2(t) u_3(t) 
          & \lambda(u_3(t) u_1(t)+u_2(t) u_3(t))
        \end{array}\right) .
  \end{align*}
\end{theo}
\begin{proof}\normalfont
  The functions $f^{j j+1}=f^{j j+1}(x, y)=\lambda x y$ of 
  Theorem~\ref{theo6.1} satisfy the local Lipschitz condition of the 
  derivatives $f_x^{j j+1}=\frac{\partial f^{j j+1}}{\partial x}(x, y)
  =\lambda y$ and $f_y^{j j+1}= \frac{\partial f^{i j+1}}{\partial y}
  (x, y)=\lambda x$ for each variable $0 \leq x$, $y \leq 1$ with 
  Lipschitz constant $\lambda$.
  \hfill $\square$
\end{proof}

\begin{rem}\label{rem7.1}\normalfont 
  Consider the system of $n$ cyclic prey-predator relations of 
  neighboring two species. Similarly as in the paper-scissors-stone 
  model, the number increasing over time $t$ of $i$-th species is equal 
  to the number decreasing by time $t$ of $i+1$-th species. Both the 
  weak law of large numbers and the central limit theorem for the 
  paper-scissors-stone model can be extended to this system of $n$ 
  cyclic prey-predator relations.
\end{rem}

\end{document}